\documentclass[pagebackref,colorlinks,citecolor=blue,linkcolor=blue,urlcolor=blue,filecolor=blue]{article}
\pdfoutput=1
\usepackage[margin=1in]{geometry}
\usepackage{tikz-cd}
\usetikzlibrary{calc}
\usepackage{float}
\usepackage{amsmath}
\usepackage{amsfonts}
\usepackage{amsthm}
\usepackage{enumitem}
\usepackage{dsfont}
\usepackage{amssymb}
\usepackage{pifont}
\usepackage{mathtools}
\usepackage{comment}
\usepackage{graphicx, caption} 
\usepackage{float}
\usepackage{spectralsequences}
\usetikzlibrary{matrix}
\usepackage[T1]{fontenc}
\usepackage[utf8]{inputenc}
\usepackage{epigraph}

\usepackage{pinlabel} 

\usepackage{hyperref}

\newtheorem*{T1}{Theorem~\ref{space homological stability}}

\newtheorem*{T2}{Theorem~\ref{almost sharp homological stability for k=1}}

\newtheorem{thm}{Theorem}[section]
\newtheorem{lem}[thm]{Lemma}
\newtheorem{prop}[thm]{Proposition}
\newtheorem{conj}[thm]{Conjecture}

\newtheorem{exam}[thm]{Example}
\newtheorem{question}[thm]{Question}
\theoremstyle{definition}

\theoremstyle{definition}
\newtheorem{const}[thm]{Construction}

\newtheorem{defn}[thm]{Definition}

\newtheorem{remark}[thm]{Remark}

\newtheorem*{claim*}{Claim}
\newtheorem*{quest*}{Question}
\newtheorem*{remark*}{Remark}
\newtheorem*{fact*}{Fact}

\newcommand{\thmtext}{
If $w\ge h\ge k+2$ and $wh-n\ge \max\big\{(k+1)(k+2), hk+2\big\}$, then the inclusion of $SF_{n}(R_{w,h})$ into $F_{n}(\R^{2})$
induces an isomorphism
\[
H_{k}\big(SF_{n}(R_{w,h})\big)\cong H_{k}\big(F_{n}(\R^{2})\big).
\]
}

\newcommand{\thmtextone}{
If $w\ge h \ge 3$ and $wh-n\ge 6$, then the inclusion of $SF_{n}(R_{w,h})$ into $F_{n}(\R^{2})$ induces an isomorphism
\[
H_{1}\big(SF_{n}(R_{w,h})\big)\cong H_{1}\big(F_{n}(\R^{2})\big).
\]
}

\newcommand{\Z}{\ensuremath{\mathbb{Z}}}

\newcommand{\R}{\ensuremath{\mathbb{R}}}

\makeatletter
\newcommand{\blfootnote}[1]{%
  \begingroup
  \renewcommand{\thefootnote}{}%
  \footnotetext{#1}%
  \endgroup
}
\makeatother

\title{The 15 Puzzle and homological stability in the space direction}
\author{Jes\'{u}s Gonz\'{a}lez
\and
Matthew Kahle
\and 
Nicholas Wawrykow}

\date{}

\begin{document}

\maketitle

\begin{abstract}
The ordered configuration space of $n$ open unit squares in the $w$ by $h$ rectangle exhibits homological stability in the space direction.
That is, for fixed $n$ and fixed homological degree $k$, once the underlying rectangle is large enough, making it any larger does not change the $k$-th homology of the square configuration space.
In this paper, we sharpen the stable range.
Finding bounds for $w$ and $h$ in terms of $n$ and $k$, we prove that most rectangles can be almost entirely filled with squares and there still be an isomorphism between the $k$-th homology of the resulting square configuration space and the $k$-th homology of the ordered configuration space of $n$ points in the plane.
\end{abstract}

\section{Introduction}

\blfootnote{
\textbf{2020 Mathematics Subject Classification:} 55R80, 52C15, 82B26.
}

\blfootnote{
\textbf{Keywords and phrases:}
Configuration space of squares in a rectangle, homological stability in the space direction, Mayer--Vietoris spectral sequence, Browder bracket, complex of injective words.
}

In 1896 Sam Loyd offered \$1,000, a sum worth over \$38,000 today, for a solution to the 15 Puzzle. 
The goal of the puzzle is to slide 15 labeled unit squares around a $4$ by $4$ rectangle to move from  the configuration depicted on the left of Figure \ref{15puzzlenew} to the configuration depicted on the right.
Loyd's money was safe, as 17 years prior Johnson and Story proved that these configurations are in different path-components of $SF_{15}(R_{4,4})$, the ordered configuration space of 15 open unit squares in the $4$ by $4$ rectangle \cite{johnson1879notes}.
Here we define the \emph{ordered configuration space of $n$ open unit squares in the $w$ by $h$ rectangle} to be 
\[
SF_{n}(R_{w,h}):=\left\{(x_{1}, y_{1},\dots, x_{n}, y_{n})\in \R^{2n}\bigg|\substack{{\displaystyle-w+\frac{1}{2}\le x_i\le -\frac{1}{2}, -h+\frac{1}{2}\le y_i\le -\frac{1}{2}, \text{ and }}\\\displaystyle\max\{|x_{i}-x_{j}|, |y_{i}-y_{j}|\}\ge 1\text{ for all }i\neq j}\right\}.
\]

\begin{figure}[h]
\centering
\captionsetup{width=.8\linewidth}
\includegraphics[width=7.5cm]{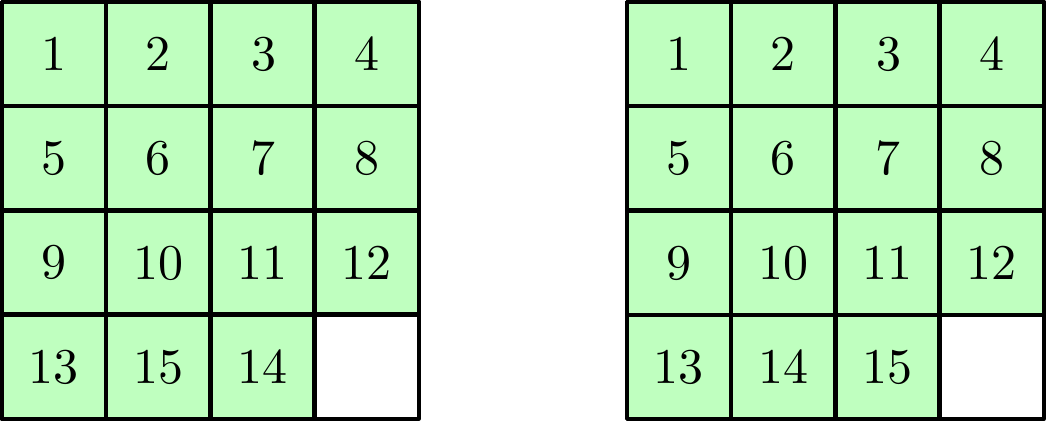}
\caption{The initial and target states of the 15 Puzzle.
}
\label{15puzzlenew}
\end{figure}

The impossibility of the 15 Puzzle raises several questions.
How big does the ambient rectangle $R_{w,h}$ need to be to guarantee that $SF_{15}(R_{w,h})$ is path-connected?
More generally:

\begin{question}\label{n puzzle quest}
What are conditions on $w$ and $h$ that make the ``$n$ Puzzle'' solvable?
That is, how big must $w$ and $h$ be to ensure that $H_{0}\big(SF_{n}(R_{w,h})\big)\cong \Z$?
Alternatively, how much \emph{empty space} $wh-n$ is needed to guarantee this isomorphism?
\end{question}

This question has a classical solution: As long as $w,h\ge 2$ and $wh-n\ge 2$, the square configuration space $SF_{n}(R_{w,h})$ is path-connected.
Moreover, Johnson and Story showed that if $wh-n=1$, then, $H_{0}\big(SF_{n}(R_{w,h})\big)\cong \Z^{2}$ \cite{johnson1879notes}.
Thus, the 14 Puzzle has a solution, but the 15 Puzzle does not.

In this paper, we study higher homological versions of Question \ref{n puzzle quest}.
Note that if $w$ and $h$ are large enough with respect to $n$ and the homological degree $k$, then there are isomorphisms
\[
H_{k}\big(SF_{n}(R_{w,h})\big)\cong H_{k}\big(F_{n}(\R^{2})\big),
\]
where $F_{n}(\R^{2})$ is the ordered configuration space of $n$ points in the plane.
Indeed, if $w$ and $h$ increase, the size of the unit squares becomes smaller relative to the size of the rectangle, and the squares behave more like points.
Moreover, if (and only if) $\min\{w,h\}\ge n$, the inclusion of $SF_{n}(R_{w,h})$ into $F_{n}(\R^{2})$ is a homotopy equivalence   \cite{plachta2021configuration}.
One can interpret this as saying that these square configuration spaces exhibit \emph{homological stability in the space direction}; see Subsection \ref{stability patterns}.
This raises the following question:

\begin{question}\label{homological stability question}
What is the \emph{stable range}?
That is, fixing $n$ and $k$, how big must $w$ and $h$ be to ensure that the inclusion of $SF_{n}(R_{w,h})$ into $F_{n}(\R^{2})$ induces an isomorphism
\[
H_{k}\big(SF_{n}(R_{w,h})\big)\cong H_{k}\big(F_{n}(\R^{2})\big)?
\]
Alternatively, what are conditions on the amount of empty space $wh-n$ that yield the above isomorphism?
\end{question}

Our main theorems answer this question, describing conditions on $w$, $h$, $k$, and $n$ that guarantee that $H_{k}\big(SF_{n}(R_{w,h})\big)$ is isomorphic to $H_{k}\big(F_{n}(\R^{2})\big)$.
In the following, we assume without loss of generality that $w\ge h$.

\begin{thm}\label{space homological stability}
\thmtext
\end{thm}

\begin{remark}
If $h=k+2$ or $k+3$, then $wh-n\ge (k+1)(k+2)$ is the more restrictive of the two bounds on the amount of empty space, whereas if $h>k+3$, then $wh-n \ge hk+2$ is the pertinent bound.  
\end{remark}

The calculations in Table 1 at the end of \cite{alpert2023homology} show that these bounds on $wh-n$ are sharp for $k\le 3$ if $w=h=k+2$.
Theorem \ref{space homological stability} proves that if $h$ and $k$ are fixed and $n$ varies, then in the minimal width $w(n)$ rectangle $R_{w(n),h}$ in which space directional homological stability occurs, asymptotically none of the rectangle is empty, that is
\[
\lim_{n\to \infty}\frac{w(n)\cdot h-n}{w(n)\cdot h}=0.
\]
This is markedly different from the unoccupied portion in the naive stable range given by the homotopy equivalence of $SF_{n}(R_{n,n})$ and $F_{n}(\R^{2})$---since $h\ge n$, the height $h$ cannot be fixed, and asymptotically all of the rectangle $R_{n,n}$ must be empty.
Moreover, Theorem \ref{space homological stability} significantly sharpens the previous best known bounds that arise from considering rectangles $R_{w(n),h}$ that model the infinite-width strip of height $h$, where $w(n)=n$ and asymptotically $\frac{h-1}{h}\ge \frac{k+1}{k+2}$ of the rectangle is empty \cite[Theorem 1.2]{alpert2021configuration}.

For $k=1$, we can remove the constraint $wh-n\ge hk+2$ in Theorem \ref{space homological stability}, showing that as long as there are $6$ unoccupied squares in $R_{w,h}$, then $H_{1}\big(SF_{n}(R_{w,h})\big)$ and $ H_{1}\big(F_{n}(\R^{2})\big)$ are isomorphic.

\begin{thm}\label{almost sharp homological stability for k=1}
\thmtextone
\end{thm}

\begin{remark}
Theorems \ref{space homological stability} and \ref{almost sharp homological stability for k=1} each consist of two bounds, one concerning the lengths of the sides of the rectangle and one concerning the amount of empty space in the rectangle.
The necessity of the side length condition is handled in Remark \ref{wh ge k+2}---the only other instance in which there is a non-trivial isomorphism $H_{k}\big(SF_{n}(R_{w,h})\big)\cong H_{k}\big(F_{n}(\R^{2})\big)$ occurs when $n=\min\{w,h\}=k+1$.
As such, the majority of this paper is dedicated to proving the validity of the empty spaces bounds $wh-n\ge \max\big\{(k+1)(k+2), hk+2\big\}$ and $wh-n\ge 6$.
\end{remark}

Finally, we note that Theorem \ref{almost sharp homological stability for k=1} positively resolves the following conjecture of Alpert, the second author, and MacPherson  for $k=1$, and Theorem \ref{space homological stability} positively resolves it for general $k$ if $h=k+2$ or $k+3$.

\begin{conj}
(Alpert--Kahle--MacPherson \cite[Conjecture 6 ii]{alpert2023asymptotic})
Assume $w\ge h\ge k+2$ and $wh-n\ge (k+1)(k+2)$, then the inclusion of $SF_{n}(R_{w,h})$ into $F_{n}(\R^{2})$ induces an isomorphism
\[
H_{k}\big(SF_{n}(R_{w,h})\big)\cong H_{k}\big(F_{n}(\R^{2})\big).
\]
\end{conj}

\subsection{Stability Patterns in Configuration Spaces}\label{stability patterns}
Our results fall in the storied tradition of homological stability, in particular, that of stability phenomena in configuration spaces.

A sequence $X_{1}\to X_{2}\to X_{3}\to \cdots$ of groups, spaces, etc. exhibits homological stability if the induced map $H_{k}(X_{n})\to H_{k}(X_{n+1})$ is an isomorphism once $n$ is large enough with respect to $k$.
In the 1970s, McDuff \cite{mcduff1975configuration} and Segal \cite{segal1979topology} proved that if $M$ is a connected non-compact finite type manifold of dimension at least $2$, then the $k$-th homology of $C_{n}(M)$, the unordered configuration space of $n$ points in $M$, stabilizes once $n\ge 2k$.
That is, the inclusion map
\[
\iota:C_{n}(M)\to C_{n+1}(M)
\]
that ``adds a point at infinity'' induces an isomorphism
\[
\iota_{*}:H_{k}\big(C_{n}(M)\big)\xrightarrow{\cong} H_{k}\big(C_{n+1}(M)\big)
\]
for $n\ge 2k$.
In short, they showed that after there are sufficiently many points in a configuration, the $k$-th homology does not change after adding more.

Instead of fixing the underlying space and letting the number of objects vary, we fix the number of objects and let the underlying space vary.
There are inclusion maps from $R_{w,h}$ to $R_{w+1,h}$ and from $R_{w,h}$ to $R_{w,h+1}$, and these maps induce maps on configuration space
\[
\iota_{w}:SF_{n}(R_{w,h})\to SF_{n}(R_{w+1,h})\indent\text{and}\indent\iota_{h}:SF_{n}(R_{w,h})\to SF_{n}(R_{w,h+1}).
\]
Theorems \ref{space homological stability} and \ref{almost sharp homological stability for k=1} prove that if $w,h\ge k+2$, any sequence of these maps exhibits homological stability; see Figure \ref{fig:stability for H1}.
That is, eventually they induce isomorphisms 
\[
(\iota_{w})_{*}:H_{k}\big(SF_{n}(R_{w,h})\big)\xrightarrow{\cong} H_{k}\big(SF_{n}(R_{w+1,h})\big)\indent\text{and}\indent(\iota_{h})_{*}:H_{k}\big(SF_{n}(R_{w,h})\big)\xrightarrow{\cong} H_{k}\big(SF_{n}(R_{w,h+1})\big),
\]
with the \emph{stable homology} being isomorphic to $H_{k}\big(F_{n}(\R^{2})\big)$.
Furthermore, our theorems also constitute a calculation of the \emph{stable range} for this space directional homological stability.

\begin{figure}[h]
\centering
 \scalebox{0.80}{
\begin{tikzpicture}
\large
\draw[opacity=0.5]  [->](-0.5,0.5) -- (7,0.5);
\draw [opacity=0.25] [->] (-0.5,1.5) -- (7,1.5);
\draw [opacity=0.25] [->] (-0.5,2.5) -- (7,2.5);
\draw [opacity=0.25] [->] (-0.5,3.5) -- (7,3.5);
\draw [opacity=0.25] [->] (-0.5,4.5) -- (7,4.5);
\draw [opacity=0.25] [->] (-0.5,5.5) -- (7,5.5);
\draw [opacity=0.25] [->] (-0.5,6.5) -- (7,6.5);
\draw [opacity=0.25] [->] (-0.5,7.5) -- (7,7.5);

\draw[opacity=0.5] [->] (-0.5,0.5) -- (-0.5,8);
\draw [opacity=0.25] [->] (0.5,0.5) -- (0.5,8);
\draw[opacity=0.25] [->] (1.5,0.5) -- (1.5,8);
\draw [opacity=0.25] [->] (2.5,0.5) -- (2.5,8);
\draw[opacity=0.25] [->] (3.5,0.5) -- (3.5,8);
\draw [opacity=0.25] [->] (4.5,0.5) -- (4.5,8);
\draw [opacity=0.25] [->] (5.5,0.5) -- (5.5,8);
\draw [opacity=0.25] [->] (6.5,0.5) -- (6.5,8);

 \node at (7.5,0.5) {\large\bf $w$};
 \node at (-0.5,8.5) {\large\bf  $h$};

\small
\node at (-1,1/2) {\large 1};
\node at (-1,3/2) {\large 2};
\node at (-1,5/2) {\large 3};
\node at (-1,7/2) {\large 4};
\node at (-1,9/2) {\large 5};
\node at (-1,11/2) {\large 6};
\node at (-1,13/2) {\large 7};
\node at (-1,15/2) {\large 8};

\node at (-1/2,0) {\large 1};
\node at (1/2,0) {\large 2};
\node at (3/2,0) {\large 3};
\node at (5/2,0) {\large 4};
\node at (7/2,0) {\large 5};
\node at (9/2,0) {\large 6};
\node at (11/2,0) {\large 7};
\node at (13/2,0) {\large 8};

\node at (-1/2,1/2) {\large$0$};
\node at (-1/2,3/2) {\large$0$};
\node at (-1/2, 5/2) {\large$0$};
\node at (-1/2, 7/2) {\large$0$};
\node at (-1/2, 9/2) {\large$0$};
\node at (-1/2, 11/2) {\large$0$};
\node at (-1/2, 13/2) {\large$0$};
\node at (-1/2, 15/2) {\large$0$};

\node  at (1/2,1/2) {\large$0$};
\node at (1/2,3/2) {\large$0$};
\node at (1/2, 5/2) {\large$0$};
\node at (1/2, 7/2) {\large$0$};
\node at (1/2, 9/2) {\large$?$};
\node at (1/2, 11/2) {\large$?$};
\node at (1/2, 13/2) {\large$?$};
\node at (1/2, 15/2) {\large$1218$};

\node  at (3/2,1/2) {\large$0$};
\node  at (3/2,3/2) {\large$0$};
\node  at (3/2, 5/2) {$15(8!)+2$};
\node  at (3/2, 7/2) {\large$?$};
\node  at (3/2, 9/2) {\large$ \bf 28$};
\node  at (3/2, 11/2) {\large$ \bf 28$};
\node  at (3/2, 13/2) {\large$ \bf 28$};
\node  at (3/2, 15/2) {\large$ \bf 28$};

\node   at (5/2, 1/2) {\large$0$};
\node   at (5/2, 3/2) {\large$0$};
\node   at (5/2, 5/2) {\large$?$};
\node   at (5/2, 7/2) {\large$ \bf 28$};
\node   at (5/2, 9/2) {\large$ \bf 28$};
\node   at (5/2, 11/2) {\large$ \bf 28$};
\node   at (5/2, 13/2) {\large$ \bf 28$};
\node   at (5/2, 15/2) {\large$ \bf 28$};

\node   at (7/2, 1/2) {\large$0$};
\node   at (7/2, 3/2) {\large$?$};
\node   at (7/2, 5/2) {\large$\bf 28$};
\node   at (7/2, 7/2) {\large$ \bf 28$};
\node   at (7/2, 9/2) {\large$ \bf 28$};
\node   at (7/2, 11/2) {\large$ \bf 28$};
\node   at (7/2, 13/2) {\large$ \bf 28$};
\node   at (7/2, 15/2) {\large$ \bf 28$};

\node   at (9/2, 1/2) {\large$0$};
\node   at (9/2, 3/2) {\large$?$};
\node   at (9/2, 5/2) {\large$\bf 28$};
\node   at (9/2, 7/2) {\large$ \bf 28$};
\node   at (9/2, 9/2) {\large$ \bf 28$};
\node   at (9/2, 11/2) {\large$ \bf 28$};
\node   at (9/2, 13/2) {\large$ \bf 28$};
\node   at (9/2, 15/2) {\large$ \bf 28$};

\node   at (11/2, 1/2) {\large$0$};
\node   at (11/2, 3/2) {\large$?$};
\node   at (11/2, 5/2) {\large$\bf 28$};
\node   at (11/2, 7/2) {\large$ \bf 28$};
\node   at (11/2, 9/2) {\large$ \bf 28$};
\node   at (11/2, 11/2) {\large$ \bf 28$};
\node   at (11/2, 13/2) {\large$ \bf 28$};
\node   at (11/2, 15/2) {\large$ \bf 28$};

\node   at (13/2, 1/2) {\large$0$};
\node   at (13/2, 3/2) {\large$1218$};
\node   at (13/2, 5/2) {\large$\bf 28$};
\node   at (13/2, 7/2) {\large$ \bf 28$};
\node   at (13/2, 9/2) {\large$ \bf 28$};
\node   at (13/2, 11/2) {\large$ \bf 28$};
\node   at (13/2, 13/2) {\large$ \bf 28$};
\node   at (13/2, 15/2) {\large$ \bf 28$};

\end{tikzpicture}
}
\caption{
The first Betti number of $SF_{8}(R_{w,h})$.
If $w,h\ge 3$ and $wh-8\ge 6$, then Theorem \ref{almost sharp homological stability for k=1} tells us that $H_{1}\big(SF_{8}(R_{w,h})\big)$ stabilizes to $H_{1}\big(F_{8}(\R^{2})\big)\cong \Z^{28}$.
Results of Alpert and Manin \cite{alpert2021configuration1} prove that if $w\ge 8$, then $H_{1}\big(SF_{8}(R_{w,2})\big)\cong H_{1}\big(SF_{8}(\mathbb{S}_{2})\big)\cong \Z^{1218}$, where $SF_{8}(\mathbb{S}_{2})$ is the ordered configuration space of $8$ open unit squares in the infinite strip of height $2$.
}

\label{fig:stability for H1}
\end{figure}

Note that while we are considering ordered configuration spaces, our theorems are homological stability results in the vein of McDuff and Segal and not representation stability results \`{a} la Church, Ellenberg, and Farb \cite{church2015fi}, as we are holding $n$ fixed and the natural symmetric group action does not come into play.

\subsection{Disk Configuration Spaces}

The study of disk configuration spaces, in the form of the 15 Puzzle---disks in the $L_{\infty}$ metric are squares---predates Hurwitz's and Artin's foundational works on the configuration space of points in the plane in connection with the braid groups \cite{hurwitz1891riemann, artin1925theorie, artin1947braids, artin1947theory}.
This field of research, which studies the space of embeddings of unit disks in a metric space, remained active in statistical mechanics and chemistry though it lay dormant in the mathematical community for over a century.
This changed in the late 2000s, when Diaconis posed several questions about these spaces in \cite{diaconis2009markov}.
Shortly thereafter,
the mathematical study of disk configuration spaces was resurrected by the work of Carlsson, Gorham, the second author, and Mason \cite{carlsson2012computational}, the second author \cite{kahle2012sparse}, and Baryshnikov, Bubenik, and the second author \cite{BBK}.
These papers, which relied on Morse theoretic ideas, focused on extremal spaces in which increasing the size of the underlying space resulted in changes to the Betti numbers of the configuration space.

Building on these works, Alpert, the second author, and MacPherson, turned to studying the ordered configuration space of $n$ open unit disks in the infinite strip of height $h$, which we denote by $BF_{n}(\mathbb{S}_{h})$ \cite{alpert2021configuration}.
They found cellular models for these spaces, which can be used to see that the inclusion of $BF_{n}(\mathbb{S}_{h})$ into $F_{n}(\R^{2})$ is $(h-1)$-connected, but not $h$-connected.
Varying $h$, $k$, and $n$, Alpert, the second author, and MacPherson used these models to determine when its homology goes through a \emph{phase transition} from the \emph{gas phase}, where there is an isomorphism $H_{k}\big(BF_{n}(\mathbb{S}_{h})\big)\cong H_{k}\big(F_{n}(\R^{2})\big)$, to the \emph{liquid phase}, where there is no such isomorphism and $H_{k}\big(BF_{n}(\mathbb{S}_{h})\big)$ is not trivial.
Alpert and Manin used this cellular model to find a basis for $H_{k}\big(BF_{n}(\mathbb{S}_{h})\big)$ for all $h$, $k$, and $n$ \cite{alpert2021configuration1}, and, in calculating this basis, they found a family of non-trivial classes called \emph{filters} that arise solely due to height constraints.
Later, the third author built on these results to prove that $H_{*}\big(BF_{\bullet}(\mathbb{S}_{h})\big)$ exhibits notions of (higher order) representation stability \cite{wawrykow2023representation}.
Despite these efforts, which elucidated the homology of $BF_{n}(\mathbb{S}_{h})$, much remained unknown about configurations of disks, round or square, in rectangles.

Unlike the infinite strip $\mathbb{S}_{h}$, the rectangle $R_{w,h}$ is bounded on all sides, and this significantly complicates the topology of $SF_{n}(R_{w,h})$ in comparison to $BF_{n}(\mathbb{S}_{h})$.
Although Plachta proved that if the width and height of the underlying rectangle are large enough with respect to $n$, then $SF_{n}(R_{w,h})$ is homotopy equivalent to $F_{n}(\R^{2})$ \cite{plachta2021configuration}, Plachta also showed that the homotopy type of $SF_{n}(R_{w,h})$ goes through a series of changes as $w$ and $h$ shrink, the first of which occurs when $\min\{w,h\}=n$.
Moreover, unlike the discrete model for $BF_{n}(\mathbb{S}_{h})$, there is no obvious homotopy equivalence between the $k$-skeleton of the discrete model for $SF_{n}(R_{w,h})$ found by Alpert, Bauer, the second author, MacPherson, and Spendlove \cite{alpert2023homology} and the $k$-skeleton of any known discrete model for $F_{n}(\R^{2})$ for any $k$.
sAdditionally, the calculations in Table 1 at the end of \cite{alpert2023homology} prove the existence of classes in $H_{k}\big(SF_{n}(R_{w,h})\big)$ outside of our stable range that seem to be worse-behaved than the filter classes in $H_{k}\big(BF_{n}(\mathbb{S}_{h})\big)$.

Despite the increased difficulty in understanding $SF_{n}(R_{w,h})$ compared to $BF_{n}(\mathbb{S}_{h})$, Alpert, Bauer, the second author, MacPherson, and Spendlove were able to prove that the homology groups of these configuration spaces also undergo phase transitions \cite{alpert2023homology}.
In the liquid phase, Alpert, the second author, and MacPherson computed asymptotic Betti numbers \cite{alpert2023asymptotic}.
Additionally, Alvarado-Gardu$\tilde{\text{n}}$o and the first and second authors proved that if $h=2$, then the unordered configuration space of $n$ open unit-squares in $R_{w,2}$ is homotopy equivalent to the unordered configuration space of $n$ open unit disks in $\mathbb{S}_{2}$ as long as $2w-n\ge 5$ \cite{MattOmarJesus}.
Despite these efforts, much remained unknown about the transition between the gas and liquid regimes of $H_{k}\big(SF_{n}(R_{w,h})\big)$.

Our results can be interpreted as a honing in on this phase transition.
That is, Theorems \ref{space homological stability} and \ref{almost sharp homological stability for k=1} give conditions on $w$, $h$, $k$, and $n$ that ensure that there are isomorphisms $H_{k}\big(SF_{n}(R_{w,h})\big)\cong H_{k}\big(F_{n}(\R^{2})\big)$.
Moreover, these conditions are sharp for $k=0$ \cite{johnson1879notes}, and for general $k$ spatial constraints suggest that these conditions are near optimal; see Remark \ref{pyramids are cool too}.
Namely, if $k>0$ and $k(k+1)\le wh-n<(k+1)^{2}$, we expect that the inclusion of $SF_{n}(R_{w,h})$ into $F_{n}(\R^{2})$ induces an epimorphism $H_{k}\big(SF_{n}(R_{w,h})\big)\to H_{k}\big(F_{n}(\R^{2})\big)$ but not an isomorphism.
For $wh-n<k(k+1)$, no such epimorphism is expected and the relationship between $H_{k}\big(SF_{n}(R_{w,h})\big)$ and $H_{k}\big(F_{n}(\R^{2})\big)$ becomes muddled; see, for example, Figure \ref{fig:stability for H1}.

Unlike previous works concerning disk configuration spaces, our paper eschews discrete Morse theory entirely. 
Though we rely on the discrete model of \cite{alpert2023homology}, we do so only to ensure the well-definedness of our methods.
Instead, we use a new paradigm for studying these spaces in the form of spectral sequences.

\subsection{Outline}
We begin by recalling several facts about square and point configuration spaces in Section \ref{Square Configuration Space}.
Since we are trying to determine when there is an isomorphism $H_{k}\big(SF_{n}(R_{w, h})\big)\cong H_{k}\big(F_{n}(\R^{2})\big)$, we recall a basis for the latter that will be used throughout the paper.
Next, we prove a series of results about the connectedness of certain square configuration spaces, culminating in Propositions \ref{free aj in k=1} and \ref{move the big stretchy squares around}.
Together, these results allow us to determine when certain classes in $H_{k}\big(SF_{n}(\R^{2})\big)$ are homologous, though they do not preclude the existence of exotic classes that arise from spatial constraints.
This ensures that the differentials of a pair of spectral sequences can be identified provided the entries of these spectral sequences are isomorphic.

In Section \ref{augmented MV}, we recall the definition of the augmented Mayer--Vietoris spectral sequence, the main tool we use to prove our homological stability results.
This spectral sequence allows one to recover the homology of a space via the homology of the subspaces covering it.
We specialize to $SF_{n}(R_{w,h})$ and $F_{n}(\R^{2})$, covering these spaces with sets of the form $SF_{n}(R_{w,h}; i_{0}, \dots, i_{p})$ and $F_{n}(\R^{2}; i_{0}, \dots, i_{p})$, in which the squares, respectively, points, $i_{0}, \dots, i_{p}$ are the right-most.
Then, we calculate the differentials of the resulting spectral sequences on a subset of the entries.
Using the results of Section \ref{Square Configuration Space}, we prove Proposition \ref{differentials are well defined as long as we have enough space}.
This allows us to show that as long as a certain subset of the $E^{1}_{p,q}$-entries of these sequences are naturally isomorphic, then their $E^{1}_{-1,k}$-entries are isomorphic, which would prove Theorems \ref{space homological stability} and \ref{almost sharp homological stability for k=1}.

Our proof of Theorem \ref{space homological stability} is contained in Section \ref{section main theorem}.
As our proof relies on knowing the terms of an augmented Mayer--Vietoris spectral sequence, we prove Theorem \ref{generalized main theorem}, a stronger result depending on a multiply inductive argument that handles the more general square configuration spaces $SF_{n}(R_{w,h}; i_{0}, \dots, i_{p})$ corresponding to the $E^{1}_{p,q}$-terms of our original spectral sequence.
Unlike point configuration spaces where one can ignore certain points in a configuration by sending them to ``infinity,'' the fact that we cannot freely homotope squares about $R_{w,h}$ due to their volume makes determining the entries of our spectral sequences an exacting process.
Moreover, unlike many classical homological stability results, which rely on the fact that certain entries vanish on the first few pages of a spectral sequence, our theorems require understanding non-trivial entries on the first $k+1$ pages of our spectral sequences.
We overcome these difficulties via Propositions \ref{super proposition} and \ref{higher terms in spectral sequence}.
To prove those propositions, we need to understand the subspaces $SF_{n}(R_{w,h}; j_{0},\dots, j_{s}; i_{0},\dots, i_{p})$ of $SF_{n}(R_{w,h}; i_{0},\dots, i_{p})$ in which the squares $j_{0},\dots, j_{p}$ are the next right-most.
It is reasonable to believe that to understand the $SF_{n}(R_{w,h}; j_{0},\dots, j_{s}; i_{0},\dots, i_{p})$, one would need to repeat this process leading to an arbitrarily long recursive argument, where at each level the corresponding spaces would look less and less like the configuration space of points in the plane.
We use a different augmented Mayer--Vietoris spectral sequence in Lemma \ref{cover j,i} to avoid this problem.

Finally, in Section \ref{tighter bound section}, we eliminate the $wh-n\ge hk+2$ term in the bounds of Theorem \ref{space homological stability} in the case $k=1$ and prove Theorem \ref{almost sharp homological stability for k=1}.
The argument is similar to that of Section \ref{section main theorem}, in that we prove a stronger result, namely, Theorem \ref{k 1 bound p}, in a quest to understand the terms of our augmented Mayer--Vietoris sequence for $SF_{n}(R_{w,h})$.
This argument uses a second type of augmented Mayer--Vietoris spectral sequence analogous to the spectral sequence that appears in Lemma \ref{cover j,i}.
Unlike that spectral sequence where the problem of determining the $k$-th homology of $SF_{n}(R_{w,h}; j_{0},\dots, j_{s}; i_{0},\dots, i_{p})$ is concentrated in the $E^{1}_{0,k}$-entry, this more detailed spectral sequence concentrates the problem in the $E^{p+1}_{p, k-p}$-entries, where $1\le p\le k$.
Since we are concerned with the case $k=1$, this becomes a problem of understanding the connectedness of a certain subspace of the square configuration space, which is possible due to the results of Section \ref{Square Configuration Space}.

\subsection{Acknowledgments}
The third author was supported in part by an AMS-Simons travel grant.
The third author greatly benefited from conversations with Benson Farb, Peter May, Daniel Minahan, Nathalie Wahl, and Jennifer Wilson.
We would also like to thank Fedor Manin for helpful comments on an earlier draft of this paper.

\section{Square Configuration Space}\label{Square Configuration Space}
Our proofs of Theorems \ref{space homological stability} and \ref{almost sharp homological stability for k=1} rely on a pair of augmented Mayer--Vietoris spectral sequences that arise from similar covers of $SF_{n}(R_{w,h})$ and $F_{n}(\R^{2})$.
In this section, we study the sets that define these covers and prove several facts about their connectedness that will be critical to our arguments.

If we identify $R_{w,h}$ with the Euclidean rectangle whose vertices lie at $(0,0)$, $(-w, 0)$, $(0, -h)$, and $(-w,-h)$, we can view the ordered configuration space of $n$ open unit squares in $R_{w,h}$ as a subspace of $\R^{2n}$ by sending each square to its center:
\[
SF_{n}(R_{w,h}):=\left\{(x_{1}, y_{1},\dots, x_{n}, y_{n})\in \R^{2n}\bigg|\substack{{\displaystyle-w+\frac{1}{2}\le x_i\le -\frac{1}{2}, -h+\frac{1}{2}\le y_i\le -\frac{1}{2}, \text{ and }}\\\displaystyle\max\{|x_{i}-x_{j}|, |y_{i}-y_{j}|\}\ge 1\text{ for all }i\neq j}\right\}.
\]
Using this identification, we call the unit squares tiling $R_{w,h}$ whose vertices lie at integer points of $\R^{2}$ \emph{grid-squares}.
Furthermore, this identification induces an inclusion of configuration spaces
\[
\iota:SF_{n}(R_{w,h}) \hookrightarrow F_{n}(\R^{2}).
\]
We will show that the induced map on homology
\[
\iota_{*}:H_{k}\big(SF_{n}(R_{w,h})\big)\to H_{k}\big(F_{n}(\R^{2})\big)
\]
is an isomorphism as long as $\min\{w, h\}\ge k+2$ and $wh-n\ge \max\big\{(k+1)(k+2), hk+2\big\}$.

We begin by noting that this map on homology is an isomorphism for all $k$ if $\min\{w,h\}\ge n$, as, in this case, $SF_{n}(R_{w,h})$ is homotopy equivalent to $F_{n}(\R^{2})$.

\begin{prop}\label{homotopy equivalence}
(Plachta \cite[Corollary 17]{plachta2021configuration})
If $\min\{w,h\}\ge n$, then
\[
SF_{n}(R_{w,h})\simeq F_{n}(\R^{2}).
\]
\end{prop}

Proposition \ref{homotopy equivalence} has two corollaries that we will use throughout the paper.
First, it implies that we can write $SF_{n}(R_{n,n})$ in place of $F_{n}(\R^{2})$.
Second, it tells us that the classes of $H_{k}\big(F_{n}(\R^{2})\big)$ can be realized in $SF_{n}(R_{n,n})$.
Since we are trying to determine when every class in $H_{k}\big(SF_{n}(R_{w,h})\big)$ can be thought of as a class in $H_{k}\big(F_{n}(\R^{2})\big)$ and vice versa, we recall a particularly nice basis for the latter.

Note that there are homotopy equivalences between $F_{1}(\R^{2})$ and $\R^{0}$ and between $F_{2}(\R^{2})$ and  $S^{1}$, and a class generating the latter configuration space's first homology can be interpreted as point $2$ orbiting point $1$ counterclockwise.
This class, which we denote by $W(1,2)\in H_{1}\big(F_{2}(\R^{2})\big)$, can be thought of as the Browder operation
\[
\psi:H_{k}\big(F_{n}(\R^{2})\big)\otimes H_{k'}\big(F_{n'}(\R^{2})\big)\to H_{k+k'+1}\big(F_{n+n'}(\R^{2})\big)
\]
applied to two copies of the fundamental class $W(1):=\big[F_{1}(\R^{2})\big]$ of $F_{1}(\R^{2})$; see \cite{browder1960homology}.
Iterating this operation $n-1$ times on $\big[F_{1}(\R^{2})\big]$ yields the \emph{wheel} $W(1, 2,\dots, n)\in H_{n-1}\big(F_{n}(\R^{2})\big)$, i.e., the class in which particle $2$ orbits particle $1$, and this orbiting pair is independently orbited by particle $3$, so on and so forth, i.e.,
\[
W(1, 2, \dots, n):=\psi\Bigg(\psi\bigg(\cdots\psi\Big(\psi\big(W(1),W(2)\big), W(3)\Big), \cdots\bigg),W(n)\Bigg);
\]
see Figure \ref{W(123n)onpoints}.
Permuting the labels of the points constituting $W(1, 2, \dots, n)$ gives the other wheels on $n$ points.

\begin{figure}[h]
\centering
\captionsetup{width=.8\linewidth}
\includegraphics[width=3cm]{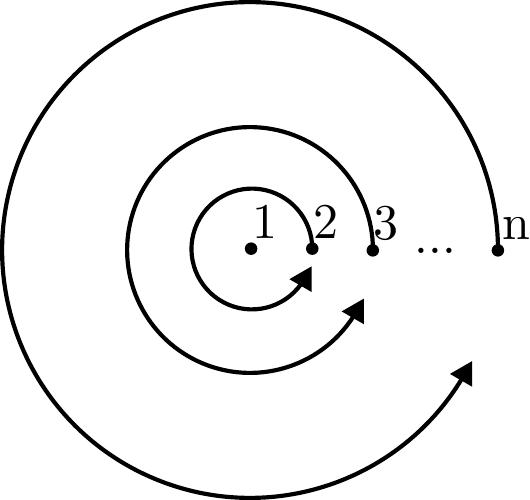}
\caption{The wheel $W(1,2,\dots, n)\in H_{n-1}\big(F_{n}(\R^{2})\big)$.
}
\label{W(123n)onpoints}
\end{figure}

In addition to the Browder operation, one can disjointly embed two copies of $\R^{2}$ in $\R^{2}$, inducing a product on $H_{*}\big(F_{\bullet}(\R^{2})\big)$:
\[
\times:H_{k}\big(F_{n}(\R^{2})\big)\otimes H_{k'}\big(F_{n'}(\R^{2})\big)\to H_{k+k'}\big(F_{n+n'}(\R^{2})\big).
\]
This product endows $H_{*}\big(F_{\bullet}(\R^{2})\big)$ with a ring structure, where Browder brackets of wheels serve as generators.

\begin{prop}\label{little cube operad structure}
\cite[Proposition 2.7, Definition 2.9, and Theorem 4.9]{sinha2006homology}
Browder brackets of wheels form a generating set for $H_{*}\big(F_{\bullet}(\R^{2})\big)$, which is an algebra satisfying the following relations:
\begin{enumerate}
 \item \[
\psi\big(W(i_{1},\dots, i_{n}), W(j_{1},\dots, j_{m})\big)-(-1)^{(n-1)(m-1)}\psi\big(W(j_{1},\dots, j_{m}), W(i_{1},\dots, i_{n})\big)=0,
\]
    \item \begin{multline*}
\psi\Big(\psi\big(W(i_{1},\dots, i_{n}), W(j_{1},\dots, j_{m})\big), W(k_{1},\dots, k_{l})\Big)\\
+\psi\Big(\psi\big(W(k_{1},\dots, k_{l}), W(i_{1},\dots, i_{n})\big), W(j_{1},\dots, j_{m})\Big)\\
+\psi\Big(\psi\big(W(j_{1},\dots, j_{m}), W(k_{1},\dots, k_{l})\big), W(i_{1},\dots, i_{n})\Big)=0,
\end{multline*}
and
    \item \[
W(i_{1},\dots, i_{n})W(j_{1},\dots, j_{m})-(-1)^{(n-1)(m-1)}W(j_{1},\dots, j_{m})W(i_{1},\dots, i_{n})=0.
\]

\end{enumerate}
\end{prop}

The first two relations allow one to replace any $\psi\big(W(i_{1},\dots, i_{n}), W(j_{1},\dots, j_{m})\big)$, with a sum of wheels on the letters $i_{1},\dots, i_{n}, j_{1},\dots, j_{m}$, implying that $H_{*}\big(F_{\bullet}(\R^{2})\big)$ is generated by wheels.
The third relation allows one to describe a basis for $H_{k}\big(F_{n}(\R^{2})\big)$ in terms of products of wheels; in the spirit of disk configuration spaces, we take Alpert and Manin's formulation of such a basis.

\begin{prop}\label{basis for homology of FnR2}
(\cite[Theorem B]{alpert2021configuration1}) For fixed $n$ and $k$, the homology group $H_{k}\big(F_{n}(\R^{2})\big)$ is free abelian with a basis consisting of products of wheels
\[
W(i_{1,1}, \dots, i_{n_{1}, 1})\cdots W(i_{1,l}, \dots, i_{n_{l}, l}),
\]
such that
\begin{enumerate}
    \item $i_{1, j}<i_{2,j}, \dots, i_{n_{j},j}$, 
    \item $n_{j}\ge n_{j+1}$, and 
    \item $i_{1, n_{j}}>i_{1, n_{j+1}}$, if $n_{j}=n_{j+1}$.
\end{enumerate}
\end{prop}

Given Proposition \ref{homotopy equivalence}, it follows that the wheel $W(1,\dots, n)$ can be realized in $H_{n-1}\big(SF_{n}(R_{n,n})\big)$; see Figure \ref{W(123n)onsquares}.
In general, this wheel can be represented as long there is an otherwise unoccupied $n\times n$ square inside $R_{w,h}$.
Our restrictions on $w$ and $h$ relative to $k$ and $n$, ensure that any basis element described in Proposition \ref{basis for homology of FnR2} can be realized in $H_{k}\big(SF_{n}(R_{w,h})\big)$; see Construction \ref{first construction for connectivity}.
The main challenge of this paper is showing that this set of classes suffices to generate $H_{k}\big(SF_{n}(R_{w,h})\big)$, i.e., there are no extra classes that come from spatial constraints that prevent the existence of some $(k+1)$-dimensional boundaries in $SF_{n}(R_{w,h})$.

\begin{figure}[h]
\centering
\captionsetup{width=.8\linewidth}
\includegraphics[width=4cm]{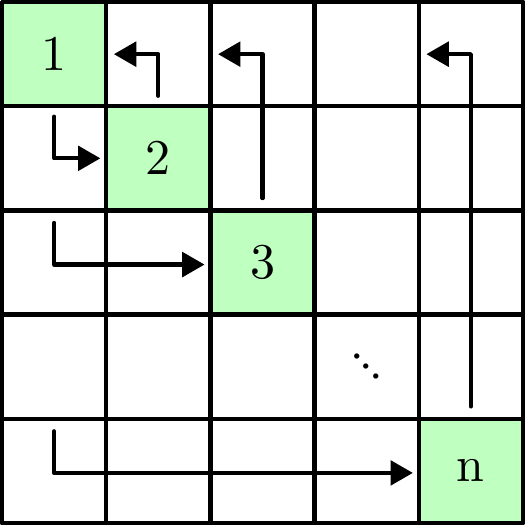}
\caption{The wheel $W(1,2,\dots, n)\in H_{n-1}\big(SF_{n}(R_{n,n})\big)$. 
Squares $1$ and $2$ orbit each other in a $2\times 2$ square, square $3$ orbits this big square in a $3\times 3$ square, etc.
As a class in $H_{n-1}\big(F_{n}(\R^{2})\big)$, this class is depicted in Figure \ref{W(123n)onpoints}.
}
\label{W(123n)onsquares}
\end{figure}

\begin{remark}\label{wh ge k+2}
In order for the wheel $W(1,\dots, k+1)$ to exist, it must be possible to vertically and horizontally align the squares $1,\dots, k+1$ inside $R_{w,h}$. 
It follows that we must have that $\min\{w,h\}\ge k+1$ in order to have any hope that $H_{k}\big(SF_{n}(R_{w,h})\big)\cong H_{k}\big(F_{n}(\R^{2})\big)$.
However, if $n>k+1$, then $\min\{w,h\}=k+1$ is not good enough, that is, in this case, the wheel $W(1,\dots, k+1)$ blocks the singleton wheel $W(k+2)$ from passing from one side of it in $R_{w,h}$ to the other;
see Figure \ref{twononhomologouscycles}.
This creates classes in $H_{k}\big(SF_{n}(R_{w,h})\big)$ that are products of the same set of wheels that are not homologous, destroying any hope of an isomorphism $H_{k}\big(SF_{n}(R_{w,h})\big)\cong H_{k}\big(F_{n}(\R^{2})\big)$.
Thus, one sees the necessity of the constraint $\min\{w,h\}\ge k+2$ in Theorems \ref{space homological stability} and \ref{almost sharp homological stability for k=1}.
\end{remark}

\begin{figure}[h]
\centering
\captionsetup{width=.8\linewidth}
\includegraphics[width=6cm]{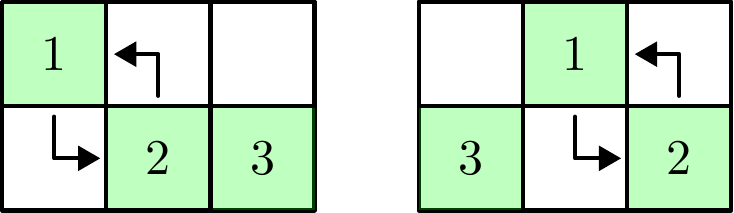}
\caption{Two non-homologous classes in $H_{1}\big(SF_{3}(R_{3,2})\big)$.
Despite the fact these cycles can be viewed as products of the same set of wheels, they represent different classes in homology, as there is no way to move the wheel $W(3)$ past the wheel $W(1,2)$, which needs the entire height of the rectangle.
}
\label{twononhomologouscycles}
\end{figure}

\begin{remark}\label{pyramids are cool too}
One can also realize $W(1,\dots, n)$ in square configuration space inside a square pyramid of base $2n-1$ and height $n$, such that at the $j$-th level, the width of the pyramid is $2n-(2j-1)$; see Figure \ref{squarewheelandpyramidwheel}.
Using this construction, one can recursively construct $W(1, \dots, n)$, treating $W(1, \dots, n-1)$ as a square pyramid of base $2n-1$ and height $n-1$.
The square and pyramid realizations for $W(1,2)$ are homotopy equivalent inside $R_{3,2}$, and in general the pyramid construction of $W(1,\dots, n)$ is homotopy equivalent to the square construction inside $R_{2n-1, n}$.
This pyramid construction plays an important role in the proof of Proposition \ref{move the big stretchy squares around}, but unless noted otherwise, we realize $W(1,\dots, n)$ as an $n\times n$ square.
Finally, we note that we expect that the square and pyramid realizations of $W(1,\dots, n)$ are minimal in the sense that any realization of this class occupies at least $n^{2}$ squares in $R_{w,h}$.
\end{remark}

\begin{figure}[h]
\centering
\captionsetup{width=.8\linewidth}
\includegraphics[width=14cm]{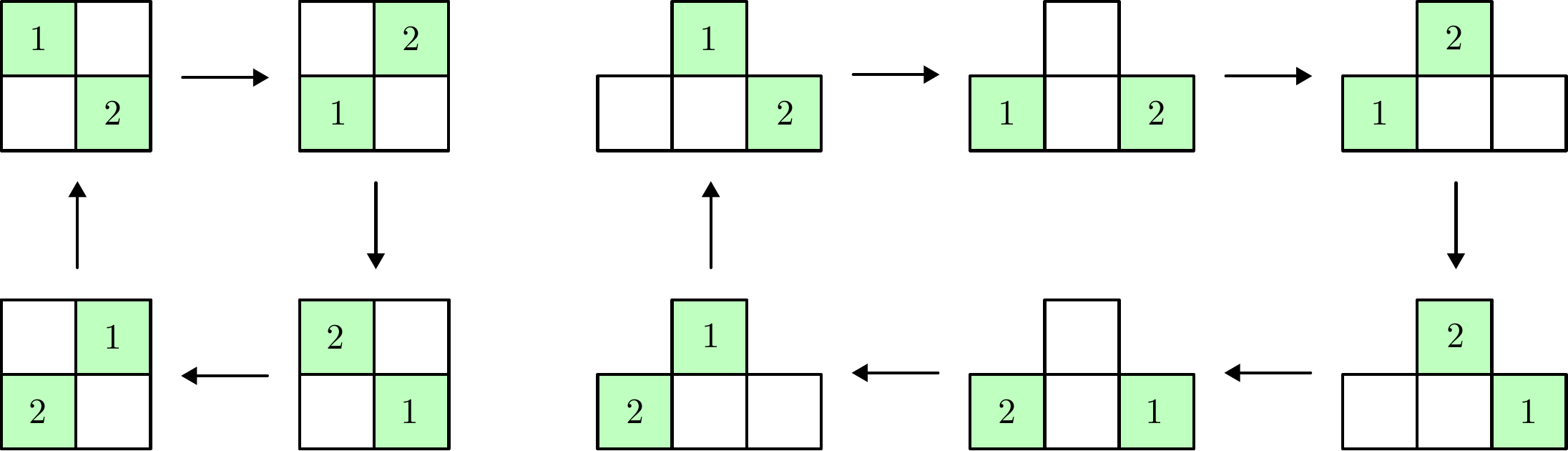}
\caption{On the left, the wheel $W(1,2)$ realized as a square.
On the right, the same wheel realized inside a square pyramid.}
\label{squarewheelandpyramidwheel}
\end{figure}

To fully take advantage of this fact, we recall a cellular model for $SF_{n}(R_{w,h})$ that we will use to establish our spectral sequences.

Since the center of any square in a configuration in $SF_{n}(R_{w,h})$ must be Chebyshev distance at least $\frac{1}{2}$ from the sides of $R_{w,h}$, the center of every square in a configuration lies inside the $(w-1)\times (h-1)$ rectangle $R^{*}_{w,h}$ whose vertices are $\{-\frac{1}{2},-\frac{1}{2}\}$, $\{-w+\frac{1}{2},-\frac{1}{2}\}$, $\{-\frac{1}{2},-h+\frac{1}{2}\}$, and $\{-w+\frac{1}{2},-h+\frac{1}{2}\}$.
This rectangle can be given the structure of a cube complex, with $2$-cells corresponding to the unit squares tiling $R^{*}_{w,h}$.
Given this structure, Alpert, Bauer, Kahle, MacPherson, and Spendlove \cite{alpert2023homology} proved that $SF_{n}(R_{w,h})$ is homotopy equivalent to $DF_{n}(R^{*}_{w,h})$, the \emph{$n$-th discrete ordered configuration space of $R^{*}_{w,h}$}, where for an arbitrary cellular complex $\Sigma$, 
\[
DF_{n}(\Sigma):=\{\sigma_{1}\times\cdots\times \sigma_{n}\in \Sigma^{n}|\overline{\sigma_{i}}\cap \overline{\sigma_{j}}=\emptyset \text{ for }i\neq j\}.
\]

\begin{prop}\label{square is cube}
(Alpert--Bauer--Kahle--MacPherson--Spendlove \cite[Theorem 3.5]{alpert2023homology})
There is an $S_{n}$-equivariant deformation retraction from $SF_{n}(R_{w,h})$ onto $DF_{n}(R^{*}_{w,h})$, where the symmetric group $S_n$ acts freely by permuting ordered configurations.
\end{prop}

Next, we recall a fact dating back to the nineteenth century about the non-path-connectedness of $SF_{wh-1}(R_{w,h})$, which we will use in conjunction with Proposition \ref{square is cube} to establish the path-connectedness of $SF_{n}(R_{w,h})$ for $wh-n\ge 2$.

\begin{prop}\label{15 puzzle has two components}
(Johnson--Story \cite{johnson1879notes})
If $\min\{w,h\}\ge 2$, then $SF_{wh-1}(R_{w,h})$ has two path-components.
\end{prop}

\begin{remark}\label{two components}
Given any triple $i,j, k$ of squares, there are two families of cyclic arrangements of these squares in $R_{2,2}$.
Johnson and Story proved that if one embeds this $R_{2,2}$ at the bottom left corner of $R_{w,h}$ and fixes an arrangement of the remaining $wh-4$ squares outside of this $R_{2,2}$, the resulting pair of configurations in $SF_{wh-1}(R_{w,h})$ that come from the two cyclic arrangements of $i,j,k$ are in different path-components of $SF_{wh-1}(R_{w,h})$.
Moreover, the two path-components correspond to the two elements of the quotient $S_{wh-1}/A_{wh-1}$ of the symmetric group by the corresponding alternating group.
\end{remark} 

We use this fact to prove the following well-known result about the path-connectedness of $SF_{n}(R_{w,h})$ for smaller $n$:

\begin{prop}\label{n le wh-2 is connected}
If $\min\{w, h\}\ge 2$ and $wh-n\ge 2$, the configuration space $SF_{n}(R_{w,h})$ is path-connected.
Therefore, the inclusion $SF_{n}(R_{w,h})$ into $F_{n}(\R^{2})$ induces an isomorphism
\[
H_{0}\big(SF_{n}(R_{w,h})\big)\cong H_{0}\big( F_{n}(\R^{2})\big).
\]
\end{prop}

\begin{proof}
Given a configuration $\mathbf{c}$ of $n\le wh-2$ squares in $R_{w,h}$, we will show that there is a path from $\mathbf{c}$ to the configuration in which the squares are arranged in lexicographical order from top to bottom, right to left, starting at the top right corner of $R_{w,h}$; we call this configuration the target configuration.
By Proposition \ref{square is cube} we may assume that each square in $\mathbf{c}$ is centered at a $0$-cell of $R^{*}_{w,h}$.
Since there are at least two unoccupied squares in the resulting configuration in $SF_{n}(R_{w,h})$ we may add $wh-1-n$ ``ghost squares''  to this configuration, labeling them $n+1, \dots, wh-1$ and placing them starting from the top right empty grid square and working down and then left.
By Remark \ref{two components}, there is a sequence of slides, i.e., a path, which allows us to move these squares into either the target configuration on $wh-1$ squares, or the target configuration in which squares $w(h-1)$ and $w(h-1)-1$ are swapped.
In the former case we are done, and in the latter case we may forget the ghost square $wh-1$, allowing us to swap the positions of squares $w(h-1)$ and $w(h-1)-1$ to get to the target configuration on $wh-2$ squares.
Forgetting the remaining ghost squares, we have found a path from any initial configuration to the target configuration, proving that $SF_{wh-2}(R_{w,h})$ is path-connected.
\end{proof}

In addition to its utility in proving Proposition \ref{n le wh-2 is connected}, the discrete model $DF_{n}(R^{*}_{w,h})$ almost gives us a sufficient cellular structure to calculate the homology of $SF_{n}(R_{w,h})$ via a Mayer--Vietoris spectral sequence.
Namely, if $i_{0}, \dots, i_{p}$ are distinct elements of $\{1,\dots, n\}$, then we define $SF_{n}(R_{w,h};i_{0}, \dots, i_{p})$ to be the subspace of $SF_{n}(R_{w,h})$ in which no square can be to the right of any of the squares $i_{0}, \dots, i_{p}$, and in which square $i_{l}$ is above square $i_{l+1}$, i.e.,
\[
SF_{n}(R_{w,h};i_{0}, \dots, i_{p}):=\left\{(x_{1}, y_{1},\dots, x_{n}, y_{n})\in SF_{n}(R_{w,h})\bigg|\substack{{\displaystyle x_{j}\le x_{i_{l}} \text{ for all }1\le j\le n \text{ and } l=0,\dots,p,}\\{\displaystyle\text{and }y_{i_{l}}>y_{i_{l+1}} \text{ for }l=0,\dots, p-1}}\right\}.
\]
One can define the subspace $F_{n}(\R^{2}; i_{0}, \dots, i_{p})$ of $F_{n}(\R^{2})$ in a similar manner.
Additionally, we say that if $j\in \{1,\dots, n\}-\{i_{0},\dots, i_{p}\}$, then square, resp. point, $j$ is \emph{free}.

These subspaces will play a fundamental role in our spectral sequence arguments; however, in $DF_{n}(R^{*}_{w,h})$, they are not cellular.
Fortunately, after subdividing the cells of $DF_{n}(R^{*}_{w,h})$ via hyperplanes of the form $x_{i}=x_{j}$, we get a cellular model for $SF_{n}(R_{w,h})$ in which the $SF_{n}(R_{w,h};i_{0}, \dots, i_{p})$ are subcomplexes, a fact we will use in Section \ref{augmented MV}.

The subspace $SF_{n}(R_{w,h};i_{0}, \dots, i_{p})$ is rather well-behaved. 
Namely, if $p+1>h$ or $p+1<0$, then it is empty, if $p+1=h$, then $SF_{n}(R_{w,h};i_{0}, \dots, i_{p})\simeq SF_{n-h}(R_{w,h})$, and if $p+1=0$, then $SF_{n}(R_{w,h};i_{0}, \dots, i_{p})=SF_{n}(R_{w,h})$.
Moreover, if $s+1+p+1\le h$ and $\{j_{0}, \dots, j_{s}\}$ and $\{i_{0}, \dots, i_{p}\}$ are disjoint, then the path components of the intersection of $SF_{n}(R_{w,h};j_{0}, \dots, j_{s})$ and $SF_{n}(R_{w,h};i_{0}, \dots, i_{p})$ are indexed by elements of $\Sigma\big((j_{0}, \dots, j_{s}),(i_{0}, \dots, i_{p})\big)$, the set of shuffles of the ordered set $(j_{0}, \dots, j_{s})$ into the ordered set $(i_{0}, \dots, i_{p})$.

\begin{prop}\label{intersection is shuffle}
If $s+1+p+1\le h$ and $\{j_0,\ldots,j_s\}\cap\{i_0,\ldots,i_p\}=\varnothing$, then
\[
SF_{n}(R_{w,h};j_{0}, \dots, j_{s})\cap SF_{n}(R_{w,h};i_{0}, \dots, i_{p})=\bigsqcup_{\sigma\in \Sigma\big((j_{0}, \dots, j_{s}),(i_{0}, \dots, i_{p})\big)}SF_{n}(R_{w,h}; \sigma),
\]
where in $SF_{n}(R_{w,h}; \sigma)$ the squares $j_{0}, \dots, j_{s}, i_{0}, \dots, i_{p}$ are arranged from top to bottom such that their order is given by $\sigma$.
\end{prop}
The path connectivity of each $SF_n(R_{w,h};\sigma)$ is addressed in Proposition \ref{connected single right configuration space} below. Additionally, we may assume that in $SF_{n}(R_{w,h};i_{0}, \dots, i_{p})$ the squares $i_{0}, \dots, i_{p}$ are tangent to the right side of $R_{w,h}$.

\begin{prop}\label{move squares to right side}
There is a deformation retraction from $SF_{n}(R_{w,h};i_{0}, \dots, i_{p})$ onto its subspace in which the squares $i_{0}, \dots, i_{p}$ are tangent to the right side of $R_{w,h}$.
\end{prop}

\begin{proof}
This retraction is given by the straight-line homotopy that moves the squares $i_{0}, \dots, i_{p}$ to the right.
\end{proof}

Furthermore, if $w$ and $h$ are large with respect to $n$, then $SF_{n}(R_{w,h}; i_{0}, \dots, i_{p})$ and $F_{n}(\R^{2}; i_{0}, \dots, i_{p})$ are homotopy equivalent.

\begin{prop}\label{homotopy equivalence for i0...ip}

If $\min\{w,h\}\ge n$ and $0\le p+1\le h$, then
\[
SF_{n}(R_{w,h}; i_{0}, \dots, i_{p})\simeq F_{n}(\R^{2}; i_{0}, \dots, i_{p}).
\]
\end{prop}

\begin{proof}
This follows from Proposition \ref{homotopy equivalence}, noting that we can push the free squares in $SF_{n}(R_{w,h}; i_{0}, \dots, i_{p})$ away from the right-most column of $R_{w,h}$, which contains the $i_{0},\dots, i_{p}$.
\end{proof}

Proposition \ref{homotopy equivalence for i0...ip} implies that  Propositions \ref{intersection is shuffle} and \ref{move squares to right side} hold for $F_{n}(\R^{2}; i_{0}, \dots, i_{p})$, and even more can be said in this setting.

\begin{prop}\label{forget points on the right}
There is a homotopy equivalence between $F_{n}(\R^{2}; i_{0}, \dots, i_{p})$ and $F_{n-p-1}(\R^{2})$.
\end{prop}

\begin{proof}
By Propositions \ref{move squares to right side} and \ref{intersection is shuffle}, we may assume in $F_{n}(\R^{2}; i_{0}, \dots, i_{p})$ that the points $i_{0}, \dots, i_{p}$ lie strictly to the right of the other points in configuration.
As such, there is a map from $F_{n}(\R^{2};i_{0}, \dots, i_{p})$ to itself that sends the point $i_{l}$ to $(-\frac{1}{2}, -l+\frac{1}{2})\in \R^{2}$, with the other $n-p-1$ points lying in the open left half plane defined by $x<-\frac{1}{2}$.
This space is homeomorphic to $\R^{2}$, giving the desired homotopy equivalence between $F_{n}(\R^{2}; i_{0}, \dots, i_{p})$ and $F_{n-p-1}(\R^{2})$.
\end{proof}

Next, we note a path-connectedness result that appears rather innocuous on first glance, though we will come to see that it serves as a strong foothold for the inductive arguments of Section \ref{section main theorem}.

\begin{prop}\label{connected single right configuration space}
If $\min\{w-1, h\}\ge 2, wh-n\ge 2$, and $0\le p+1\le h$, then $SF_{n}(R_{w,h}; i_{0}, \dots, i_{p})$ is path-connected.
In particular, the inclusion of $SF_{n}(R_{w,h}; i_{0}, \dots, i_{p})$ into $F_{n}(\R^{2}; i_{0}, \dots, i_{p})$ induces an isomorphism
\[
H_{0}\big(SF_{n}(R_{w,h}; i_{0}, \dots, i_{p})\big)\cong H_{0}\big(F_{n}(\R^{2}; i_{0}, \dots, i_{p})\big).
\]
\end{prop}

\begin{proof}
By Proposition \ref{move squares to right side}, we may assume that the squares $i_{0}, \dots, i_{p}$ are tangent to the right side of $R_{w,h}$.
Moreover, Proposition \ref{square is cube} can be extended to tell us that given any such configuration, we may slide the squares so that they are all centered at the $0$-cells of $R^{*}_{w,h}$.
We show that we can slide the squares in any configuration in $SF_{n}(R_{w,h}; i_{0}, \dots, i_{p})$, to get the target configuration in which squares $i_{0}, \dots, i_{p}$ are at the top right of $R_{w,h}$ and the free squares are in lexicographical order when read from top to bottom, right to left, starting directly under $i_{p}$.

Note that in any configuration in which the $i_{j}$ are tangent to the right side of $R_{w,h}$, we may assume that square $i_{l}$ is tangent to square $i_{l+1}$, and that square $i_{0}$ is tangent to the top of $R_{w,h}$.
To see this, assume that there is some $l$ such that there is a free square $j$ between squares $i_{l}$ and $i_{l+1}$.
By sliding the free squares around the rectangle, we may assume that the square immediately to the left of $j$ is unoccupied.
As such, we may slide square $j$ out of the right-most column of $R_{w,h}$, and slide $i_{l}$ up.
It follows that we may assume that the squares $i_{0}, \dots, i_{p}$ are at the top right of $R_{w,h}$.

Next, we note that there are either $0$, $1$, or $2$ or more grid-squares in the right-most column of $R_{w,h}$ that are not occupied by the $i_{0}, \dots, i_{p}$.
If there are $0$ such grid-squares, there are at most $(w-1)h-2$ squares in the rectangle $R_{w-1, h}$ at the left of $R_{w,h}$.
Viewing the free squares as a configuration in $SF_{n-h}(R_{w-1,h})$, we can slide these squares around to get the target configuration since
the resulting configuration space is path-connected by Proposition \ref{n le wh-2 is connected}.

If there is $1$ such grid square below the $i_{j}$, we can slide one of the remaining $n-p-1$ free squares, say square $j$ into it.
As such there are at most $(w-1)h-2$ free squares in the $R_{w-1,h}$ at the left of $R_{w,h}$.
If $j$ is first free square in the lexicographical order, then we can use Proposition \ref{n le wh-2 is connected} to slide the remaining free squares into the target configuration, by treating them as a configuration in $SF_{n-h}(R_{w-1,h})$.
Otherwise, treating the remaining free squares as a configuration in $SF_{n-h}(R_{w-1,h})$, we can can use Proposition \ref{n le wh-2 is connected} to slide the remaining free squares into a configuration such that the two right-most squares in the bottom row of this $R_{w-1,h}$ are unoccupied, and such that first free square in the lexicographical order is in the second from the bottom row in the right-most column of this $R_{w-1,h}$.
This yields a square pyramid of base 3 and height 2 at the bottom right of $R_{w,h}$ containing only the free square $j$ in the right-most column of $R_{w,h}$ and first free square in the lexicographical order.
Using a sequence of slides, we may interchange these squares; see Figure \ref{squarewheelandpyramidwheel}.
This yields a configuration of $n-h$ squares in the left $R_{w-1,h}$ in $R_{w,h}$; using Proposition \ref{n le wh-2 is connected}, there is a sequence of slides that move these squares to the target configuration.

Finally, if there are at least $s\ge 2$ such grid-squares, we see that a similar operation allows us to put $s\le h-p-1$ of the remaining $n-p-1$ free squares below $i_{p}$.
Using Proposition \ref{n le wh-2 is connected}, we may arrange the squares in the left $R_{w-1,h}$ such that the first $h-p-1$ free squares in the lexicographical order are in the $R_{w,h-p-1}$ found at the bottom of $R_{w,h}$; moreover, we can ensure that there are at least $2$ unoccupied squares in this $R_{w,h-p-1}$.
Again, we note that Proposition \ref{n le wh-2 is connected} tells us we can put the first $n-h$ free squares in this subrectangle to the far right, so they lie below $i_{0}, \dots, i_{p}$ in $R_{w,h}$.
Another application of Proposition \ref{n le wh-2 is connected}, tells us we can arrange the remaining squares in the left $R_{w-1,h}$ to be in the target configuration, completing the proof.
\end{proof}

Our proof Theorem \ref{space homological stability} relies on an inductive argument based on the number of squares fixed on the right side of $R_{w,h}$.
As such, we will need to show that if either 
\begin{enumerate}
    \item $\min\{w-1,h\}\ge k+2$, $wh-n\ge \max\big\{(k+1)(k+2),hk+2\big\}$, and $0\le p+1\le h$, or
    \item $w\ge h=k+2$, $wh-n\ge (k+1)(k+2)$, and $p=-1$,
\end{enumerate}
then
\[
H_{k}\big(SF_{n}(R_{w, h};i_{0}, \dots, i_{p})\big)\cong H_{k}\big(F_{n}(\R^{2};i_{0}, \dots, i_{p})\big).
\]
This is done via another Mayer--Vietoris spectral sequence in which we cover $SF_{n}(R_{w, h};i_{0}, \dots, i_{p})$ with sets of the form $SF_{n}(R_{w,h};j;i_{0}, \dots, i_{p})$.
In general, given distinct $j_{0}, \dots, j_{s},i_{0}, \dots, i_{p}$ in $\{1, \dots, n\}$, we define the space $SF_{n}(R_{w,h};j_{0}, \dots, j_{s};i_{0}, \dots, i_{p})$ to be the subspace of $SF_{n}(R_{w,h};i_{0}, \dots, i_{p})$ in which no square other than squares $i_{0}, \dots, i_{p}$ can be to the right of the squares $j_{0}, \dots, j_{s}$, which are aligned vertically in that order.

Like $SF_{n}(R_{w,h};i_{0}, \dots, i_{p})$, we may subdivide $DF_{n}(R_{w,h})$ via hyperplanes to get a complex in which $SF_{n}(R_{w,h};j_{0}, \dots, j_{s};i_{0}, \dots, i_{p})$ is a subcomplex.
Moreover, we may assume that in the configuration space $SF_{n}(R_{w,h};j_{0}, \dots, j_{s};i_{0}, \dots, i_{p})$, the squares $i_{0}, \dots, i_{p}$ are tangent to the right side of $R_{w, h}$ and that the squares $j_{0}, \dots, j_{s}$ are in the two right-most columns of $R_{w,h}$.

\begin{prop}\label{squares in the two rightmost columns}
There is a deformation retraction from $SF_{n}(R_{w,h};j_{0}, \dots, j_{s};i_{0}, \dots, i_{p})$ onto its subspace in which the squares $i_{0}, \dots, i_{p}$ are tangent to the right side of $R_{w,h}$, and the squares $j_{0}, \dots, j_{s}$ are in the two right-most columns of $R_{w,h}$.
\end{prop}

The desired deformation retraction is a straight-line homotopy, and in $\R^{2}$ even more is true, as the configuration space $F_{n}(\R^{2};j_{0}, \dots, j_{m};i_{0}, \dots, i_{p})$ is homotopy equivalent to $F_{n}(\R^{2};j_{0}, \dots, j_{m},i_{0}, \dots, i_{p})$.

\begin{prop}\label{in points can move two columns left}
There are homotopy equivalences
\[
F_{n}(\R^{2};j_{0}, \dots, j_{s};i_{0}, \dots, i_{p})\simeq F_{n-p-1}(\R^{2};j_{0}, \dots, j_{s})\simeq F_{n}(\R^{2};j_{0}, \dots, j_{m},i_{0}, \dots, i_{p}).
\]
\end{prop}

Unfortunately, analogues of Propositions \ref{forget points on the right} and \ref{in points can move two columns left} do not hold for general configuration spaces of squares in a rectangle.
That said, the following proposition does hold for these square configuration spaces, a fact we will use to prove Theorem \ref{almost sharp homological stability for k=1}.

\begin{prop}\label{connected double right configuration space}
If $n$, $w$, $h$, $s$ and $p$ are such that $\min\{w-1,h\}\ge 3$, $wh-n\ge \max\{6, h-p-1\}$, 
$0\le s\le 2$, and $0\le s+1+p+1\le h$, then $SF_{n}(R_{w,h}; j_{0}, \dots, j_{s};i_{0}, \dots, i_{p})$ is path-connected.
In particular, the inclusion of the square configuration space $SF_{n}(R_{w,h}; j_{0}, \dots, j_{s};i_{0}, \dots, i_{p})$ into the point configuration space $F_{n}(\R^{2}; j_{0}, \dots, j_{s};i_{0}, \dots, i_{p})$ induces isomorphisms
\begin{align*}
H_{0}\big(SF_{n}(R_{w,h}; j_{0}, \dots, j_{s};i_{0}, \dots, i_{p})\big)&\cong H_{0}\big(SF_{n}(R_{w,h}; j_{0}, \dots, j_{s}, i_{0}, \dots, i_{p})\big)\\
&\cong H_{0}\big(F_{n}(\R^{2}; j_{0},\dots, j_{s}, i_{0}, \dots, i_{p})\big)\\
&\cong H_{0}\big(F_{n}(\R^{2}; j_{0}, \dots, j_{s}; i_{0}, \dots, i_{p})\big).
\end{align*}
\end{prop}

\begin{proof}

We show that we can slide the squares in any configuration in $SF_{n}(R_{w,h}; j_{0}, \dots, j_{s};i_{0}, \dots, i_{p})$ to the target configuration in which the squares $j_{0}, \dots, j_{s}, i_{0},\dots, i_{p}$ are the top $s+1+p+1$ squares of the right-most column of $R_{w,h}$ in that order, and the free squares are arranged in lexicographical order from top to bottom, right to left, starting directly below $i_{p}$.

If $p=-1$, then we can apply Proposition \ref{connected single right configuration space} to prove this result.
Otherwise, by Proposition \ref{squares in the two rightmost columns}, we may assume that the squares $j_{0}, \dots, j_{s},i_{0}, \dots, i_{p}$ are in the two right-most columns of $R_{w,h}$.
If the squares $j_{0}, \dots, j_{s}$ are already in the right-most column, then proceed with the next step; otherwise, note that since $s+1+p+1\le h$, we may slide the squares $i_{0}, \dots, i_{p}$ up and down so that the squares immediately to the right the $j_{0}, \dots, j_{s}$ are unoccupied. 
As such, we may slide the squares $j_{0}, \dots, j_{s}$ into the right-most column of $R_{w,h}$.
Moreover, applying Proposition \ref{connected single right configuration space} to this configuration as an element of $SF_{n}(R_{w,h};\sigma)$, where $\sigma$ is the shuffle of $(j_{0},\dots, j_{s})$ into $(i_{0},\dots, i_{p})$, we can move these squares to be at the top of the right-most column of $R_{w,h}$.
Additionally, since $(w-1)h+p+1\ge n$, we may assume that the free squares are all in the $R_{w-1,h}$ in the left side of $R_{w,h}$.

Next, we consider two cases.
If $s+1+p+1=h$, then since $wh-n\ge 6$, Proposition \ref{n le wh-2 is connected} tells us that we may slide the $n-h\le(w-1)h-6$ free squares about the $R_{w-1,h}$ in the left side of $R_{w,h}$ such that the grid-squares immediately to the left of the $j_{0},\dots, j_{s}$ are unoccupied. 
Moreover, if there is an $i_{l}$ directly above one of the $j$, we can further arrange the free squares such that the grid-square immediately to its left is unoccupied.
As such, we may slide $j_{0}, \dots, j_{s}$ into the second right-most column of $R_{w,h}$ and up one square if there is an $i_{l}$ immediately above it, the offending $i_{l}$ down one square, and the $j_{0},\dots, j_{s}$ back to the right, moving the $j_{0},\dots, j_{s}$ above the offending $i_{l}$.
Repeating this process if necessary, we can arrange the $j_{0}, \dots, j_{s}$ to be at the top of the right-most column of $R_{w,h}$ and the $i_{0},\dots, i_{p}$ to be directly below them.
Applying Proposition \ref{n le wh-2 is connected} allows us to rearrange the free squares to get the target configuration.

If $s+1+p+1<h$, then since $(w-1)h+p+1\ge n$, there are at least $s+1$ empty squares in the $R_{w-1,h}$ in the left side of $R_{w,h}$. 
We can slide these free squares about so that the grid-square immediately to the left of each of the $j_{0},\dots, j_{s}$ is unoccupied.
Slide all the $i_{0},\dots, i_{p}$ that are below $j_{s}$ to the bottom of the right-most column of $R_{w,h}$, and let $i_{l}$ be the last of the $i_{0},\dots, i_{p}$ above any of the $j_{0}, \dots, j_{s}$.
Since $s+1+p+1<h$, there is at least one square gap between $j_{s}$ and $i_{l+1}$.
Slide the $j_{0}, \dots, j_{s}$ to the left, the $i_{l}$ down to the square directly above $i_{l+1}$, and then the $j_{0}, \dots, j_{s}$ back to the right.
Repeating this process if necessary, we can arrange the squares such that the $j_{0},\dots, j_{s}$ are above the $i_{0}, \dots, i_{p}$.
Applying Proposition \ref{n le wh-2 is connected}, allows us to rearrange the free squares to get the target configuration.
\end{proof}

One could consider subspaces of $SF_{n}(R_{w,h}; j_{0},\dots, j_{s}; i_{0}, \dots, i_{p})$ where there is a third right-most set of squares.
Fortunately, this is not necessary for our purposes due to Lemma \ref{cover j,i}, so we do not study these spaces.

Having proved several results about the path-connectedness of unit square configuration spaces, we turn to defining ``big square'' configuration spaces, which we will use to study the higher homology groups of $SF_{n}(R_{w,h}; i_{0}, \dots, i_{p})$.
In general, the relevant big square configuration spaces are not path-connected, though if we allow a little plasticity, we can find a path-component that suffices for our purposes.

\subsection{Big Square Configuration Space}\label{Big square config}
Proposition \ref{homotopy equivalence} proves that we can realize the wheel $W(i_{0}, \dots, i_{m})$ inside an $(m+1)\times(m+1)$ square; see Figure \ref{W(123n)onsquares}.
Proposition \ref{basis for homology of FnR2} tells us that products of these wheels form a basis for $H_{k}\big(F_{n}(\R^{2})\big)$; since we wish to determine when
\[
H_{k}\big(SF_{n}(R_{w,h})\big)\cong H_{k}\big(F_{n}(\R^{2})\big),
\]
it will be useful to determine when we can realize these wheels in $SF_{n}(R_{w,h})$.
To do this we treat these wheels as ``big squares'' and consider the configuration space of such squares in $R_{w,h}$.

Recall that a \emph{weighted set} $(A, \mathcal{W})$ is a pair of sets with a fixed surjection from $A$ to $\mathcal{W}$ that sends an element $a\in A$ to its \emph{weight} $w_{a}\in \mathcal{W}$.

\begin{defn}
Given a weighted set $(A, \mathcal{W})$, the \emph{big square configuration space} $SF_{(A, \mathcal{W})}(R_{w,h})$ is the space of ways of embedding a set of $|A|$ open squares with labels in $A$ into the rectangle $R_{w,h}$ such that the square labeled $a$ has side length $w_{a}$.
\end{defn}

Our big square configuration spaces arise by treating the wheel $W(i_{0}, \dots, i_{m})$ as an $(m+1)\times (m+1)$ square labeled $i_{0}\cdots i_{m}$.
In particular, given Proposition \ref{basis for homology of FnR2}, our weighted sets will have labels $a$ consisting of injective words $i_{0}\cdots i_{m}$ corresponding to disjoint ordered sets $(i_{0}, \dots, i_{m})$ that partition $\{1, \dots, n\}$, such that $i_{0}\le i_{1},\dots, i_{m}$, with weight $w_{a}=m+1$.

Like their unweighted counterparts, we will need to consider subspaces of $SF_{(A,\mathcal{W})}(R_{w, h})$ in which certain squares are right-most.
Given a subset $\{a_{0}, \dots, a_{p}\}\subset A$, we write
\[
SF_{(A,\mathcal{W})}(R_{w, h}; a_{0}, \dots, a_{p})
\]
for the subspace of $SF_{(A,\mathcal{W})}(R_{w, h})$ in which square $a_{l}$ is above square $a_{l+1}$, the right sides of the squares $a_{0}, \dots, a_{p}$ are aligned, and none of the other squares are to the right of these squares.
Building off the idea that $SF_{n}(R_{w,h}; i_{0}, \dots, i_{p})$ is connected, we see that if only one square is big, i.e., not a unit square, then $SF_{(A,\mathcal{W})}(R_{w, h}; a_{0}, \dots, a_{p})$ is connected.
We will use the case where the big square has side length $2$ in Section \ref{tighter bound section} to prove Theorem \ref{almost sharp homological stability for k=1}.

\begin{prop}\label{free aj in k=1}
Given $n$, $w$, $h$, $k$ and $p$ such that $\min\{w-1,h\}\ge k+2$, $wh-n\ge (k+1)(k+2)$, and $0\le p+1\le h$, let $(A, \mathcal{W})$ be a weighted set such that $\sum w_{a}=n$, and $w_{a_{j}}=1$ for $j\neq p+1$, whereas for $j=p+1$, $w_{a_{p+1}}=k+1$.
Then the big square configuration space $SF_{(A,\mathcal{W})}(R_{w, h}; a_{0}, \dots, a_{p})$ is path-connected.
\end{prop}

\begin{proof}
The $k=0$ case is proven in Proposition \ref{connected single right configuration space}, so we restrict our attention to $k\ge 1$.
We show that the squares in any configuration can be moved to the target configuration in which squares $a_{0},\dots, a_{p}$ are at the top of the right-most column of $R_{w,h}$, the big square $a_{p+1}$ is at the top left of $R_{w,h}$, and the free unit squares are in lexicographical order when read from top to bottom, right to left, starting below square $a_{p}$.

After moving the big squares to be centered on the grid-squares of $R_{w,h}$---$a_{p+1}$ can be positioned so that it lies in only $(k+1)^{2}$ grid-squares---we see that at least $2(k+1)$ grid-squares of $R_{w,h}$ are unoccupied.
Note, we can push the $a_{0}, \dots, a_{p}$ to the right, to ensure they are in the right-most column of $R_{w,h}$.
Having done so, we can move the free unit squares so that the $k+1$ squares immediately to the left of $a_{p+1}$ are unoccupied, allowing us to slide it to the left.
Repeating this process if necessary, we can slide this $(k+1)\times (k+1)$ square so that it is tangent to the left side of $R_{w,h}$; a similar argument shows that we can move it up so that it is also tangent to the top of $R_{w,h}$. 

Next, in a manner similar to that of Proposition \ref{connected single right configuration space}, we see that we can rearrange the unit squares, so that the $a_{0},\dots, a_{p}$ are at the top right of $R_{w,h}$.

Finally, we note that we can rearrange the free unit squares to be in the target configuration.
To see this, first note we can slide the free unit squares to get a configuration of the same shape as the target configuration.
Next, we consider four cases, depending on $w$, $h$, and $p$.

If $w=k+3$ and $h=k+2$, we leave it to the reader to verify that such a configuration of at most $k+4$ big squares can be moved to the target configuration regardless of $p$---the key idea is that the empty row below square $a_{p+1}$ allows us to interchange the position of any two free squares since there can be at most $k+3$ unit squares in the configuration. 

If $w=k+3$ and $h>k+2$ and $p\ge k$, note that all the free unit squares can be placed in the $R_{k+3,h-k-1}$ at the bottom of $R_{k+3,h}$, and at least $k+1$ grid-squares are unoccupied.
Applying Proposition \ref{connected single right configuration space} to the resulting configuration in $SF_{n-2(k+1)}(R_{k+3,h-k-1}; a_{k+1},\dots, a_{p})$, we see that we can rearrange it to get the target configuration.
If $0\le p<k$, then since there are $2(k+1)$ unoccupied squares, we can rearrange the free unit squares so that they all lie in the $R_{k+3,h-k-1}$ at the bottom of $R_{k+3,h}$.
Since there are at most $(k+3)(h-k-1)-1$ such free squares, Proposition \ref{15 puzzle has two components} and Remark \ref{two components} tell us that we can rearrange this configuration so that the first free square in the lexicographical order is at the top right corner of this subrectangle.
Sliding it up to be tangent to $a_{p}$, and treating it as a right-most square, we see that there are now $(k+3)(h-k-1)-2$ free squares in the $R_{k+3,h-k-1}$ at the bottom of $R_{k+3,h}$.
Using Proposition \ref{n le wh-2 is connected}, we may arrange these free squares, so that the next $k+p-1$ free squares in the lexicographical order are at the top right of this subrectangle.
Sliding them up (or right then up), and reapplying Proposition \ref{n le wh-2 is connected} allows us to see that the squares can be put in the target configuration.
If $p=-1$, then a similar argument as Proposition \ref{connected single right configuration space} works, yielding that this space is path-connected.

If $w>k+3$ and $h=k+2$, then we can ignore the $(k+1)$ left-most columns of $R_{w,k+2}$ and get a configuration of unit squares in $R_{w-2,k+2}$, with at least $(k+1)\ge2$ unoccupied grid-squares.
Applying Proposition \ref{connected single right configuration space}, we see that we can rearrange them to get the target configuration.

If $w>k+3$ and $h>k+2$, then we can once again use Proposition \ref{connected single right configuration space} first on the configuration restricted to the bottom $R_{w,h-k-1}$, then restricted to the right $R_{w-k-1,h}$, and finally again on the bottom $R_{w,h-k-1}$ to get the desired configuration.
\end{proof}

Ideally, more general big square configuration spaces would be connected, as that would greatly simplify our arguments.
Unfortunately, this is not the case.
Still we are able to describe a nice path-component of these configuration spaces in the following construction, which we will use to study the differentials of our spectral sequences.

\begin{const}\label{first construction for connectivity}
Given $n$, $w$, $h$, and $\kappa$, where $\min\{w-1,h\}\ge \kappa+2$, $wh-n\ge \max\big\{(\kappa+1)(\kappa+2), h\kappa+2\big\}$,
 and $0\le s+1, p+1\le h$,
let $(A, \mathcal{W})$ be a weighted set such that $\sum w_{a}=n$, $\sum(w_{a}-1)=k\le \kappa$, and $w_{a_{0}}=\cdots=w_{a_{s+1+p}}=1$.
We describe a big square configuration in $SF_{(A, \mathcal{W})}(R_{w, h}; a_{0},\dots, a_{s}; a_{s+1},\dots, a_{s+1+p})$.

Without loss of generality, we can assume that in $SF_{(A, \mathcal{W})}(R_{w, h}; a_{0},\dots, a_{s}; a_{s+1},\dots, a_{s+1+p})$, the squares $a_{s+1+p+1},\dots, a_{s+1+p+1+m}$ are the free squares of size bigger than $1$.
Moreover, we may assume that $w_{a_{s+1+p+1}}\ge\cdots\ge w_{a_{s+1+p+1+m}}$, and that if $w_{a_{s+1+p+1+j}}=w_{a_{s+1+p+1+j+1}}$, then the smallest label appearing in $a_{s+1+p+1+j}$ is bigger than the smallest label appearing in $a_{s+1+p+1+j+1}$
Additionally, we may assume that $a_{s+1+p+1+m+1},\dots, a_{s+1+p+1+m+l}$ are the free squares of size $1$; without loss of generality we may further assume that in the lexicographical order the label of $a_{s+1+p+1+m+1+j}$ comes before the label of $a_{s+1+p+1+m+1+j+1}$.

If $s+1+p+1\le h$, place the squares $a_{0},\dots, a_{s+1+p}$ in a column starting at the top of the right-most column of $R_{w,h}$.
If $s+1+p+1>h$, place the squares $a_{s+1}, \dots, a_{s+1+p}$ in a column starting at the top of the right-most column of $R_{w,h}$.
Then place the squares $a_{0}, \dots, a_{s}$ in a column starting at the top of the second right-most column of $R_{w,h}$.

Next, we place the big squares $a_{s+1+p+1},\dots, a_{s+1+p+1+m}$.
This depends on $m$, and we consider $4$ cases:
\begin{itemize}
\item $m=0$. In this case, all the squares $a_{s+1+p+1},\dots,a_{s+1+p+1+m+l}$ have side length $1$, and proceed to the next step.
\item $m=1$. In this case, there is a single square $a_{s+1+p+1}$ of size greater than $1$, namely of side length $k+1$. 
Place it at the top left corner of $R_{w,h}$, and proceed to the next step. 
\item $m=2$. In this case, there are two squares $a_{s+1+p+1}, a_{s+1+p+2}$ of side length greater than $1$.
Place square $a_{s+1+p+1}$ at the top left corner of $R_{w,h}$, and square $a_{s+1+p+2}$ directly below it so that their left sides are aligned. 
Proceed to the next step.
\item $m\ge 3$. Pair off the squares $a_{s+1+p+1}$ and $a_{s+1+p+2}$, and $a_{s+1+p+3}$ and $a_{s+1+p+4}$, etc. If $m$ is odd, leave the last square by itself.
Having paired them off, group them into $w_{a_{s+1+p+2j-1}}$ by $w_{a_{s+1+p+2j-1}}+w_{a_{s+1+p+2j}}$ rectangles, so that square $a_{s+1+p+2j-1}$ is directly above square $a_{s+1+p+2j}$, and so that their left sides are aligned.
Place these rectangles so that the top left corner of rectangle $j+1$ touches the top right corner of rectangle $j$, and so that the top left corner of rectangle $1$ touches the top left corner of $R_{w,h}$, and proceed with the next step. 
\end{itemize}

Finally, we place the unit squares $a_{s+1+p+1+m+1},\dots, a_{s+1+p+1+m+l}$, doing so in lexicographical order.
First fill out the column that the squares $a_{0}, \dots, a_{s}$ are in.
Then, starting from the column two to the left of the column the $a_{0}, \dots, a_{s}$ are in, start placing the remaining squares, working top to bottom and right to left avoiding the rectangular hull of the $a_{s+1+p+1},\dots, a_{s+1+p+1+m}$; see Figure \ref{bigsquareconfig}.
\end{const}

\begin{figure}[h]
\centering
\captionsetup{width=.8\linewidth}
\includegraphics[width=12cm]{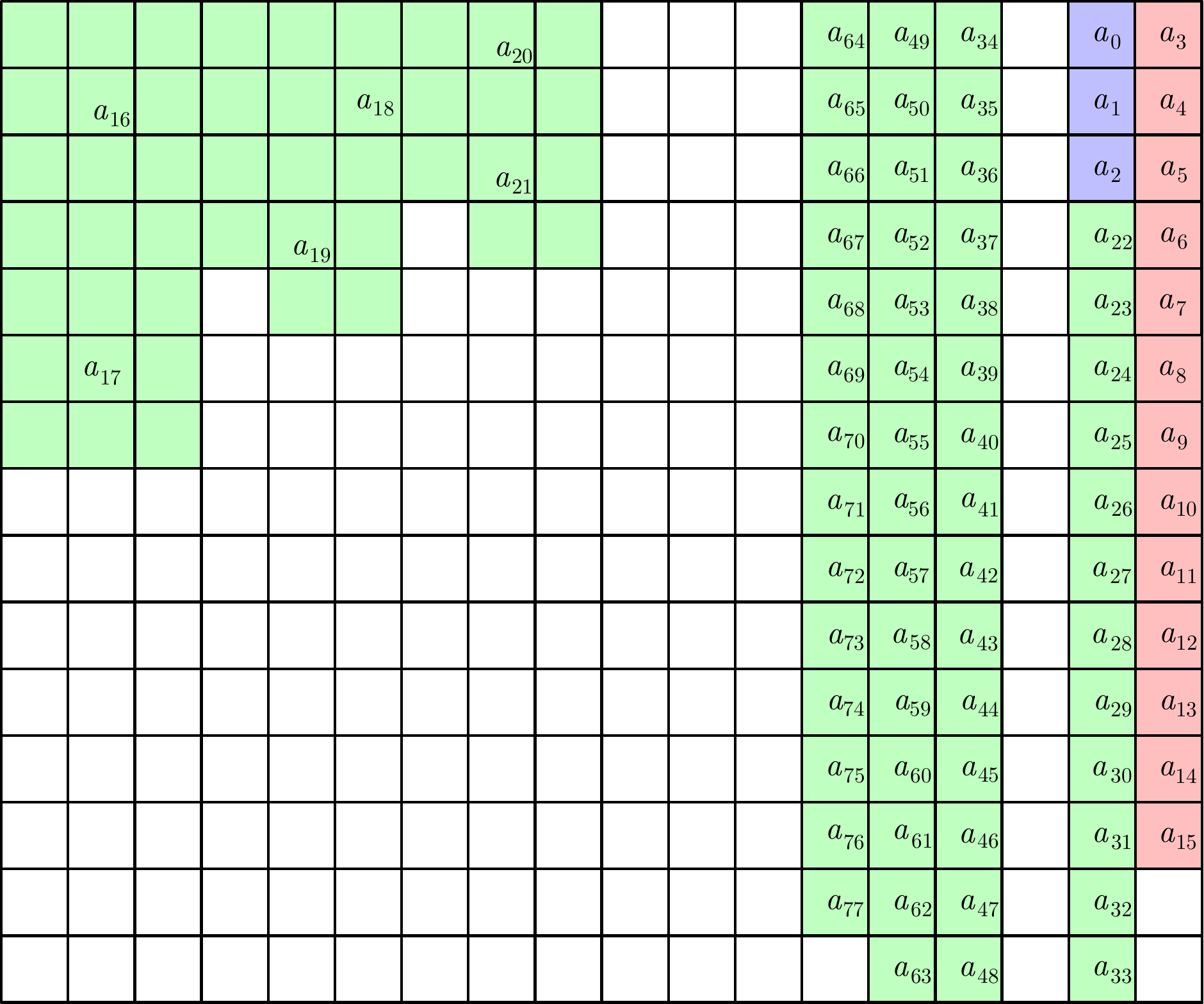}
\caption{A configuration described in Construction \ref{first construction for connectivity}.
In this case, the total weight of the big squares in this configuration is $88$, which is maximal if $\kappa=12$.
This configuration corresponds to a non-trivial class in $H_{12}\big(SF_{88}(R_{18, 15}; a_{3},\dots, a_{15})\big)\cong H_{12}\big(F_{88}(\R^{2}; a_{3},\dots, a_{15})\big)$.
Note that the greater the number of ``big squares,'' the more excess space there is in the rectangle, i.e., there are 158 unoccupied squares in $R_{18, 15}$, but if we replaced squares $a_{16}, \dots, a_{21}$ with a single $11\times 11$ square and five unit squares, we would get another configuration that leads to a class in $H_{12}\big(SF_{88}(R_{18, 15}; a_{3},\dots, a_{15})\big)$, where there are only $72$ unoccupied squares in $R_{18,15}$.
By further replacing this new square and $a_{1}$ and $a_{2}$ with a $13\times 13$ square, we would get another configuration corresponding to a class in $H_{12}\big(SF_{88}(R_{18, 15}; a_{3},\dots, a_{15})\big)$, with only $26$ squares being unoccupied in $R_{w,h}$.
}
\label{bigsquareconfig}
\end{figure}

\begin{remark}
Note that if $m\ge 3$, the width of rectangular hull of the squares $a_{s+1+p+1},\dots, a_{s+1+p+1+m}$ is at most $k+1$ and of height at most $k+1$.
In the $m=2$ case, this rectangle has width at most $k$, and height $k+2$.
In these settings, the total weight of the $a_{s+1+p+1},\dots, a_{s+1+p+1+m}$ is bigger than $k+1$, so it suffices to check that these yield valid configurations if we have either $0$ or $1$ big squares.
This is immediate.
\end{remark}

Note that if we assume that no squares other than $a_{0},\dots, a_{s+1+p}$ that are of size smaller than $r\ge 2$ where $k+r-1\le \kappa-1$, then the $r$ squares immediately to the left of each of the $a_{0},\dots, a_{s}$ are unoccupied.
This allows us to show that in certain circumstances two configurations arising from Construction \ref{first construction for connectivity} are in the same path-connected of the big square configuration space. 
We will use this fact to determine the image of the $d^{r}$-differential into the $E^{r}_{-1,\kappa}$-entry of an augmented spectral sequence in the next two sections.

\begin{prop}\label{move the big squares around}
Given $n$, $w$, $h$, and $\kappa$, such that $\min\{w-1,h\}\ge \kappa+2$, $wh-n\ge \max\big\{(\kappa+1)(\kappa+2)+\left\lfloor\frac{\kappa}{3}\right\rfloor, h\kappa+2\big\}$,
and $0\le s+1, p+1\le h$,
let $(A, \mathcal{W})$ be a weighted set such that $\sum w_{a}=n$, $\sum(w_{a}-1)=k\le \kappa$, and $w_{a_{0}}=\cdots=w_{a_{s+1+p}}=1$.
Let $\textbf{c}$ be the big square configuration in $SF_{(A, \mathcal{W})}(R_{w, h}; a_{0},\dots, a_{s};a_{s+1},\dots, a_{s+1+p})$ described in Construction \ref{first construction for connectivity}.
Moreover, assume that all of the free squares are of size at least $s+1$, where $k+s\le \kappa$.

If we replace the squares $a_{0},\dots, a_{s}$ with an $(s+1)\times (s+1)$ square $a'$ labeled $a_{0}\cdots a_{s}$, and free it, then the resulting configuration is in the same path-component as $\textbf{c'}\in SF_{(A', \mathcal{W}')}(R_{w, h};a_{s+1},\dots, a_{s+1+p})$, where $\textbf{c'}$ is the configuration described in Construction \ref{first construction for connectivity}, and $(A', \mathcal{W}')$ arises from $(A, \mathcal{W})$ by combining $a_{0}, \dots, a_{s}$ into a single $(s+1)\times (s+1)$ square labeled $a_{0}\cdots a_{s}$.
\end{prop}

\begin{proof}
If $s+1=1$, then the proof is nearly identical to that of Proposition \ref{free aj in k=1}, noting that we have ample space to rearrange the free unit squares.
As such, we omit it for brevity.

If $s+1>1$, note that in the configurations $\textbf{c}$ and $\textbf{c'}$ all of the squares of size $s+1$ can be found at the right of the rectangle defined by the convex hull of the $a_{s+1+p+1},\dots, a_{s+1+p+1+m}$.
We consider four cases depending on $m$.

If $m=0$, then the path is immediate as we can just slide the square $a'$ to the left.

If $m=1$, then $s+1+w_{a_{s+1+p+1}}\le h$.
If $w_{a_{s+1+p+1}}>s+1$, or $w_{a_{s+1+p+1}}=s+1$ and the smallest label of $a_{s+1+p+1}$ is bigger than $a_{0},\dots$, or $a_{s}$, then we can slide $a'$ to the left and then down to be below $a_{s+1+p+1}$. 
Similarly, if $m=1$ and $w_{a_{s+1+p+1}}=s+1$ and the smallest label of $a_{s+1+p+1}$ is smaller than $a_{0},\dots$, or $a_{s}$, then we can slide $a_{s+1+p+1}$ down and square $a'$ to the left to get to $\textbf{c'}$.

If $m=2$ and $w_{a_{s+1+p+2}}>s+1$, or $w_{a_{s+1+p+2}}=s+1$ and the smallest label of $a_{s+1+p+2}$ is bigger than the smallest label of $a'$, then we can simply slide $a'$ to the left.
Otherwise, $w_{a_{s+1+p+2}}=s+1$ and the smallest label of $a_{s+1+p+2}$ is smaller than the smallest label of $a'$.
In this case, we can slide square $a_{s+1+p+2}$ underneath the squares $a_{s+1}, \dots, a_{s+1+p}$ since $wh-n\ge (\kappa+1)(\kappa+2)+\left\lfloor\frac{k}{3}\right\rfloor$.
Doing this allows one to slide the remaining two squares into position noting that $w_{a_{s+1+p+1}}+w_{a_{s+1+p+2}}+s+1\le w$ by assumption.
After doing so, we get the configuration $\textbf{c'}$.

Finally, if $m\ge 3$, then note that if $w_{a_{s+1+p+m}}>s+1$ or $w_{a_{s+1+p+m}}=s+1$ and  the smallest label of $a_{s+1+p+2}$ is bigger than the smallest label of $a'$, then we can simply slide $a'$ to the left to get the configuration $\textbf{c'}$.
Otherwise, note that since $s\ge 1$, we have that $h\ge3(s+1)$ and the distance from the left side of the left-most square of size $s+1$ in $\textbf{c}$ to the right side of $R_{w,h}$ is also at least $3(s+1)$---if $p\neq -1$, it is at least $3(s+1)+1$.
This means that in the part of $R_{w,h}$ to the left of the left side of the left-most square of size $s+1$, there is enough space to rearrange the squares of size $s+1$ in a manner similar to Proposition \ref{connected single right configuration space} to get configuration $\textbf{c}'$, since $wh-n\ge(\kappa+1)(\kappa+2)$.
\end{proof}

One can check that the proof of Proposition \ref{move the big squares around} holds if we replace $wh-n\ge \max\big\{(\kappa+1)(\kappa+2)+\left\lfloor\frac{\kappa}{3}\right\rfloor, h\kappa+2\big\}$
with $wh-n\ge \max\big\{(\kappa+1)(\kappa+2), h\kappa+2\big\}$
in all cases but one.
Namely, if $w=\kappa+3$, $h=\kappa+2$, $\sum_{a\in A'}w_{a}=n$, $\sum_{a\in A;}(w_{a}-1)=\kappa$, $s+1>1$, $m=2$, and either $w_{a_{s+1+p+1}}>w_{a'}$ or $w_{a_{s+1+p+1}}=w_{a_{s+1+p+2}}=w_{a'}$ and the smallest label in $a_{s+1+p+1}$ is bigger than the smallest label in $a'$, which is bigger than the smallest label in $a_{s+1+p+2}$, then $wh-n\ge (\kappa+1)(\kappa+2)$ does not suffice.
Fortunately, in this paper, big square configuration spaces are merely a means to an end.
Recall that big squares represent the wheels $W(i_{1},\dots, i_{m})$, and if we allow ourselves to homotope between the big square $a'$ and big pyramid construction of same wheel as described in Remark \ref{pyramids are cool too}, one can slide it into the bottom right corner of $R_{w,h}$, and then rearrange $a_{p+1+s+1}$ and $a_{p+1+s+2}$ as necessary. 
Once these squares are in place, there is enough space to revert the pyramid to the big square $a'$ and slide it into the appropriate position; see Figure \ref{3bigwheelsannoying} for a visualization of this process.
Thus, we get the following plastic version of Proposition \ref{move the big squares around}.

\begin{figure}[h]
\centering
\captionsetup{width=.8\linewidth}
\includegraphics[width=16cm]{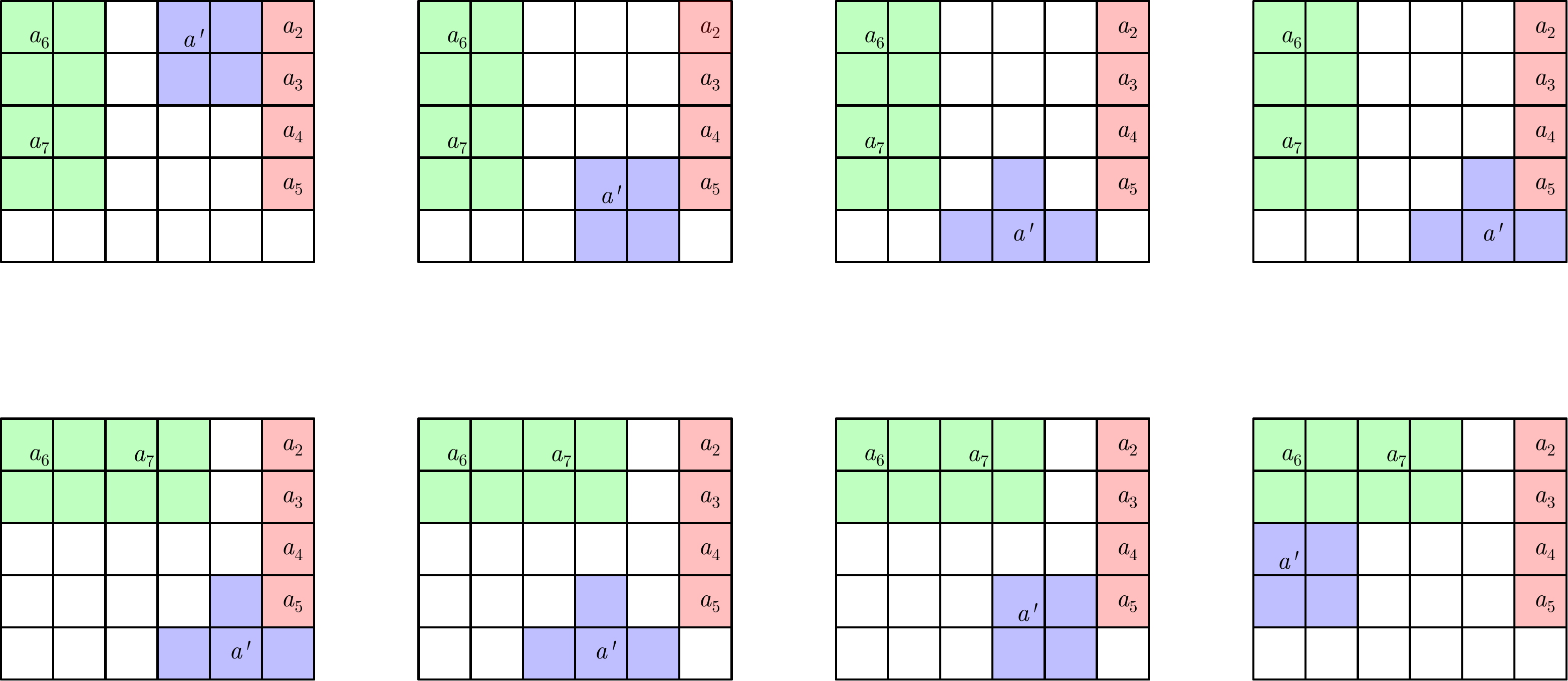}
\caption{The path (read left to right, top to bottom) that connects the big square configuration at the top left to the configuration $\textbf{c'}$ at the bottom right, which only exists if we allow deformations of squares into pyramids.
The green squares are free, the blue square $a'$ arises from combining $a_{0}$ and $a_{1}$ in the configuration $\textbf{c}$, and the red squares must remain right-most.
}
\label{3bigwheelsannoying}
\end{figure}

\begin{prop}\label{move the big stretchy squares around}
Given $n$, $w$, $h$, and $\kappa$, where $\min\{w-1,h\}\ge \kappa+2$, $wh-n\ge \max\big\{(\kappa+1)(\kappa+2), h\kappa+2\big\}$,
and $0\le s+1, p+1\le h$,
let $(A, \mathcal{W})$ be a weighted set such that $\sum w_{a}=n$, $\sum(w_{a}-1)=k\le \kappa$, and $w_{a_{0}}=\cdots=w_{a_{s+1+p}}=1$.
Let $\textbf{c}$ be the big square configuration in $SF_{(A, \mathcal{W})}(R_{w, h}; a_{0},\dots, a_{s};a_{s+1},\dots, a_{s+1+p})$ described in Construction \ref{first construction for connectivity}.
Moreover, assume that all of the free squares are of size at least $s+1$, where $k+s\le \kappa$.

If we replace the squares $a_{0},\dots, a_{s}$ into an $(s+1)\times (s+1)$ square $a'$ labeled $a_{0}\cdots a_{s}$, and free this square, then, if we have enough space to homotope the square $a'$ into a square pyramid of the same volume, and we allow for such homotopies, the resulting configuration is in the same component of the big square and pyramid configuration space as the configuration $\textbf{c'}$ in $SF_{(A', \mathcal{W}')}(R_{w, h};a_{s+1},\dots, a_{s+1+p})$, where $(A', \mathcal{W}')$ arises from $(A, \mathcal{W})$ by combining $a_{0}, \dots, a_{s}$ into a single $(s+1)\times (s+1)$ square labeled $a_{0}\cdots a_{s}$.
\end{prop}

In Section \ref{section main theorem} we will use the configurations described in Construction \ref{first construction for connectivity} and Proposition \ref{move the big stretchy squares around} to see that the differentials of a spectral sequence computing $H_{k}\big(SF_{n}(R_{w,h}; i_{0}, \dots, i_{p})\big)$ agree with the differentials of a spectral sequence computing $H_{k}\big(F_{n}(\R^{2}; i_{0}, \dots, i_{p})\big)$.
In the next section, we recall this spectral sequence.

\section{Our Mayer--Vietoris Spectral Sequences}\label{augmented MV}

Our proofs of Theorems \ref{space homological stability} and \ref{almost sharp homological stability for k=1} rely on a pair of augmented Mayer--Vietoris spectral sequences.
In this section, we recall the definition of the augmented Mayer--Vietoris spectral sequence, which arises from taking a cover of a space, and determine its $E^{1}$-page.
Next, we specialize to the spaces $SF_{n}(R_{w,h}; i_{0}, \dots, i_{p})$ and $F_{n}(\R^{2}; i_{0}, \dots, i_{p})$, covering them with sets of the form $SF_{n}(R_{w,h}; j_{0} ; i_{0}, \dots, i_{p})$ and $F_{n}(\R^{2}; j_{0};i_{0}, \dots, i_{p})$, respectively.
In the case of $F_{n}(\R^{2})$, we see in Lemma \ref{MV is Arc} that this augmented Mayer--Vietoris spectral sequence is equivalent to the augmented spectral sequence arising from the arc complex $\text{Arc}_{\bullet}\big(F_{n}(\R^{2})\big)$.
Finally, we use this equivalence to prove Proposition \ref{differentials are well defined as long as we have enough space},  calculating some of the differentials in our Mayer--Vietoris spectral sequence for $SF_{n}(R_{w,h})$ provided we have enough space in $R_{w,h}$.
In Section \ref{section main theorem}, we use our knowledge of these differentials to prove Theorem \ref{space homological stability}.

\subsection{The Mayer--Vietoris Spectral Sequence: An Overview}\label{defn of mv ss section}
This subsection contains an overview of the augmented Mayer--Vietoris spectral sequence; those familiar with this spectral sequence are encouraged to skip ahead to Subsection \ref{our mv ss}.
For a more detailed introduction, see, for example, \cite[Chapter VII.4]{brown2012cohomology}.

Given a  CW-complex $X$, let $\mathcal{U}=\{U_{i}\}_{i\in I}$ be a cover of $X$ by subcomplexes.
Let $K$ be the \emph{nerve complex} of $\mathcal{U}$, i.e., the abstract simplicial complex whose vertex set is $I$, and whose $p$-simplices are the subsets $\sigma\subset I$ of cardinality $p+1$, such that
\[
U_{\sigma}:=\bigcap_{i\in \sigma}U_{i}
\]
is non-empty.
For $p\ge 0$, let $C_{p}$ be the chain complex
\[
C_{p}:=\bigoplus_{\sigma\in K^{(p)}}C(U_{\sigma}),
\]
where $K^{(p)}$ is the set of $p$-simplices of $K$, and $C(U_{\sigma})$ is the chain complex on $U_{\sigma}$ arising from restricting the chain complex
\[
\cdots\xrightarrow{\partial} C_{q+1}(X)\xrightarrow{\partial} C_{q}(X)\xrightarrow{\partial} C_{q-1}(X)\xrightarrow{\partial} \cdots\xrightarrow{\partial}C_{1}(X)\xrightarrow{\partial} C_{0}(X)\to 0
\]
to $U_{\sigma}$.

Given a simplex $\sigma=[i_{0}, \dots, i_{p}]\in K^{(p)}$, set
\[
d_{j}\sigma=[i_{0},\dots, i_{j-1}, \widehat{i_{j}},i_{j+1},\dots, i_{p}].
\]
The inclusions $C(U_{\sigma})\hookrightarrow C(U_{d_{j}\sigma})$ induce chain maps
\[
d_{j}:C_{p}\to C_{p-1},
\]
which yield a boundary map
\[
d:=\sum^{p}_{j=0}(-1)^{j}d_{j}
\]
for the augmented chain complex
\[
\cdots\xrightarrow{d} C_{p}\xrightarrow{d}C_{p-1}\xrightarrow{d}\cdots \xrightarrow{d}C_{0}\xrightarrow{d}C(X)\to 0.
\]
This gives us a chain complex in the category of chain complexes, i.e., a \emph{double complex}
\[
C_{p,q}=\bigoplus_{\sigma\in K^{(p)}}C_{q}(U_{\sigma}).
\]

Given this double complex, we get the \emph{total complex} $TC$, where
\[
(TC)_{n}:=\bigoplus_{p+q=n}C_{p,q},
\]
which has differential $D:(TC)_{n}\to (TC)_{n-1}$ given by
\[
D:=d+(-1)^{p}\partial.
\]
Filtering $TC$ by setting
\[
F_{p}(TC)_{n}:=\bigoplus_{j\le p}C_{j, n-j},
\]
yields the \emph{augmented Mayer--Vietoris spectral sequence} associated to the cover $\mathcal{U}$ of $X$, which we denote by $E[X]$.

The $E^{0}$- and $E^{1}$-pages of this spectral sequence are easily computable from this definition; moreover, as we are considering the augmented spectral sequence, its $E^{\infty}$-page is $0$.

\begin{prop}\label{E0E1EinftyMV}
Let $I$ be a countable set and let $\mathcal{U}=\{U_{i}\}_{i\in I}$ be a cellular cover of a CW-complex $X$.
The $E^{0}$-page of the resulting augmented Mayer--Vietoris spectral sequence is
\[
E^{0}_{p,q}[X]=C_{p,q}=\bigoplus_{\sigma\in K^{(p)}}C_{q}(U_{\sigma}).
\]
The $E^{1}$-page is
\[
E^{1}_{p,q}[X]=\bigoplus_{\sigma\in K^{(p)}} H_{q}(U_{\sigma}).
\]
Moreover, for $r> q+1$
\[
E^{r}_{-1,q}[X]=0, 
\]
and for $p\ge 0$ and $r>p+q+1$
\[
E^{r}_{p,q}[X]=0.
\]
\end{prop}

Next, we specialize to the cases $X=SF_{n}(R_{w,h}; i_{0}, \dots, i_{p})$ and $X=F_{n}(\R^{2}; i_{0}, \dots, i_{p})$.
By finding similar covers for these spaces, we will show that their $E^{1}_{p,q}$-entries are isomorphic in a range, and that their differentials can be identified in this range, which will allow us to prove Theorem \ref{space homological stability}.

\subsection{Covering $SF_{n}(R_{w,h})$ and $F_{n}(\R^{2})$}\label{our mv ss}

Our augmented Mayer--Vietoris spectral sequences for $SF_{n}(R_{w,h};i_{0}, \dots, i_{p})$ and $F_{n}(\R^{2}; i_{0}, \dots, i_{p})$ arise from the observation that $R_{w,h}$ and $\R^{2}$ have well-defined notions of right, so we can consider the right-most square in $R_{w,h}$, resp., right-most point in $\R^{2}$.
This yields augmented Mayer--Vietoris spectral sequences, whose $E^{1}_{-1,q}$-entries are $H_{q}\big(SF_{n}(R_{w,h}; i_{0}, \dots, i_{p})\big)$ and $H_{q}\big(F_{n}(\R^{2}; i_{0}, \dots, i_{p})\big)$, respectively.

Namely, given the square configuration space $SF_{n}(R_{w,h}; i_{0}, \dots, i_{p})$, cover it by $SF_{n}(R_{w,h}; j;i_{0}, \dots, i_{p})$, where $j\in \{1,\dots, n\}-\{i_{0}, \dots, i_{p}\}$.
Proposition \ref{move squares to right side} allows us to assume that the squares $i_{0}, \dots, i_{p}$ are tangent to the right side of $R_{w,h}$; after subdividing $DF_{n}(R^{*}_{w,h})$ by hyperplanes,  we get a cellular model for $SF_{n}(R_{w,h}; i_{0}, \dots, i_{p})$.
Further subdividing this cell-complex, we can view the $SF_{n}(R_{w,h}; j;i_{0}, \dots, i_{p})$ as subcomplexes, which yields 
an augmented Mayer--Vietoris spectral sequence $E\big[SF_{n}(R_{w,h};  i_{0}, \dots, i_{p})\big]$.

Given ordered sets $(j_{0},\dots, j_{m})$ and $(j'_{0}, \dots, j'_{l})$ whose entries are disjoint, Proposition \ref{intersection is shuffle} extends to give that if $m+1+l+1\le h$, then
\begin{multline*}
SF_{n}(R_{w,h};j_{0}, \dots, j_{m}; i_{0}, \dots, i_{p})\cap SF_{n}(R_{w,h};j'_{0}, \dots, j'_{l};i_{0}, \dots, i_{p})\\=\bigsqcup_{\sigma\in \Sigma\big((j_{0}, \dots, j_{m}),(j'_{0}, \dots, j'_{l})\big)}SF_{n}(R_{w,h}; \sigma;i_{0}, \dots, i_{p}).
\end{multline*}
By Proposition \ref{E0E1EinftyMV}, we have that 
\begin{align*}
E^{1}_{s,t}\big[SF_{n}(R_{w,h};  i_{0}, \dots, i_{p})\big]&=\bigoplus_{\substack{(j_{0}, \dots, j_{s})|j_{l}\in\{1,\dots,n\}\\j_{l}\neq j_{m}\text{ for }l\neq m}}H_{t}\big(SF_{n}(R_{w,h};  j_{0}, \dots, j_{s};i_{0}, \dots, i_{p})\big)\\
&\cong\bigoplus_{\tau\in S_{s+1}}\bigoplus_{1\le j_{0}<\cdots<j_{s}\le n}H_{t}\big(SF_{n}(R_{w,h};  j_{\tau(0)}, \dots, j_{\tau(s)};i_{0}, \dots, i_{p})\big);
\end{align*}
in particular
\[
E^{1}_{s,t}\big[SF_{n}(R_{w,h})\big]\cong\bigoplus_{\tau\in S_{s+1}}\bigoplus_{1\le j_{0}<\cdots<j_{s}\le n}H_{t}\big(SF_{n}(R_{w,h};  j_{\tau(0)}, \dots, j_{\tau(s)})\big).
\]

It follows that if the inclusion of $SF_{n}(R_{w,h};  i_{0}, \dots, i_{p})$ into $F_{n}(\R^{2};  i_{0}, \dots, i_{p})$ induces an isomorphism
\[
H_{q}\big(SF_{n}(R_{w,h};  i_{0}, \dots, i_{p})\big)\cong H_{q}\big(F_{n}(\R^{2};  i_{0}, \dots, i_{p})\big),
\]
then we would have 
\[
E^{1}_{p,q}\big[SF_{n}(R_{w,h})\big]
\cong E^{1}_{p,q}\big[F_{n}(\R^{2})\big].
\]
If this held for enough $p$ and $q$, and we could also show that the $d^{r}$-differentials of this pair of spectral sequences are the same, i.e., 
\[\begin{tikzcd}
	{E^{r}_{p,q}\big[SF_{n}(R_{w,h})\big]} && {E^{r}_{p-r,q+r-1}\big[SF_{n}(R_{w,h})\big]} \\
	\\
	{E^{r}_{p,q}\big[F_{n}(\R^{2})\big]} && {E^{r}_{p-r,q+r-1}\big[F_{n}(\R^{2})\big]}
	\arrow["{d^{r}_{SF_{n}(R_{w,h})}}", from=1-1, to=1-3]
	\arrow["\cong"', from=1-1, to=3-1]
	\arrow["\cong", from=1-3, to=3-3]
	\arrow["{d^{r}_{F_{n}(\R^{2})}}", from=3-1, to=3-3]
\end{tikzcd}\]
then it would follow that
\[
H_{k}\big(SF_{n}(R_{w,h})\big)\cong H_{k}\big(F_{n}(\R^{2})\big).
\]
To calculate these differentials, we take a quick detour into the complex of injective words.

\subsection{The Complex of Injective Words}
We recall a few facts about a simplicial complex introduced by Farmer \cite{farmer1978cellular} called the complex of injective words that will greatly simplify the computation of the differentials of our Mayer--Vietoris spectral sequences. 

If $A=i_{0}\cdots i_{p}$ is a word in the alphabet $\{1,\dots, n\}$, then we write $|A|:=p+1$ for its length, and we say that it is \emph{injective} if each letter in the alphabet appears at most once.
One defines the \emph{complex of injective words} $\text{Inj}_{\bullet}(n)$ by setting
\[
\text{Inj}_{p}(n):=\text{Hom}_{\text{FI}}\big(\{0,\dots, p\},\{1,\dots, n\}\big),
\]
where FI is the category of finite sets and injections.
It follows that $\text{Inj}_{\bullet}(n)$ has the structure of an augmented semi-simplicial set, where the $p$-simplices are the injective words of length $p+1$, and the face maps $d_{j}$ are defined by
\[
d_{j}(i_{0}\cdots i_{p})=i_{0}\cdots i_{j-1}\widehat{i_{j}}i_{j+1}\cdots i_{p}.
\]
Writing $C^{(n)}_{p}$ for the free abelian group on words on $p+1$ letters in $\{1,\dots, n\}$, i.e., $C^{(n)}_{p}=C_{p}\big(\text{Inj}_{\bullet}(n)\big)$, this gives rise to a boundary map $d:C^{(n)}_{p}\to C^{(n)}_{p-1}$ defined by 
\[
d:=\sum_{j=0}^{p}(-1)^{j}d_{j}.
\]

The parallels between this boundary map and the differential $d$ described in Subsection \ref{defn of mv ss section} are immediate and are further deepened by noting that the summands of
\[
E^{1}_{p,q}\big[SF_{n}(R_{w,h})\big]=\bigoplus_{\tau\in S_{p+1}}\bigoplus_{1\le i_{0}<\cdots<i_{p}\le n}H_{q}\big(SF_{n}(R_{w,h};  i_{\tau(0)}, \dots, i_{\tau(p)})\big)
\]
are indexed by the injective words on $p+1$ letters in the alphabet $\{1,\dots, n\}$.
In fact, if there is ``enough space'' in $SF_{n}(R_{w,h})$, these boundary maps are the same, an idea we will formalize shortly.
To take advantage of this, we seek a tractable basis for $C^{(n)}_{p}$, as the naive basis that consists of all injective words of length $p+1$ does not readily lend itself to computations.

The elements of this superior basis will be Lie polynomials.
Recall that $P$, an integer linear combination of words $A$, is a \emph{polynomial} in $\{1,\dots, n\}$, and we write $|P|$ for the length of the longest word in $P$.
For words $A$ and $B$, we write $AB$ to denote their concatenation, and we extend this linearly to polynomials.
This allows us to define a Lie bracket on words by setting
\[
[A,B]:=AB-(-1)^{|A||B|}BA,
\]
which can be extended bilinearly to a Lie bracket on the free abelian group of words in $\{1,\dots, n\}$.
One can check that if $A$, $B$, and $C$ are homogeneous polynomials, this Lie bracket satisfies the graded antisymmetry rule
\[
[A,B]+(-1)^{|A||B|}[B,A]=0
\]
as well as the graded Jacobi identity
\[
(-1)^{|A||C|}\big[A,[B,C]\big]+(-1)^{|A||B|}\big[B,[C,A]\big]+(-1)^{|B||C|}\big[C,\big[A,B]\big]=0.
\]
A \emph{Lie polynomial} is any element of the smallest submodule of the free abelian group on words in $\{1, \dots, n\}$ that contains the elements of $\{1,\dots, n\}$ and is closed under the Lie bracket.
We focus on a set of particularly nice Lie polynomials, the homogeneous Lie polynomials.

For a finite set $\{1,\dots, n\}$, where $n\ge 2$, let $\mathcal{L}_{n}$ denote the subset of homogeneous Lie polynomials of degree $n$ whose terms are all injective words in $\{1,\dots, n\}$.
If $n=1$ or $0$, we define $\mathcal{L}_{n}$ to be $\emptyset$.
Our calculations use a basis for $C^{(n)}_{p}$ that can be described in terms of the $\mathcal{L}_{m}$; as such, we seek a basis for $\mathcal{L}_{m}$.

\begin{prop}\label{Reutenauer basis}
(Miller--Wilson \cite[Theorem 2.33]{miller2019higher})
The abelian group $\mathcal{L}_{m+1}$ is free of rank $m!$ with a basis consisting of all elements of the form
\[
\Bigg[\bigg[\Big[\cdots\big[[1,i_{1}], i_{2}\big],\cdots\Big], i_{m-1}\bigg], i_{m}\Bigg] \quad \text{ for any ordering }(i_{1}, \dots, i_{m})\text{ of }\{2,\dots, m+1\}.
\]
\end{prop}

This is \emph{Reutenauer's basis} for $\mathcal{L}_{m+1}$, and for  general $S\subset \{1,\dots, n\}$, we define the Reutenauer basis for $\mathcal{L}_{S}$ to be the $\big(|S|-1\big)!$ elements as above with the letter $1$ replaced by the smallest element of $S$.
This allows one to define a basis for $C^{(p+1)}_{p}$, which we will use to define a basis for $C^{(n)}_{p}$.

\begin{prop}\label{basis for inj}
(Miller--Wilson \cite[Theorem 2.35]{miller2019higher})
Fix $p+1\ge 2$. 
For each subset $S\subset \{0, \dots, p\}$ with $|S|\ge 2$, let $B_{S}$ be the basis of Proposition \ref{Reutenauer basis} for $\mathcal{L}_{S}$.
For each singleton subset $S=\{i\}\subset \{0, \dots, p\}$, let $B_{S}=\{i\}$.
Put a total order $\le$ on $B=\sqcup_{S\subset \{0,\dots, p\}}B_{S}$.
Then the set $\Pi^{p+1}_{p}(1)$ of Lie polynomials of the form
\[
P_{1}P_{2}\cdots P_{m}\quad \text{such that}\quad \{0,\dots, p\}=S_{1}\sqcup S_{2}\sqcup \cdots\sqcup S_{m}, P_{i}\in B_{S_{i}},\quad \text{and}\quad P_{1}<P_{2}<\cdots< P_{m}\in B
\]
is a $\Z$-basis for $C^{(p+1)}_{p}$.
\end{prop}

In particular, if we replace $\{0,\dots, p\}$ with any set $T=\{i_{0},\dots, i_{p}\}$, we get a $\Z$-basis $\Pi^{T}_{p}$ for $C^{T}_{p}$.
Taking every possible $\Sigma\subset\{1, \dots, n\}$ of size $p+1$, we get a basis for $C^{(n)}_{p}$. 

\begin{prop}\label{basis for Cnp}
Given a set $T$ of size $p+1$, let $\Pi^{T}_{p}$ denote the $\Z$-basis for $C^{T}_{p}$.
Then
\[
\Pi^{n}_{p}:=\bigsqcup_{T\subset\{1,\dots, n\}\big||T|=p+1}\Pi^{T}_{p}
\]
is a $\Z$-basis for $C^{(n)}_{p+1}$.
\end{prop}

Farmer noted that $\|\text{Inj}_{\bullet}(n)\|$ has reduced homology concentrated in the top dimension \cite{farmer1978cellular}; let $\mathcal{T}_{p}:=H_{p-1}\left(C^{(p)}_{*}\right)$ denote this  homology group.
Miller and Wilson found a basis for $\mathcal{T}_{p}$ that we will use to define a basis for the $E^{1}$-entries of our augmented Mayer--Vietoris spectral sequences \cite{miller2019higher}.
Later, we will see that there is a nice subset of this basis that describes the $E^{r}$-entries of our augmented Mayer--Vietoris spectral sequences.

\begin{prop}\label{basis for top homology of inj word}
(Miller--Wilson \cite[Lemma 2.38 and Theorem 2.40]{miller2019higher})
Fix a finite $p+1\ge 2$. 
$\mathcal{T}_{p}$ is equal to the subgroup of $C^{(p)}_{p-1}$ spanned by $\mathcal{L}$-products. 
Moreover, it has the following basis:
As in Proposition \ref{basis for inj}, for each subset $S\subseteq[p+1]$, let $B_{S}$ be the basis for $\mathcal{L}_{S}$ of Proposition \ref{Reutenauer basis}.
Put a total order $\le$ on $B:=\bigcup_{S\subset [p+1], |S|\ge 2}B_{S}$.
Then, the set $\Pi^{*}_{p}$ of polynomials
\[
\big\{P_{1}P_{2}\cdots P_{m}\big|[p+1]=S_{1}\sqcup S_{2}\sqcup \cdots \sqcup S_{m}, P_{i}\in B_{S_{i}}, \text{ and } P_{1}<P_{2}<\cdots< P_{m}\in B\big\}
\]
is a basis for $\mathcal{T}_{p}$.
\end{prop}

Note that all of the Lie polynomials factors of a basis element for $\mathcal{T}_{p}$ have degree at least $2$.
When we calculate the $E^{r}$-entries of our spectral sequences we will need the subspace $\mathcal{T}_{p}^{r}$ which we define to be the subspace $\mathcal{T}_{p}$ spanned by the basis elements of Proposition \ref{basis for top homology of inj word} that are products of polynomials of degree at least $r$.
We write $\Pi^{*}_{p}(r)$ for this set.

\begin{exam}

The basis $\Pi^{4}_{2}$ for $C^{(4)}_{2}$ described by Proposition \ref{basis for Cnp} is
\begin{multline*}
\big[[1,2],3\big], \big[[1,3],2\big], 1[2,3], 2[1,3], 3[1,2], 123, \big[[1,2], 4\big], \big[[1,4],2\big], 1[2,4], 2[1,4], 4[1,2], 124,\\
 \big[[1,3],4\big], \big[[1,4],3\big],1[3,4], 3[1,4], 4[1,3], 134, \big[[2,3],4\big], \big[[2,4],3\big],2[3,4], 3[2,4], 4[2,3], 234.
\end{multline*}

The basis $\Pi^{4}_{3}$ for $C^{(4)}_{3}$ described by Proposition \ref{basis for inj} is 
\begin{multline*}
1234, 12[3,4], 13[2,4], 14[2,3], 23[1,4], 24[1,3], 34[1,2], 1\big[[2,3],4\big], 1\big[[2,4],3\big], 2\big[[1,3],4\big], 2\big[[1,4],3\big], \\ 3\big[[1,2],4\big], 3\big[[1,4],2\big],
4\big[[1,2],3\big], 4\big[[1,3],2\big],  [1,2][3,4], [1,3][2,4], [1,4][2,3], \Big[\big[[1,2],3\big],4\Big], \Big[\big[[1,2],4\big],3\Big],\\\Big[\big[[1,3],2\big],4\Big],
\Big[\big[[1,3],4\big],2\Big],
\Big[\big[[1,4],2\big],3\Big],
\Big[\big[[1,4],3\big],2\Big].
\end{multline*}

The basis $\Pi^{*}_{3}=\Pi^{*}_{3}(2)$ for $\mathcal{T}_{4}=\mathcal{T}^{2}_{4}$ described by Proposition \ref{basis for top homology of inj word} is
\begin{multline*}
[1,2][3,4], [1,3][2,4], [1,4][2,3], \Big[\big[[1,2],3\big],4\Big], \Big[\big[[1,2],4\big],3\Big],\\\Big[\big[[1,3],2\big],4\Big],
\Big[\big[[1,3],4\big],2\Big],
\Big[\big[[1,4],2\big],3\Big],
\Big[\big[[1,4],3\big],2\Big].
\end{multline*}

The basis $\Pi_{3}^{*}(3)$ for $\mathcal{T}_{4}^{3}$ described by Proposition \ref{basis for top homology of inj word} is
\[
\Big[\big[[1,2],3\big],4\Big], \Big[\big[[1,2],4\big],3\Big],\\\Big[\big[[1,3],2\big],4\Big],
\Big[\big[[1,3],4\big],2\Big],
\Big[\big[[1,4],2\big],3\Big],
\Big[\big[[1,4],3\big],2\Big].
\]
\end{exam}

While the bases of Proposition \ref{basis for Cnp} might seem superfluous at first, they will prove to be highly useful for calculating differentials in the Mayer--Vietoris spectral sequence.
In the next subsection, we explore this idea in more depth by connecting our Mayer--Vietoris spectral sequences to the arc resolution spectral sequence.

\subsection{The Arc Resolution Spectral Sequence}

We recall the definitions of the arc complex and arc resolution spectral sequence, and prove that for $F_{n}(\R^{2};i_{0}, \dots, i_{p})$, this spectral sequence is equivalent to the augmented Mayer--Vietoris spectral sequence $E\big[F_{n}(\R^{2};i_{0}, \dots, i_{p})\big]$ defined in Subsection \ref{our mv ss}.

\begin{defn}
Let $M$ be the interior of a manifold $\overline{M}$ with nonempty boundary $\partial M$.
Fix an embedding $\gamma:[0,1]\to \partial M$.
Let
\[
\text{Arc}_{p}\big(F_{n}(M)\big)\subset F_{n}(M)\times \text{Emb}\big(\sqcup_{p+1}[0,1], \overline{M}\big)
\]
be the subspace of points and arcs $(x_{1},\dots, x_{n};\alpha_{0}, \dots, \alpha_{p})$ satisfying the following conditions:
\begin{itemize}
    \item $\alpha_{i}(0)\in\gamma\big([0,1]\big)$,
    \item $\alpha_{i}(1)\in\{x_{1},\dots, x_{n}\}$,
    \item $\alpha_{i}(t)\notin \partial M\cup \{x_{0}, \dots, x_{n}\}$ for $t\in (0,1)$, and 
    \item $\gamma^{-1}\big(\alpha_{j_{1}}(0)\big)>\gamma^{-1}\big(\alpha_{j_{2}}(0)\big)$ whenever $j_{1}>j_{2}$.
\end{itemize}
\end{defn}

In the case of $F_{n}(\R^{2})$, we realize $\R^{2}$ as the interior of the closed rectangle $\overline{R_{n,n}}$ and set $\gamma\big([0,1]\big)$ to be the right side of the rectangle parametrized by the constant speed path that starts at the top corner.
For the case of $F_{n}(\R^{2}; i_{0}, \dots, i_{p})$, we use Proposition \ref{forget points on the right} to define $\text{Arc}_{s}\big(F_{n}(\R^{2}; i_{0}, \dots, i_{p})\big)$ to be $\text{Arc}_{s}\big(F_{\{1,\dots, n\}-\{i_{0},\dots, i_{p}\}}(\R^{2})\big)$. 

Letting $p$ vary, one gets an augmented semi-simplicial space $\text{Arc}_{\bullet}\big(F_{n}(M)\big)$, where the $j$-th face map
\[
d_{j}:\text{Arc}_{p}\big(F_{n}(M)\big)\to \text{Arc}_{p-1}\big(F_{n}(M)\big)
\]
is given by forgetting the $j$-th arc $\alpha_{j}$.
Noting that the space $\text{Arc}_{-1}\big(F_{n}(M)\big)$ is homeomorphic to $F_{n}(M)$, one sees that there is an augmentation map
\[
\big\|\text{Arc}_{\bullet}\big(F_{n}(M)\big)\big\|\to F_{n}(M).
\]
This gives rise to an augmented spectral sequence called the \emph{arc resolution spectral sequence}, which we denote by $arcE\big[F_{n}(M)\big]$.

\begin{prop}\label{differential definition of arc}
(Bendersky--Gilter \cite[Prop 1.2]{bendersky1991cohomology} and Miller--Wilson \cite[Prop 3.13]{miller2019higher})
Let $A_{\bullet}$ be an augmented semi-simplicial space.
Beginning on the $E^{1}$-pages, the geometric realization spectral sequence agrees with the spectral sequence for the double complex $C_{*}(A_{\bullet})$ that has $d^{0}$ differential induced by $\partial$ and $d^{1}$ induced by the alternating sum of the face maps.
\end{prop}

In the case $M=\R^{2}$, the arc resolution spectral sequence is the same as the augmented Mayer--Vietoris spectral sequence of Subsection \ref{our mv ss}.

\begin{lem}\label{MV is Arc}
Recall that $E\big[F_{n}(\R^{2})\big]$ denotes the augmented Mayer--Vietoris spectral sequence arising from covering $F_{n}(\R^{2})$ by sets of the from $F_{n}(\R^{2};i)$, where $i\in \{1,\dots, n\}$.
There is an isomorphism of spectral sequences
\[
E\big[F_{n}(\R^{2})\big]\cong arcE\big[F_{n}(\R^{2})\big]
\]
such that the $d^{r}$-differential commute.
\end{lem}

\begin{proof}
This follows from the isomorphism between $\bigsqcup F_{n}(\R^{2}; i_{0}, \dots, i_{p})$ and $Arc_{p}\big(F_{n}(\R^{2})\big)$, noting that we can model $\R^{2}$ by $R_{n,n}$ and contract the arcs, giving us an isomorphism between the double complexes.

Inspecting the two pairs of differentials $d$ and $\partial$ defining $E\big[F_{n}(\R^{2})\big]$ and $arcE\big[F_{n}(\R^{2})\big]$, we see that they are naturally equivalent, which along with Proposition \ref{differential definition of arc} yields the desired isomorphism.
\end{proof}

Unfortunately, no such result holds for general square configuration spaces. 
Indeed, even finding a reasonable definition for the arc complex for $SF_{n}(R_{w,h})$ remains elusive.
Still, we will be able to use this equivalence of spectral sequences in the $F_{n}(i_{0}, \dots, i_{p})$ setting to calculate the differentials of augmented Mayer--Vietoris spectral sequence for $SF_{n}(R_{w,h}; i_{0}, \dots, i_{p})$.

\subsection{The $d^{r}$-Differentials of Our Augmented Mayer--Vietoris Spectral Sequences}

In this subsection, we calculate the differentials of our augmented Mayer--Vietoris spectral sequence provided there is ``enough space'' in $R_{w,h}$.
We first prove the differentials of our spectral sequence satisfy a Leibniz rule.
Then we focus on the bottom row of the spectral sequence in the $SF_{n}(R_{n, n})\simeq F_{n}(\R^{2})$ setting and determine the $d^{r}$-differentials, and we use these calculations and the Leibniz rule to calculate the $r$-th page of our Mayer--Vietoris spectral sequence for $F_{n}(\R^{2})$.
Finally, we use the Leibniz rule to compute the differentials for more general $SF_{n}(R_{w, h})$ provided we have ``enough space.''

We begin by noting that in general the differentials of our spectral sequences satisfy a Leibniz rule.
Namely, there is an inclusion
\[
\iota:R_{w_{1}, h_{1}}\sqcup R_{w_{2}, h_{2}}\hookrightarrow R_{w, h_{1}+h_{2}},
\]
where $w=\max\{w_{1}, w_{2}\}$, that sends $R_{w_{1}, h_{1}}$ to the top right of $R_{w, h_{1}+h_{2}}$ and $R_{w_{2}, h_{2}}$ to the bottom right of $R_{w, h_{1}+h_{2}}$.
This map induces an embedding
\[
\iota_{n_{1}, n_{2}}:SF_{n_{1}}(R_{w_{1}, h_{1}})\times SF_{n_{2}}(R_{w_{2}, h_{2}}\big)\to SF_{n_{1}+n_{2}}(R_{w, h_{1}+h_{2}}),
\]
which leads to a Leibniz rule on the differentials of $E\big[SF_{n_{1}+n_{2}}(R_{w, h_{1}+h_{2}})\big]$.

\begin{prop}\label{Leibniz Rule}
The maps $\iota_{n_{1}, n_{2}}$ induce maps
\[
\iota^{r}:E^{r}_{p_{1},q_{1}}\big[SF_{n_{1}}(R_{h_{1}, w_{1}})\big]\otimes E^{r}_{p_{2},q_{2}}\big[SF_{n_{2}}(R_{w_{2}, h_{2}})\big]\to E^{r}_{p_{1}+p_{2}+1, q_{1}+q_{2}}\big[SF_{n_{1}+n_{2}}(R_{w, h_{1}+h_{2}})\big].
\]
Moreover, these maps satisfy a Leibniz rule with respect to the differentials: Given $a\in E^{r}_{p_{1},q_{1}}\big[SF_{n_{1}}(R_{n_{1}, n_{1}})\big]$ and $b\in E^{r}_{p_{2},q_{2}}\big[SF_{n_{2}}(R_{n_{2}, n_{2}})\big]$, we have that
\[
d^{r}\big(\iota^{r}(a\otimes b)\big)=\iota^{r}\big(d^{r}(a)\otimes b\big)+(-1)^{p_{1}+q_{1}}\iota^{r}\big(a\otimes d^{r}(b)\big).
\]
\end{prop}

The proof of this proposition is identical to that of \cite[Lemma 3.15]{miller2019higher} if we instead consider the filtered double complexes defining our augmented Mayer--Vietoris spectral sequences in place of the filtered double complex defining their arc resolution spectral sequences, so we direct the reader there for more details.

Next, we focus on the $F_{n}(\R^{2})\simeq SF_{n}(R_{n,n})$ case, where computing the action of the differentials is straight-forward.
Namely, given any class $[\alpha]\in H_{q}\big(F_{n}(\R^{2};  i_{0}, \dots, i_{p})\big)$ and any $\tau\in S_{p+1}$, there is a class $\big[\tau(\alpha)\big]\in H_{q}\big(F_{n}(\R^{2};  i_{\tau(0)}, \dots, i_{\tau(p)})\big)$ that arises by applying $\tau$ to any chain representing $[\alpha]$.
There are $(p+1)!$ such permutations of $[\alpha]$, which are linearly independent in $\bigoplus_{\tau\in S_{p+1}}H_{q}\big(F_{n}(\R^{2};  i_{\tau(0)}, \dots, i_{\tau(p)})\big)$.
Since $F_{n}(\R^{2}; i_{0}, \dots, i_{p})\simeq F_{n-p-1}(\R^{2})$ by Proposition \ref{forget points on the right}, and $H_{q}\big(F_{n-p-1}(\R^{2})\big)$ is torsion free by Proposition \ref{basis for homology of FnR2}, the span of the $\big[\tau(\alpha)\big]$ is naturally isomorphic to $C^{(p+1)}_{p}$, and we can use this isomorphism and Proposition \ref{basis for inj} to describe a basis for this space in terms of Lie polynomials.
This basis can be extended to give a basis for $E^{1}_{p,q}\big[SF_{n}(R_{n,n})\big]$.
Namely, if $B$ is a basis for $H_{q}\big(F_{n-p-1}(\R^{2})\big)$, then $B\Pi^{n}_{p}$ is a basis for $\bigoplus_{\tau\in S_{p+1}}\bigoplus_{1\le i_{0}<\cdots<i_{p}\le n}H_{q}\big(SF_{n}(R_{n,n};  i_{\tau(0)}, \dots, i_{\tau(p)})\big)$.
In particular, if $p+1=n$, then, since $H_{0}\big(SF_{n}(R_{n,n};  i_{\tau(0)}, \dots, i_{\tau(n-1)})\big)\cong H_{0}\big(SF_{0}(R_{n,n})\big) \cong\Z$, we see that the $E^{1}_{n-1, 0}$-entry of our augmented Mayer--Vietoris spectral sequence for $SF_{n}(R_{n,n})$ is $\Z^{(n-1)!}$, and the basis $\Pi^{*}_{n-1}$ described in Proposition \ref{basis for top homology of inj word} works as a basis for this entry.
With this in mind, we calculate the images of the differentials on some of the elements of this basis.

\begin{lem}\label{differentials on Reut}
Let
\[
L_{n}=\Bigg[\bigg[\Big[\cdots\big[1,i_{1}],i_{2}\big],\cdots\Big],i_{n-2}\bigg], i_{n-1}\Bigg]
\]
be an element of Reutenauer's basis for $\mathcal{L}_{n}\subset \mathcal{T}_{n}\cong E^{1}_{n-1, 0}\big[SF_{n}(R_{n,n})\big]$.
Then $d^{r}(L_{n})=0$ for $r<n$, and 
\[
d^{n}(L_{n})=W(1, i_{1},\dots, i_{n-1})\in E^{n}_{-1, n-1}\big[SF_{n}(R_{n,n})\big].
\]
\end{lem}

\begin{proof}
This follows directly from the equivalence between $E[SF_{n}(R_{n, n})]$ and $E[F_{n}(\R^{2})]$ induced by Proposition \ref{homotopy equivalence}, the equivalence between $E[F_{n}(\R^{2})]$ and $arcE[F_{n}(\R^{2})]$ described in Lemma \ref{MV is Arc}, and Miller and Wilson's calculation of the corresponding differential \cite[Lemma 3.17]{miller2019higher}.
\end{proof}

Since the $d^{r}$-differential frees $r$ squares, it follows that if $r<n$, then $d^{r}(L_{n})=0$ holds in $SF_{n}(R_{r+1,n})$, since at most $r+1$ squares can be horizontally aligned in $d^{r}(L)$.

Having determined how the differentials act on elements of Reutenauer's basis for $\mathcal{L}_{n}$, we calculate the $E^{r}$-page of our augmented Mayer--Vietoris spectral sequence for $SF_{n}(R_{n,n})\simeq F_{n}(\R^{2})$.
Since we are comparing this spectral sequence to the augmented Mayer--Vietoris spectral sequence for $SF_{n}(R_{w,h})$ to determine conditions on $n$, $w$, $h$, and $k$ that guarantee that 
\[
H_{k}\big(SF_{n}(R_{w,h})\big)\cong H_{k}\big(F_{n}(\R^{2})\big),
\]
knowing the $r$-th page will greatly expedite the process.
Before we do these calculations, we recall the definitions of the twisted algebras corresponding to higher order representation stability in an effort to simplify the calculations.

Consider the twisted algebra
\[
\mathcal{A}_{r}:=\begin{cases}
\text{Sym} H_{r-1}\big(F_{r}(\R^{2})\big),&\text{for } r \text{ odd,}\\
\bigwedge H_{r-1}\big(F_{r}(\R^{2})\big),&\text{for } r \text{ even}.
\end{cases}
\]
These are the twisted algebras generated by the wheels $W(1,\dots, r)$ subject to the (anti)-commutation relations of Proposition \ref{little cube operad structure}.
Given a twisted algebra $\mathcal{A}$ and an $\mathcal{A}$-module $\mathcal{V}$, we write
\[
H^{\mathcal{A}}_{0}(\mathcal{V})_{S}:=\text{coker}\bigg(\bigoplus_{S=P\sqcup Q, P\neq \emptyset}\mathcal{A}_{P}\otimes \mathcal{V}_{Q}\to \mathcal{V}_{S}\bigg)
\]
for the \emph{zeroth $\mathcal{A}$-homology of $\mathcal{V}$}.
Specializing to $F_{\bullet}(\R^{2})$, we see that the maps
\[
H_{r-1}\big(F_{r}(\R^{2})\big)\otimes H_{k}\big(F_{n}(\R^{2})\big)\to H_{k+r-1}\big(F_{n+r}(\R^{2})\big) 
\]
corresponding to multiplying a class in $H_{k}\big(F_{n}(\R^{2})\big)$ by a wheel $W(i_{0}, \dots, i_{r-1})$ gives $H_{k+j(r-1)}\big(F_{\bullet}(\R^{2})\big)$ the structure of an $\mathcal{A}_{r}$-module. 
In particular, Proposition \ref{little cube operad structure} shows that this module is free, and that taking the zeroth $\mathcal{A}_{r}$ homology of $H_{k+j(r-1)}\big(F_{\bullet}(\R^{2})\big)$ has the effect of quotienting by all classes that contain a wheel on $r$ points as a factor.
Additionally, 
\[
H^{\mathcal{A}_{r}}\bigg(\cdots H^{\mathcal{A}_{1}}\Big(H_{k}\big(F(\R^{2})\big)\Big)\bigg)_{n}
\]
is well-defined; it is the subspace of $H_{k}\big(F_{n}(\R^{2})\big)$ spanned by products of wheels on at least $r+1$ points.
Moreover, the sub-basis of Proposition \ref{basis for homology of FnR2} consisting of products of wheels on at least $r+1$ points is a basis for this vector space.
For more on twisted algebras in general, see Sam and Snowden's introduction to the subject \cite{sam2012introduction}, and for more on the $\mathcal{A}_{r}$ in particular, see \cite{church2017homology, nagpal2019noetherianity, miller2019higher}.

Using the free $\mathcal{A}_{1}$-structure of $H_{k}\big(F_{\bullet}(M)\big)$, where $M$ is a connected non-compact finite type manifold of dimension at least $2$, Miller and Wilson computed the $E^{2}$-page of the arc resolution spectral sequence.

\begin{prop}\label{E2pageforFnR2}
(Miller--Wilson \cite[Proposition 3.10]{miller2019higher})
Let $M$ be a connected non-compact finite type manifold of dimension at least $2$, then the $E^{2}$-page of the augmented arc resolution spectral sequence satisfies
\[
arcE^{2}_{p,q}\big[F_{n}(M)\big]\cong \text{Ind}^{S_{n}}_{S_{n-p-1}\times S_{p+1}}H^{\mathcal{A}_{1}}_{0}\Big(H_{q}\big(F(M)\big)\Big)_{n-p-1}\boxtimes \mathcal{T}_{p+1}.
\]
\end{prop}

Miller and Wilson's proof does not explicitly use the calculation of $d^{1}$-differentials of this spectral sequence; instead it follows from a fact about the complex of injective words.
This does not readily lend itself to computing the $E^{r}$-page for $r\ge2$.
In the case $M=\R^{2}$, there is a free $\mathcal{A}_{r}$-module structure on $H_{*}\big(F_{\bullet}(\R^{2})\big)$, so Proposition \ref{Leibniz Rule} and Lemma \ref{differentials on Reut} allow us to compute the $r$-th page from our knowledge of the differentials.

\begin{prop}\label{Er page for FnR2}
For $r\ge 2$, the $E^{r}$-page of the augmented Mayer--Vietoris spectral sequence for $F_{n}(\R^{2})$ satisfies
\[
E^{r}_{p,q}\big[F_{n}(\R^{2})\big]\cong arcE^{r}_{p,q}\big[F_{n}(\R^{2})\big]\cong \text{Ind}^{S_{n}}_{S_{n-p-1}\times S_{p+1}}H^{\mathcal{A}_{r-1}}\bigg(\cdots H^{\mathcal{A}_{1}}\Big(H_{q}\big(F(\R^{2})\big)\Big)\bigg)_{n-p-1}\boxtimes \mathcal{T}^{r}_{p+1}.
\]
\end{prop}

\begin{proof}
Lemma \ref{MV is Arc} tells us that $E^{r}_{p,q}\big[F_{n}(\R^{2})\big]\cong arcE^{r}_{p,q}\big[F_{n}(\R^{2})\big]$, so it suffices to compute $arcE^{r}_{p,q}\big[F_{n}(\R^{2})\big]$.
We do this by inducting on $r$, noting that Proposition \ref{E2pageforFnR2} handles the base case.

Assume that the proposition holds for all $r\le \rho$.
We wish to show that it holds for $r=\rho+1$.
By assumption
\[
arcE^{\rho}_{p,q}\big[F_{n}(\R^{2})\big]\cong \text{Ind}^{S_{n}}_{S_{n-p-1}\times S_{p+1}}H^{\mathcal{A}_{\rho-1}}\bigg(\cdots H^{\mathcal{A}_{1}}\Big(H_{q}\big(F(\R^{2})\big)\Big)\bigg)_{n-p-1}\boxtimes \mathcal{T}^{\rho}_{p+1}.
\]
Moreover, by Propositions \ref{little cube operad structure} and \ref{basis for homology of FnR2} and our assumptions, we have a basis for
\[
\text{Ind}^{S_{n}}_{S_{n-p-1}\times S_{p+1}}H^{\mathcal{A}_{\rho-1}}\bigg(\cdots H^{\mathcal{A}_{1}}\Big(H_{q}\big(F(\R^{2})\big)\Big)\bigg)_{n-p-1}\boxtimes \mathcal{T}^{\rho}_{p+1},
\]
consisting of products of wheels on at least $\rho$ letters and Lie polynomials of degree at least $\rho$.

Next, we will show that
\[
\text{Ind}^{S_{n}}_{S_{n-p-1}\times S_{p+1}}
H^{\mathcal{A}_{\rho}}\Bigg(
H^{\mathcal{A}_{\rho-1}}\bigg(\cdots H^{\mathcal{A}_{1}}\Big(H_{q}\big(F(\R^{2})\big)\Big)\bigg)\Bigg)_{n-p-1}\boxtimes \mathcal{T}^{\rho+1}_{p+1}
\]
is a summand of $arcE^{\rho+1}_{p,q}\big[F_{n}(\R^{2})\big]$.
Proposition \ref{Leibniz Rule} and Lemma \ref{differentials on Reut} prove that the $d^{\rho}$-differential sends all elements of $\mathcal{T}^{\rho+1}_{s+1}$ to $0$, so 
\[
d^{\rho}\Bigg(\text{Ind}^{S_{n}}_{S_{n-s-1}\times S_{s-1}}
H^{\mathcal{A}_{\rho}}\Bigg(
H^{\mathcal{A}_{\rho-1}}\bigg(\cdots H^{\mathcal{A}_{1}}\Big(H_{q}\big(F(\R^{2})\big)\Big)\bigg)\Bigg)_{n-s-1}\boxtimes \mathcal{T}^{\rho+1}_{s+1}\Bigg)=0\in arcE^{\rho}_{s-\rho, t+\rho-1}\big[F_{n}(\R^{2})\big].
\]
Furthermore, nothing in
\[
\text{Ind}^{S_{n}}_{S_{n-p-1}\times S_{p+1}}
H^{\mathcal{A}_{\rho}}\Bigg(
H^{\mathcal{A}_{\rho-1}}\bigg(\cdots H^{\mathcal{A}_{1}}\Big(H_{q}\big(F(\R^{2})\big)\Big)\bigg)\Bigg)_{n-s-1}\boxtimes \mathcal{T}^{\rho+1}_{p+1}\subset arcE^{\rho}_{p,q}\big[F_{n}(\R^{2})\big]
\]
lies in the image of $d^{\rho}$.
This follows from Propositions \ref{little cube operad structure} and \ref{basis for homology of FnR2} and the fact that there are no wheels on $\rho$ elements in this subspace.

Finally, we show every element in
\[
arcE^{\rho}_{p,q}\big[F_{n}(\R^{2})\big]/\Bigg(\text{Ind}^{S_{n}}_{S_{n-p-1}\times S_{p+1}}
H^{\mathcal{A}_{\rho}}\Bigg(
H^{\mathcal{A}_{\rho-1}}\bigg(\cdots H^{\mathcal{A}_{1}}\Big(H_{q}\big(F(\R^{2})\big)\Big)\bigg)\Bigg)_{n-s-1}\boxtimes \mathcal{T}^{\rho+1}_{p}\Bigg)
\]
in the kernel of $d^{\rho}$ is in the image of $arcE^{\rho}_{p+\rho, q-\rho+1}\big[F_{n}(\R^{2})\big]$ under $d^{\rho}$.

By Propositions \ref{basis for homology of FnR2} and \ref{basis for top homology of inj word} we have a basis for 
\[
arcE^{\rho}_{p,q}\big[F_{n}(\R^{2})\big]/\Bigg(\text{Ind}^{S_{n}}_{S_{n-p-1}\times S_{p+1}}
H^{\mathcal{A}_{\rho}}\Bigg(
H^{\mathcal{A}_{\rho-1}}\bigg(\cdots H^{\mathcal{A}_{1}}\Big(H_{q}\big(F(\R^{2})\big)\Big)\bigg)\Bigg)_{n-s-1}\boxtimes \mathcal{T}^{\rho+1}_{p}\Bigg)
\]
such that every element either has a wheel on $\rho$ letters as a factor or a Lie polynomial of degree $\rho$ as a factor.
By Lemma \ref{differentials on Reut} the differential $d^{\rho}$ has non-trivial image on the Lie polynomials of degree $\rho$, and it can only produce wheels on $\rho$ letters.
Proposition \ref{little cube operad structure} tells us there cannot be relations between wheels on different sets of letters or of different sizes, so we only need to consider the wheels of size $\rho$ and the Lie polynomial of degree $\rho$ factors of the basis elements.
Moreover, Proposition \ref{little cube operad structure} and the Lie polynomial relations tell us that we only need to consider sequences where the same wheels and the corresponding Lie polynomials appear.
With this in mind, one can observe that if we have a basis element of $E^{\rho}_{p,q}\big[F_{n}(\R^{2})\big]$ that has exactly $m$ wheels on $\rho$-elements: 
\[
W(i_{1, 1}, \dots, i_{\rho, 1}), \dots, W(i_{1, m}, \dots, i_{\rho, m})
\]
and $l$ Lie polynomials: 
\[
\Big[\big[[i_{1,m+1}, i_{2, m+1}], \dots\big], i_{\rho, m+1}\Big], \dots, \Big[\big[[i_{1,m+l}, i_{2, m+l}], \dots\big], i_{\rho, m+l}\Big],
\]
then there are $\binom{m+l}{l}$ such elements in $E^{\rho}_{p,q}\big[F_{n}(\R^{2})\big]$ consisting of permutations of the ordered sets $(i_{1, j},\dots,i_{\rho, j})$ as wheels or Lie polynomials.
In $E^{\rho}_{p+\rho,q-\rho+1}\big[F_{n}(\R^{2})\big]$ there are $\binom{m+l}{l+1}$ such permutations of basis elements where $l+1$ of these $(i_{1, j},\dots,i_{\rho, j})$ are Lie polynomials and the remaining $m-1$ are wheels.
In fact, we see that there is a sequence of subspaces whose elements are contained in $$\Big\{E^{\rho}_{p-(l-k)\rho,q-(l-k)\rho+l}\big[F_{n}(\R^{2})\big]\Big\}_{k=0}^{m+l}$$ that comes from permuting the roles of the ordered sets $(i_{1, j},\dots,i_{\rho, j})$ as wheels and as Lie polynomials.
Basis elements for this sequence correspond the cells of an $(m+l-1)$-simplex, and the boundary map on this simplex corresponds to the $d^{\rho}$-differential of the spectral sequence. 
Since the $(m+l-1)$-simplex is contractible and there cannot be relations between terms coming from different sequences due to Proposition \ref{little cube operad structure} and the Lie polynomial relations, we see that the basis elements of $E^{\rho}_{p,q}\big[F_{n}(\R^{2})\big]$ that have either a wheel on $\rho$ letters or a Lie polynomial of degree $\rho$ do not survive to the $E^{\rho+1}$-page, completing the proof.
\end{proof}

Given that we have calculated the terms of the augmented Mayer--Vietoris spectral sequence for $F_{n}(\R^{2})$ in Proposition \ref{Er page for FnR2} and that we need to compare them to the corresponding terms in the augmented Mayer--Vietoris spectral sequence for $SF_{n}(R_{w,h})$, we turn to calculating the differentials of this second spectral sequence.
We will see that if $w$ and $h$ are large enough with respect to $p$, $q$ and $n$, then we only need to consider the action of these differentials on the elements of $\mathcal{T}_{p}$.
Provided that the corresponding entries are isomorphic, this will allow us prove Theorem \ref{space homological stability}.

Consider the map
\[
\text{proj}_{l}:SF_{n}(R_{w,h})\to R_{w,h}
\]
that projects the $l$-th square in a configuration onto its location in $R_{w,h}$.
There are $n$ such maps, and they induce maps on $C_{q}\big(SF_{n}(R_{w,h})\big)$ that we also call $\text{proj}_{l}$.
If most of these maps avoid enough of $R_{w,h}$, then, as a Corollary of Propositions \ref{Leibniz Rule} and \ref{differentials on Reut}, we can calculate the image of the differentials of our spectral sequence.

\begin{prop}\label{differentials are well defined as long as we have enough space}
Let $[\alpha]\in H_{t}\big(SF_{n}(R_{w,h}; j_{0},\dots, j_{s}; i_{0}, \dots, i_{p})\big)$ be a class that can be realized by a $q$-cycle $\alpha$ such that for $l\neq j_{0}, \dots, j_{s}$ the image
\[
\text{proj}_{l}(\alpha)\subset R_{w,h}
\]
does not intersect the $(r+1)\times (s+1)$ rectangle at the top right of $R_{w,h}$, whereas for $l=j_{0}, \dots,j_{s}$ this image is contained in this $(r+1)\times (s+1)$ rectangle.
Let 
\[
\sum_{\tau\in S_{s+1}}a_{\tau}\big[\tau(\alpha)\big]\in \bigoplus_{\tau\in S_{s+1}}H_{q}\big(SF_{n}(R_{w,h}; j_{\tau(0)},\dots, j_{\tau(s)}; i_{0},\dots, i_{p})\big)\subset E^{1}_{s,t}\big[SF_{n}(R_{w,h};i_{0},\dots, i_{p})\big]
\]
be the class that arises by having $S_{s+1}$ act on the $j_{0}, \dots, j_{s}$ by permuting the indices.

Then
\[
d^{r}\Bigg(\sum_{\tau\in S_{s+1}}a_{\tau}\big[\tau(\alpha)\big])\Bigg)=[\alpha]d^{r}\Bigg(\sum_{\tau\in S_{s+1}}a_{\tau}j_{\tau(0)}\cdots j_{\tau(p)}\Bigg),
\]
where, by abuse of notation $[\alpha]$ is a class on $n-s-1$ squares in $R_{w,h}$ that avoids the $(r+1)\times (s+1)$ rectangle containing the $j_{m}$, and 
\[
d^{r}\Bigg(\sum_{\tau\in S_{s+1}}a_{\tau}j_{\tau(0)}\cdots j_{\tau(p)}\Bigg)
\]
is the class in $E^{r}_{s-r, t+r-1}\big[SF_{s+1}(R_{r+1, s+1})\big]$ that arises from applying $d^{r}$ to $\sum_{\tau\in S_{s+1}}a_{\tau}j_{\tau(0)}\cdots j_{\tau(s)}$ as an element of $E^{r}_{s-1, 0}\big[SF_{s+1}(R_{r+1, s+1})\big]$, and this class is contained in $R_{r+1, s+1}$ at the top right of $R_{w,h}$.
\end{prop}

In particular, if the inclusion of $SF_{n}(R_{w,h}; j_{0}, \dots, j_{s}; i_{0}, \dots i_{p})$ into $F_{n}(\R^{2}; j_{0}, \dots, j_{s}; i_{0}, \dots i_{p})$ induces an isomorphism 
\[
H_{k}\big(SF_{n}(R_{w,h}; j_{0}, \dots, j_{s}; i_{0}, \dots i_{p})\big)\cong H_{k}\big(F_{n}(\R^{2}; j_{0}, \dots, j_{s}; i_{0}, \dots i_{p})\big),
\]
then we get a basis for $H_{k}\big(SF_{n}(R_{w,h}; j_{0}, \dots, j_{s}; i_{0}, \dots i_{p})\big)$ in terms of big wheels as long as $\min\{w,h\}\ge k+2$ and $wh-n\ge \max\big\{(k+1)(k+2), hk+2\big\}$.
Realizing these wheels as big squares, Construction \ref{first construction for connectivity} yields a situation where Proposition \ref{differentials are well defined as long as we have enough space} holds, allowing us to calculate the differentials of our spectral sequence since
\[
E^{1}_{s,t}\big[SF_{n}(R_{w,h}; i_{0}, \dots, i_{p})\big]\cong \bigoplus_{\substack{j_{0}, \dots, j_{s}\in \{1,\dots, n\}-\{i_{0}, \dots, i_{p}\}\\j_{0}, \dots, j_{s} \text{ distinct}}} H_{t}\big(SF_{n}(R_{w, h}; j_{0}, \dots, j_{s}; i_{0}, \dots, i_{p})\big)
\]

In conjuncture with Proposition \ref{move the big stretchy squares around}, this will allow us to prove that many of the differentials of our augmented Mayer--Vietoris spectral sequence for $SF_{n}(R_{w,h}; i_{0},\dots, i_{p})$ are the same as the differentials of our augmented Mayer--Vietoris spectral sequence for $F_{n}(\R^{2}; i_{0},\dots, i_{p})$.
Next, we turn to proving Theorem \ref{space homological stability}, relying heavily on the results of this subsection to do so.

\section{Proof of Theorem \ref{space homological stability}}\label{section main theorem}

In this section we prove Theorem \ref{space homological stability} via a spectral sequence argument.
Namely, we compare the pair of augmented Mayer--Vietoris spectral sequences coming from the covers of $SF_{n}(R_{w, h})$ and $F_{n}(\R^{2})$ described in the previous section and prove that their entries and differentials are isomorphic in a range.
Since the $E^{1}_{p,q}$-entries of these spectral sequences are isomorphic to direct sums of terms of the form $H_{q}\big(SF_{n}(R_{w, h};i_{0},\dots, i_{p})\big)$ and $H_{q}\big(F_{n}(\R^{2};i_{0},\dots, i_{p})\big)$, respectively, we need another pair of augmented Mayer--Vietoris spectral sequences to prove that these homology groups are isomorphic.
One might worry that such an argument might need to be repeated ad infinitum in order to compare the terms of the successive spectral sequences.
Fortunately, this is not the case, as we are able to get away with only one iteration and a series of inductive arguments due to the technical Lemma \ref{cover j,i}, which we delay until the end of this section.
We begin by recalling the statement of Theorem \ref{space homological stability}.

\begin{T1}
  \thmtext
\end{T1}

As mentioned above, an attempt to use the augmented Mayer--Vietoris spectral sequences $E\big[SF_{n}(R_{w,h})\big]$ and $E\big[F_{n}(\R^{2})\big]$ to prove this theorem would only work if we knew that the inclusion of $SF_{n}(R_{w, h}; i_{0}, \dots, i_{p})$ into $F_{n}(\R^{2}; i_{0}, \dots, i_{p})$ induces an isomorphism
\[
H_{q}\big(SF_{n}(R_{w, h}; i_{0}, \dots, i_{p})\big)\cong H_{q}\big(F_{n}(\R^{2}; i_{0}, \dots, i_{p})\big)
\]
for sufficiently many $p$ and $q$.
As such, we prove the following stronger theorem.

\begin{thm}\label{generalized main theorem}
Let $n,w,h, p$ and $k$ be such that either
\begin{enumerate}
    \item $h\ge w=k+2$, $wh-n\ge (k+1)(k+2)$, and  $p=-1$, or
    \item $\min\{w-1,h\}\ge k+2$, $wh-n\ge \max\big\{(k+1)(k+2), hk+2\big\}$, and $0\le p+1\le h$, 
\end{enumerate}
then the inclusion of $SF_{n}(R_{w,h};i_{0},\dots, i_{p})$ into $F_{n}(\R^{2}; i_{0}, \dots, i_{p})$ induces an isomorphism
\[
H_{k}\big(SF_{n}(R_{w,h};i_{0},\dots, i_{p})\big)\cong H_{k}\big(F_{n}(\R^{2}; i_{0}, \dots, i_{p})\big).
\]
\end{thm}

\begin{proof}
We proceed by a nested induction, first inducting on $k$.
After noting that the theorem holds in the case $k=0$, we assume that the theorem holds for all $k\le \kappa$, and prove that it holds for $k=\kappa+1$.
To do this, we fix $h=\kappa+3$ and induct on $w$.
Finally, we turn our rectangle $R_{w,\kappa+3}$ on its side, and again induct on $w$, obtaining the theorem for all $k$, $w$, $h$, and $p$.

In the base case, where $k=0$, Proposition \ref{n le wh-2 is connected} proves that if $\min\{w,h\}\ge 2$ and $wh-n\ge 2$, then the inclusion of $SF_{n}(R_{w,h})$ into $F_{n}(\R^{2})$ induces an isomorphism
\[
H_{0}\big(SF_{n}(R_{w,h})\big)\cong H_{0}\big(F_{n}(\R^{2})\big).
\]
Additionally, Proposition  \ref{connected single right configuration space} proves that if $\min\{w-1,h\}\ge 2$, $wh-n\ge 2$, and $0\le p+1\le h$, then the inclusion of $SF_{n}(R_{w,h}; i_{0}, \dots, i_{p})$ into $F_{n}(\R^{2}; i_{0}, \dots, i_{p})$ induces an isomorphism
\[
H_{0}\big(SF_{n}(R_{w,h}; i_{0}, \dots, i_{p})\big)\cong H_{0}\big(F_{n}(\R^{2}; i_{0}, \dots, i_{p})\big),
\]
completing the base case of our induction on $k$.

Next, assume that the theorem holds for all $k\le \kappa$, i.e., if $k\le \kappa$ and either 
\begin{enumerate}
    \item $h\ge w=k+2$, $wh-n\ge (k+1)(k+2)$, and $p=-1$, or
    \item $\min\{w-1,h\}\ge k+2$, $wh-n\ge \max\big\{(k+1)(k+2), hk+2\big\}$, $0\le p+1\le h$,
\end{enumerate}
then the inclusion of $SF_{n}(R_{w,h};i_{0}, \dots, i_{p})$ into $F_{n}(\R^{2};i_{0}, \dots, i_{p})$ induces an isomorphism
\[
H_{k}\big(SF_{n}(R_{w,h};i_{0}, \dots, i_{p})\big)\cong H_{k}\big(F_{n}(\R^{2};i_{0}, \dots, i_{p})\big).
\]

We wish to show that for all $n$, $w$, $h$, and $p$, such that either
\begin{enumerate}
\item $h\ge w=\kappa+3$, $wh-n\ge (\kappa+2)(\kappa+3)$, and $p=-1$, or
\item $\min\{w-1,h\}\ge \kappa+3$, $wh-n\ge \max\big\{(\kappa+2)(\kappa+3), h(\kappa+1)+2\big\}$, and $0\le p+1\le h$, 
\end{enumerate}
then the inclusion of $SF_{n}(R_{w,h};i_{0}, \dots, i_{p})$ into $F_{n}(\R^{2};i_{0}, \dots, i_{p})$ induces an isomorphism
\[
    H_{\kappa+1}\big(SF_{n}(R_{w,h};i_{0}, \dots, i_{p})\big)\cong H_{\kappa+1}\big(F_{n}(\R^{2};i_{0}, \dots, i_{p})\big).
\]

We proceed by fixing $h=\kappa+3$ and inducting on $w$.
Note that in this case,
\[
(\kappa+2)(\kappa+3)>(\kappa+3)(\kappa+1)+2,
\]
so we can ignore the restriction that $wh-n\ge h(\kappa+1)+2$.

For the base case, $w=\kappa+3$, so we only need to consider $n$ at most $\kappa+3$.
In this setting, Proposition \ref{homotopy equivalence} tells us that
\[
SF_{n}\big(R_{\kappa+3,\kappa+3})\simeq F_{n}(\R^{2}),
\] 
so 
\[
H_{\kappa+1}\big(SF_{n}\big(R_{\kappa+3,\kappa+3})\big)\cong H_{\kappa+1}\big(F_{n}(\R^{2})\big).
\]
To finish the base case, we need to prove that if $w=\kappa+4$, $h=\kappa+3$, $wh-n\ge(\kappa+2)(\kappa+3)$, and $0\le  p+1\le \kappa+3$, then the inclusion of $SF_{n}\big(R_{\kappa+4,\kappa+3};i_{0}, \dots, i_{p})$ into $F_{n}(\R^{2}; i_{0}, \dots, i_{p})$ induces an isomorphism 
\[
H_{\kappa+1}\big(SF_{n}(R_{\kappa+4, \kappa+3};i_{0}, \dots, i_{p})\big)\cong H_{\kappa+1}\big(F_{n}(\R^{2};i_{0}, \dots, i_{p})\big).
\]
This follows from Proposition \ref{super proposition}.

Next, still holding $h=\kappa+3$, we complete the inductive step on $w$. 
Assume that there is some $\omega\ge \kappa+4$ such if $n$, $w$, and $p$ are such that $\omega\ge w\ge \kappa+4$, $wh-n\ge(\kappa+2)(\kappa+3)$, and $0\le p+1\le \kappa+3$, then the inclusion of $SF_{n}\big(R_{w,\kappa+3};i_{0}, \dots, i_{p})$ into $F_{n}(\R^{2}; i_{0}, \dots, i_{p})$ induces an isomorphism 
\[
H_{\kappa+1}\big(SF_{n}(R_{w, \kappa+3};i_{0}, \dots, i_{p})\big)\cong H_{\kappa+1}\big(F_{n}(\R^{2};i_{0}, \dots, i_{p})\big).
\]
In particular, this holds if $p=-1$.
We wish to prove that for $w=\omega+1$, $(\omega+1)h-n\ge(\kappa+2)(\kappa+3)$, and $0\le p+1\le \kappa+3$, the inclusion of $SF_{n}\big(R_{\omega+1,\kappa+3};i_{0}, \dots, i_{p})$ into $F_{n}(\R^{2}; i_{0}, \dots, i_{p})$ induces an isomorphism 
\[
H_{\kappa+1}\big(SF_{n}(R_{\omega+1, \kappa+3};i_{0}, \dots, i_{p})\big)\cong H_{\kappa+1}\big(F_{n}(\R^{2};i_{0}, \dots, i_{p})\big).
\]
This also follows from Proposition \ref{super proposition}, completing our induction on $w$ if we assume that $h=\kappa+3$.

To allow for general $h\ge \kappa+3$, we flip our rectangle on its side and again induct on $w$; namely, we show that if either 
\begin{enumerate}
    \item $h\ge w=\kappa+3$, $wh-n\ge(\kappa+2)(\kappa+3)$, and $p=-1$, or
    \item $\min\{w-1, h\}\ge \kappa+3$, $wh-n\ge  \max\big\{(\kappa+2)(\kappa+3), h(\kappa+1)+2\big\}$, and $0\le p+1\le h$,
\end{enumerate}
then the inclusion of $SF_{n}\big(R_{w,h};i_{0}, \dots, i_{p})$ into $F_{n}(\R^{2}; i_{0}, \dots, i_{p})$ induces an isomorphism 
\[
H_{\kappa+1}\big(SF_{n}(R_{w, h};i_{0}, \dots, i_{p})\big)\cong H_{\kappa+1}\big(F_{n}(\R^{2};i_{0}, \dots, i_{p})\big).
\]

We proceed by induction on $w$.
By our previous induction on $w$, when we fixed $h=\kappa+3$, we have that if $(\kappa+3)h-n\ge (\kappa+2)(\kappa+3)$, then the inclusion of $SF_{n}\big(R_{\kappa+3,h})\simeq SF_{n}\big(R_{h,\kappa+3})$ into $F_{n}(\R^{2})$ induces isomorphisms
\[
H_{\kappa+1}\big(SF_{n}(R_{\kappa+3, h})\big)\cong H_{\kappa+1}\big(SF_{n}(R_{h, \kappa+3})\big)\cong  H_{\kappa+1}\big(F_{n}(\R^{2})\big).
\]
This handles the first part of the base case.
For the second part, note that Proposition \ref{super proposition} tells us that if $(\kappa+4)h-n\ge\max\{(\kappa+2)(\kappa+3), h(\kappa+1)+2\big\}$ and $0\le p+1\le h$, the inclusion of $SF_{n}\big(R_{\kappa+4,h}; i_{0}, \dots, i_{p})$ into $F_{n}(\R^{2}; i_{0}, \dots, i_{p})$ induces an isomorphism
\[
H_{\kappa+1}\big(SF_{n}(R_{\kappa+4, h}; i_{0}, \dots, i_{p})\big)\cong H_{\kappa+1}\big(F_{n}(\R^{2}; i_{0}, \dots, i_{p})\big),
\]
which completes the base case of our induction on $w$ for general $h\ge \kappa+3$.

Now assume  that there is some $\omega\ge \kappa+4$ such if $n$, $w$, $h$, and $p$ are such that $h\ge \kappa+3$, $\omega\ge w\ge \kappa+4$, $wh-n\ge \max\big\{(\kappa+2)(\kappa+3), h(\kappa+1)+2\big\}$, and $0\le p+1\le h$, the inclusion of $SF_{n}\big(R_{w,h};i_{0}, \dots, i_{p})$ into $F_{n}(\R^{2}; i_{0}, \dots, i_{p})$ induces an isomorphism 
\[
H_{\kappa+1}\big(SF_{n}(R_{w, h};i_{0}, \dots, i_{p})\big)\cong H_{\kappa+1}\big(F_{n}(\R^{2};i_{0}, \dots, i_{p})\big).
\]
In particular, this holds if $p=-1$.
We wish to prove that if $w=\omega+1$, $(\omega+1)h-n\ge\max\big\{(\kappa+2)(\kappa+3), h(\kappa+1)+2\big\}$, and $0\le p+1\le h$, then the inclusion of $SF_{n}\big(R_{\omega+1,h};i_{0}, \dots, i_{p})$ into $F_{n}(\R^{2}; i_{0}, \dots, i_{p})$ induces an isomorphism 
\[
H_{\kappa+1}\big(SF_{n}(R_{\omega+1, h};i_{0}, \dots, i_{p})\big)\cong H_{\kappa+1}\big(F_{n}(\R^{2};i_{0}, \dots, i_{p})\big).
\]
This follows from Proposition \ref{super proposition}, so we have completed the inductive step on $w$ for arbitrary $h\ge \kappa+3$.
This completes the inductive argument on $k$, completing the proof of the theorem.
\end{proof}

Setting $p=-1$ in both cases in the statement of Theorem \ref{generalized main theorem}, we see that if $\min\{w,h\}\ge k+2$ and $wh-n\ge \max\big\{(k+1)(k+2), hk+2\big\}$, then the inclusion of $SF_{n}\big(R_{w,h})$ into $F_{n}(\R^{2})$ induces an isomorphism 
\[
H_{k}\big(SF_{n}(R_{w, h})\big)\cong H_{k}\big(F_{n}(\R^{2})\big).
\]
Since $SF_{n}(R_{w, h})\simeq SF_{n}(R_{h, w})$, we can interchange the roles of $w$ and $h$ to see that if $w\ge h\ge k+2$ and $wh-n\ge \max\big\{(k+1)(k+2), hk+2\big\}$, then the inclusion of $SF_{n}\big(R_{w,h})$ into $F_{n}(\R^{2})$ induces an isomorphism 
\[
H_{k}\big(SF_{n}(R_{w, h})\big)\cong H_{k}\big(F_{n}(\R^{2})\big),
\]
i.e., Theorem \ref{space homological stability} holds.

It remains to prove Proposition \ref{super proposition}, which handles the inductive step on $w$ in Theorem \ref{generalized main theorem}.
The proof of this proposition is a downward induction on $p$; we handle the base case in the following proposition.

\begin{prop}\label{beginning of inductive step on w}
If the inclusion of $SF_{n}(R_{w,h})$ into $F_{n}(\R^{2})$ induces an isomorphism
\[
H_{k}\big(SF_{n}(R_{w,h})\big)\cong H_{k}\big(F_{n}(\R^{2})\big),
\]
then the inclusion of $SF_{n+h}(R_{w+1,h};i_{0},\dots, i_{h-1})$ into $F_{n+h}(\R^{2};i_{0},\dots, i_{h-1})$ induces an isomorphism
\[
H_{k}\big(SF_{n+h}(R_{w+1,h};i_{0},\dots, i_{h-1})\big)\cong H_{k}\big(F_{n+h}(\R^{2};i_{0},\dots, i_{h-1})\big).
\]
\end{prop}

\begin{proof}
We have that
\begin{align*}
H_{k}\big(SF_{n+h}(R_{w+1,h};i_{0},\dots, i_{h-1})\big)&\cong H_{k}\big(SF_{n}(R_{w,h})\big)\\
&\cong H_{k}\big(F_{n}(\R^{2})\big)\\
&\cong H_{k}\big(F_{n+h}(\R^{2};i_{0},\dots, i_{h-1})\big),
\end{align*}
where the first isomorphism is given by Proposition \ref{move squares to right side} and the fact only $h$ squares can fit in the right-most column of $R_{w+1,h}$, the second isomorphism is by assumption, and the third is given by Proposition \ref{forget points on the right}.
Moreover, this shows that the composite isomorphism is given by the inclusion of $SF_{n+h}(R_{w+1,h};i_{0},\dots, i_{h-1})$ into $F_{n+h}(\R^{2};i_{0},\dots, i_{h-1})$.
\end{proof}

With the base case of the downward induction on $p$ handled, we complete our inductive step on $w$ by completing the inductive step on $p$.

\begin{prop}\label{super proposition} 
Fix $\kappa\ge 0$, and assume that for all $n$, $w$, $h$, $k$ and $p$ such that $k\le\kappa$, $\min\{w-1, h\}\ge k+2$, $wh-n\ge \max\big\{(k+1)(k+2), hk+2\big\}$, and $0\le p+1\le h$, the inclusion of $SF_{n}(R_{w,h}; i_{0}, \dots, i_{p})$ into $F_{n}(\R^{2}; i_{0}, \dots, i_{p})$ induces an isomorphism
\[
H_{k}\big(SF_{n}(R_{w,h}; i_{0}, \dots, i_{p})\big)\cong H_{k}\big(F_{n}(\R^{2}; i_{0}, \dots, i_{p})\big).
\]
Moreover, assume that there is some $\omega\ge \kappa+3$ and some $h\ge \kappa+3$ such that for all $n$ and $w$ where $\omega\ge w\ge \kappa+3$, and $wh-n\ge\max\big\{(\kappa+2)(\kappa+3), h(\kappa+1)+2\big\}$, the inclusion of $SF_{n}(R_{w,h})$ into $F_{n}(\R^{2})$ induces an isomorphism
\[
H_{\kappa+1}\big(SF_{n}(R_{w,h})\big)\cong H_{\kappa+1}\big(F_{n}(\R^{2})\big).
\]
Then, for this $h$ and all $n$ and $p$ such that  $(\omega+1)h-n\ge\max\big\{(\kappa+2)(\kappa+3), h(\kappa+1)+2\big\}$ and $0\le p+1\le h$, the inclusion of $SF_{n}(R_{\omega+1,h}; i_{0}, \dots, i_{p})$ into $F_{n}(\R^{2}; i_{0}, \dots, i_{p})$ induces an isomorphism
\[
H_{\kappa+1}\big(SF_{n}(R_{\omega+1,h}; i_{0}, \dots, i_{p})\big)\cong H_{\kappa+1}\big(F_{n}(\R^{2}; i_{0}, \dots, i_{p})\big).
\]
\end{prop}

\begin{proof}
We prove this by a downwards induction on $p$, starting with $p+1=h$.
That is we show it is true if $(\omega+1)h-n\ge\max\big\{(\kappa+2)(\kappa+3), h(\kappa+1)+2\big\}$ and $p+1=h$.
By assumption, we have that if $\omega h-n\ge\max\big\{(\kappa+2)(\kappa+3), h(\kappa+1)+2\big\}$, then
\[
H_{\kappa+1}\big(SF_{n}(R_{\omega,h})\big)\cong H_{\kappa+1}\big(F_{n}(\R^{2})\big),
\]
so Proposition \ref{beginning of inductive step on w} tells us that
\[
H_{\kappa+1}\big(SF_{n}(R_{\omega+1,h}; i_{0}, \dots, i_{h-1})\big)\cong  H_{\kappa+1}\big(F_{n}(\R^{2}; i_{0}, \dots, i_{h-1})\big),
\]
completing the base case of our induction.

Assume that we have proven the proposition for all $p>\pi$; we wish to show that it holds for $p=\pi$.
We proceed by comparing a pair of augmented Mayer--Vietoris spectral sequences that arise from covering $SF_{n}(R_{\omega+1, h};i_{0}, \dots, i_{\pi})$ and $F_{n}(\R^{2}; i_{0}, \dots, i_{\pi})$ with sets of the form $SF_{n}(R_{\omega+1, h};j;i_{0}, \dots, i_{\pi})$ and $F_{n}(\R^{2}; j; i_{0}, \dots, i_{\pi})$, respectively, where $j$ ranges over $\{1,\dots, n\}-\{i_{0},\dots, i_{\pi}\}$.
By Proposition \ref{E0E1EinftyMV}, the $E^{1}_{s,t}$-entries of the resulting spectral sequences are
\[
E^{1}_{s,t}\big[SF_{n}(R_{\omega+1, h};i_{0}, \dots, i_{\pi})\big]\cong \bigoplus_{\substack{j_{0}, \dots, j_{s}\in \{1,\dots, n\}-\{i_{0}, \dots, i_{\pi}\}\\j_{0}, \dots, j_{s} \text{ distinct}}} H_{t}\big(SF_{n}(R_{\omega+1, h}; j_{0}, \dots, j_{s}; i_{0}, \dots, i_{\pi})\big)
\]
and 
\[
E^{1}_{s,t}\big[F_{n}(\R^{2};i_{0}, \dots, i_{\pi})\big]\cong \bigoplus_{\substack{j_{0}, \dots, j_{s}\in \{1,\dots, n\}-\{i_{0}, \dots, i_{\pi}\}\\j_{0}, \dots, j_{s} \text{ distinct}}} H_{t}\big(F_{n}(\R^{2}; j_{0}, \dots, j_{s}; i_{0}, \dots, i_{\pi})\big).
\]

By Proposition \ref{higher terms in spectral sequence}, it would suffice to show that if $s=0$ and $t=k+1$ or $0\le s\le h-1$ and $t<\kappa+1$, then these entries are isomorphic via the inclusion of $SF_{n}(R_{\omega+1, h}; j_{0}, \dots, j_{s}; i_{0}, \dots, i_{\pi})$ into  $F_{n}(\R^{2}; j_{0}, \dots, j_{s}; i_{0}, \dots, i_{\pi})$. 
We do this by proving that 
\[
H_{t}\big(SF_{n}(R_{\omega+1, h}; j_{0}, \dots, j_{s}; i_{0}, \dots, i_{\pi})\big)\cong H_{t}\big(F_{n}(\R^{2}; j_{0}, \dots, j_{s}; i_{0}, \dots, i_{\pi})\big)
\]
for the desired $s$ and $t$.

There are two cases to consider, namely, $ s+1+\pi+1\le h$ and $s+1+\pi+1>h$.
In the latter case, the $j_{0}, \dots, j_{s}$ and $i_{0}, \dots, i_{\pi}$ cannot all fit into the same column in $R_{\omega+1, h}$.
Thus, it follows from Proposition \ref{squares in the two rightmost columns} that 
\[
SF_{n}(R_{\omega+1, h}; j_{0}, \dots, j_{s}; i_{0}, \dots, i_{\pi})\simeq SF_{n-\pi-1}(R_{\omega, h}; j_{0}, \dots, j_{s}),
\]
so
\begin{align*}
H_{t}\big(SF_{n}(R_{\omega+1, h}; j_{0}, \dots, j_{s}; i_{0}, \dots, i_{\pi})\big)&\cong H_{t}\big(SF_{n-\pi-1}(R_{\omega, h}; j_{0}, \dots, j_{s})\big)\\
&\cong H_{t}\big(F_{n-\pi-1}(\R^{2}; j_{0}, \dots, j_{s})\big)\\
&\cong H_{t}\big(F_{n}(\R^{2}; j_{0}, \dots, j_{s}; i_{0}, \dots, i_{\pi})\big).
\end{align*}
The second isomorphism follows from our assumption that $(\omega+1)h-n\ge\max\big\{(\kappa+2)(\kappa+3), h(\kappa+1)+2\big\}$, and the third follows from Proposition \ref{forget points on the right}.
Moreover, Construction 
\ref{first construction for connectivity} 
shows that one can realize the basis elements for $H_{t}\big(F_{n}(\R^{2}; j_{0}, \dots, j_{s}; i_{0}, \dots, i_{\pi})\big)$ described in Proposition \ref{basis for homology of FnR2} in $SF_{n}(R_{\omega+1, h}; j_{0}, \dots, j_{s}; i_{0}, \dots, i_{\pi})$.
It follows that this isomorphism is induced by the inclusion of $SF_{n}(R_{\omega+1, h}; j_{0}, \dots, j_{s}; i_{0}, \dots, i_{\pi})$ into the point configuration space $F_{n}(\R^{2}; j_{0}, \dots, j_{s}; i_{0}, \dots, i_{\pi})$.

In the former case, where $s+1+\pi+1\le h$, Lemma \ref{cover j,i} tells us that
\[
H_{t}\big(SF_{n}(R_{\omega+1, h}; j_{0}, \dots, j_{s}; i_{0}, \dots, i_{\pi})\big)\cong H_{t}\big(SF_{n}(R_{\omega+1, h}; j_{0}, \dots, j_{s}, i_{0}, \dots, i_{\pi})\big).
\]
Since
\begin{align*}
H_{t}\big(SF_{n}(R_{\omega+1, h}; j_{0}, \dots, j_{s}, i_{0}, \dots, i_{\pi})\big)&\cong H_{t}\big(F_{n}(\R^{2}; j_{0}, \dots, j_{s}, i_{0}, \dots, i_{\pi})\big)\\
&\cong H_{t}\big(F_{n}(\R^{2}; j_{0}, \dots, j_{s}; i_{0}, \dots, i_{\pi})\big)
\end{align*}
where the first isomorphism follows from assumption in the case $t\le \kappa$ and by the inductive hypothesis in the case $t=\kappa+1$, and the second isomorphism follows from Proposition \ref{in points can move two columns left}, we
have that
\[
H_{t}\big(SF_{n}(R_{\omega+1, h}; j_{0}, \dots, j_{s}; i_{0}, \dots, i_{\pi})\big)\cong H_{t}\big(F_{n}(\R^{2}; j_{0}, \dots, j_{s}; i_{0}, \dots, i_{\pi})\big).
\]
In the case $t<\kappa+1$, since $\omega h-n\ge h(\kappa+1)+2$, we further have that these homology groups are isomorphic to 
\[
H_{t}\big(F_{n-\pi-1}(\R^{2}; j_{0}, \dots, j_{s})\big)\\
\cong H_{t}\big(SF_{n-\pi-1}(R_{\omega, h}; j_{0}, \dots, j_{s})\big).
\]

Applying Proposition \ref{higher terms in spectral sequence}, we see that the inclusion of $SF_{n}(R_{\omega+1, h}; i_{0}, \dots, i_{\pi})$ into $F_{n}(\R^{2}; i_{0}, \dots, i_{\pi})$ induces an isomorphism
\[
H_{\kappa+1}\big(SF_{n}(R_{\omega+1, h}; i_{0}, \dots, i_{\pi})\big)\cong H_{\kappa+1}\big(F_{n}(\R^{2}; i_{0}, \dots, i_{\pi})\big),
\]
completing the inductive step and therefore the proof of the proposition.
\end{proof}

To complete our proof of Proposition \ref{super proposition}, we show that if the $E^{1}$-pages of the augmented Mayer--Vietoris spectral sequences for $SF_{n}(R_{w,h}; i_{0}, \dots, i_{p})$ and $F_{n}(\R^{2}; i_{0}, \dots, i_{p})$ agree on enough entries, then
\begin{align*}
H_{\kappa+1}\big(SF_{n}(R_{\omega+1, h}; i_{0}, \dots, i_{\pi})\big)&\cong E^{1}_{-1, k}\big[SF_{n}(R_{\omega+1, h}; i_{0}, \dots, i_{\pi})\big]\\&\cong E^{1}_{-1, k}\big[F_{n}(\R^{2}; i_{0}, \dots, i_{\pi})\big]\\&\cong H_{\kappa+1}\big(F_{n}(\R^{2}; i_{0}, \dots, i_{\pi})\big).
\end{align*}

\begin{prop}\label{higher terms in spectral sequence}
Assume that $\min\{w-1, h\}\ge k+2$, $wh-n\ge \max\big\{(k+1)(k+2), hk+2\big\}$, and $0\le p+1\le h$.
Moreover, assume that the inclusion of $SF_{n}(R_{w,h}; j_{0}; i_{0}, \dots, i_{p})$ into $F_{n}(\R^{2}; j_{0}; i_{0}, \dots, i_{p})$
induces an isomorphism
\[
H_{k}\big(SF_{n}(R_{w,h}; j_{0}; i_{0}, \dots, i_{p})\big)\cong H_{k}\big(F_{n}(\R^{2}; j_{0}; i_{0}, \dots, i_{p})\big), 
\]
and for $0\le s<h$ and $t<k$, the inclusion of $SF_{n}(R_{w,h}; j_{0}, \dots, j_{s}; i_{0}, \dots, i_{p})$ into $F_{n}(\R^{2}; j_{0}, \dots, j_{s}; i_{0}, \dots, i_{p})$ induces isomorphisms
\begin{align*}
H_{t}\big(SF_{n}(R_{w,h}; j_{0}, \dots, j_{s}; i_{0}, \dots, i_{p})\big)&\cong H_{t}\big(SF_{n-p-1}(R_{w-1,h}; j_{0}, \dots, j_{s})\big)\\&\cong H_{t}\big(F_{n-p-1}(\R^{2}; j_{0}, \dots, j_{s})\big)\\&\cong H_{t}\big(F_{n}(\R^{2}; j_{0}, \dots, j_{s}; i_{0}, \dots, i_{p})\big).
\end{align*}
Then the inclusion of $SF_{n}(R_{w,h}; i_{0}, \dots, i_{p})$ into $F_{n}(\R^{2};  i_{0}, \dots, i_{p})$ induces an isomorphism
\[
H_{k}\big(SF_{n}(R_{w,h}; i_{0}, \dots, i_{p})\big)\cong H_{k}\big(F_{n}(\R^{2}; i_{0}, \dots, i_{p})\big).
\]
\end{prop}

\begin{proof}
Consider $E\big[SF_{n}(R_{w,h};i_{0}, \dots, i_{p})\big]$ and $E\big[F_{n}(\R^{2};i_{0}, \dots, i_{p})\big]$.
By assumption, the inclusion of the square configuration space $SF_{n}(R_{w, h}; j_{0}, \dots, j_{s}; i_{0}, \dots, i_{p})$ into $F_{n}(\R^{2}; j_{0}, \dots, j_{s}; i_{0}, \dots, i_{p})$ induces isomorphisms
\begin{align*}
E^{1}_{s,t}\big[SF_{n}(R_{w,h};i_{0}, \dots, i_{p})\big]&\cong\bigoplus_{\substack{j_{0}, \dots, j_{s}\in \{1,\dots, n\}-\{i_{0}, \dots, i_{p}\}\\j_{0}, \dots, j_{s} \text{ distinct}}} H_{t}\big(SF_{n}(R_{w, h}; j_{0}, \dots, j_{s}; i_{0}, \dots, i_{p})\big)\\
&\cong \bigoplus_{\substack{j_{0}, \dots, j_{s}\in \{1,\dots, n\}-\{i_{0}, \dots, i_{p}\}\\j_{0}, \dots, j_{s} \text{ distinct}}} H_{t}\big(F_{n}(\R^{2}; j_{0}, \dots, j_{s}; i_{0}, \dots, i_{p})\big)\\
&\cong E^{1}_{s,t}\big[F_{n}(\R^{2};i_{0}, \dots, i_{p})\big].
\end{align*}
for $s=0$ and $t=k$ and for $0\le s<h$ and $t<k$.
Moreover, if $s\ge h$, then $E^{1}_{s,t}\big[SF_{n}(R_{w,h}; i_{0},\dots, i_{p})\big]=0$, since $SF_{n}\big(R_{w,h}; j_{0}, \dots, j_{s}; i_{0}, \dots, i_{p})$ is empty as there is no way to vertically align $s+1$ unit squares in $R_{w,h}$. 

Next, we prove that the $d^{1}$-differentials from $E^{1}_{s,t}\big[SF_{n}(R_{w,h};i_{0}, \dots, i_{p})\big]$ and $E^{1}_{s,t}\big[F_{n}(\R^{2};i_{0}, \dots, i_{p})\big]$ can be naturally identified for $s=0$ and $t=k$ and for $0\le s<h$ and $t<k$.
To see this we consider two cases, namely $s=0$ and $0<s<h$.

In the case $s=0$, we have that
\begin{align*}
E^{1}_{0,t}\big[SF_{n}(R_{w,h};i_{0}, \dots, i_{p})\big]&\cong\bigoplus_{j_{0}\in \{1,\dots, n\}-\{i_{0}, \dots, i_{p}\}} H_{t}\big(SF_{n}(R_{w, h}; j_{0}; i_{0}, \dots, i_{p})\big)\\
&\cong \bigoplus_{j_{0}\in \{1,\dots, n\}-\{i_{0}, \dots, i_{p}\}} H_{t}\big(F_{n}(\R^{2}; j_{0}; i_{0}, \dots, i_{p})\big)\\
&\cong E^{1}_{0,t}\big[F_{n}(\R^{2};i_{0}, \dots, i_{p})\big],
\end{align*}
and the $d^{1}$-differentials have the affect of ``freeing'' $j_{0}$. 
By assumption, the inclusion of $SF_{n}(R_{w, h}; j_{0}; i_{0}, \dots, i_{p})$ into $F_{n}(\R^{2}; j_{0}; i_{0}, \dots, i_{p})$ induces an isomorphism $H_{t}\big(SF_{n}(R_{w, h}; j_{0}; i_{0}, \dots, i_{p})\big)\cong H_{t}\big(F_{n}(\R^{2}; j_{0}; i_{0}, \dots, i_{p})\big)$; moreover, we can realize the wheels constituting the basis for $H_{t}\big(SF_{n}(R_{w, h}; j_{0}; i_{0}, \dots, i_{p})\big)$ via Construction \ref{first construction for connectivity}.  
It follows from Proposition \ref{move the big stretchy squares around} that any relation in $\bigoplus_{j_{0}\in \{1,\dots, n\}-\{i_{0}, \dots, i_{p}\}} H_{t}\big(F_{n}(\R^{2}; j_{0}; i_{0}, \dots, i_{p})\big)$ that arises by freeing $j_{0}$  must be satisfied in $\bigoplus_{j_{0}\in \{1,\dots, n\}-\{i_{0}, \dots, i_{p}\}} H_{t}\big(SF_{n}(R_{w, h}; j_{0}; i_{0}, \dots, i_{p})\big)$.
Therefore, the image of $E^{1}_{0,t}\big[SF_{n}(R_{w,h};i_{0}, \dots, i_{p})\big]$ must be spanned by the image of $E^{1}_{0,t}\big[F_{n}(\R^{2};i_{0}, \dots, i_{p})\big]$.
Since any relations in $SF_{n}(R_{w,h};i_{0}, \dots, i_{p})$, would yield relations in the $F_{n}(\R^{2};i_{0}, \dots, i_{p})$ by including the square configuration space into the point configuration space, we can identify these images in the $E^{1}_{-1, 0}$-terms.

In the $0<s<h$ setting, we once again use the fact that the inclusion of 
configuration spaces induces an isomorphism between $H_{t}\big(SF_{n}(R_{w,h}; j_{0}, \dots, j_{s}; i_{0}, \dots, i_{p})\big)$ and $H_{t}\big(F_{n}(\R^{2}; j_{0}, \dots, j_{s}; i_{0}, \dots, i_{p})\big)$.
In particular, since $wh-n\ge hk+2$, by Lemma \ref{cover j,i} we have isomorphisms
\begin{align*}
H_{t}\big(SF_{n}(R_{w,h}; j_{0}, \dots, j_{s}; i_{0}, \dots, i_{p})\big)&\cong H_{t}\big(SF_{n-p-1}(R_{w-1,h}; j_{0}, \dots, j_{s})\big)\\
&\cong H_{t}\big(F_{n-p-1}(\R^{2}; j_{0}, \dots, j_{s})\big)\\
&\cong H_{t}\big(F_{n}(\R^{2}; j_{0}, \dots, j_{s}; i_{0}, \dots, i_{p})\big).
\end{align*}
Realizing the wheel basis of Proposition \ref{basis for homology of FnR2} for $H_{t}\big(SF_{n-p-1}(R_{w-1,h}; j_{0}, \dots, j_{s})\big)$ via Construction \ref{first construction for connectivity}, and including into $SF_{n}(R_{w,h}; j_{0}, \dots, j_{s}; i_{0}, \dots, i_{p})$ by including $R_{w-1,h}$ into $R_{w,h}$ on the left, we see that any arrangement of the same set of wheels in $R_{w,h}$ must be homologous, as this is the case in $\R^{2}$, and the homologies are isomorphic via an inclusion.
It follows from Propositions \ref{Leibniz Rule} and \ref{differentials are well defined as long as we have enough space} and Lemma \ref{differentials on Reut}, that the images and kernels of these differentials can be identified in Construction \ref{first construction for connectivity} since the column of squares immediately to the left of $j_{0},\dots, j_{s}$ is empty.
Therefore, for $0\le s<h-1$ and $t<k$ there is a natural isomorphism
$E^{2}_{s,t}\big[SF_{n}(R_{w,h};i_{0}, \dots, i_{p})\big]\cong E^{2}_{s,t}\big[F_{n}(\R^{2};i_{0}, \dots, i_{p})\big]$.

If $n-p-1<h$, the proposition follows from Proposition \ref{E2pageforFnR2}, as there is an isomorphism between all terms of these spectral sequences and there is enough space in $R_{w,h}$ to ensure the differentials are the same.
Similarly, if $n-p-1>2h$, all of the entries $E^{2}_{s,t}$ are $0$ for $s\ge 0$, $t<k$, and $s+t\le k$ in both sequences.
It follows that for $r\ge 2$, no non-zero entry maps into the $E^{\infty}_{-1,k}$ entry of either spectral sequence, yielding the proposition.
It remains to handle the middle cases.

We begin by noting that the $E^{2}_{h-1, t}\big[SF_{n}(R_{w,h}; i_{0}, \dots, i_{p})\big]$-entries might be ``too big'' as though the kernels of the $d^{1}$-differentials from $E^{1}_{h-1, t}\big[SF_{n}(R_{w,h}; i_{0}, \dots, i_{p})\big]$ and $E^{1}_{h-1, t}\big[F_{n}(\R^{2}; i_{0}, \dots, i_{p})\big]$ are isomorphic, the image of $E^{1}_{h, t}\big[F_{n}(\R^{2}; i_{0}, \dots, i_{p})\big]$ in  $E^{1}_{h-1, t}\big[F_{n}(\R^{2}; i_{0}, \dots, i_{p})\big]$ might be non-zero whereas the image of $E^{1}_{h, t}\big[SF_{n}(R_{w,h}; i_{0}, \dots, i_{p})\big]$ in $E^{1}_{h-1, t}\big[SF_{n}(R_{w,h}; i_{0}, \dots, i_{p})\big]$must be zero, as this entry is zero.
We will see that this potential excess is not a problem. 

On the $E^{r}$-page of the augmented--Mayer--Vietoris spectral sequence for $SF_{n}(R_{w,h}; i_{0}, \dots, i_{p})$ the $E^{r}_{s,t}$-entries for $0\le s<h$ and $t<k$ can be one of three things: isomorphic to the corresponding entry in the Mayer--Vietoris spectral sequence for $F_{n}(\R^{2}; i_{0}, \dots, i_{p})$, ``too large'' compared to the corresponding entry, or ``too small'' compared to the corresponding entry.
In each of these three cases, we consider the images of the differential from these entries.

In the first case, where $E^{r}_{s,t}\big[SF_{n}(R_{w,h}; i_{0}, \dots, i_{p})\big]\cong E^{r}_{s,t}\big[F_{n}(\R^{2}; i_{0}, \dots, i_{p})\big]$, we consider two possibilities.
If $s\neq r-1$ and $t+r-1>k$, then these entries map into the $E^{r}_{s-r, t+r-1}$-entry, which has no effect on the $E^{r'}_{-1, k}$-entries for any $r'\ge r$, so the differentials do not matter.
Otherwise, note that since $wh-n\ge hk+2$, we can use Construction \ref{first construction for connectivity} to get a representative for each basis element of $E^{r}_{s,t}\big[SF_{n}(R_{w,h}; i_{0}, \dots, i_{p})\big]$ such that the $r$ columns immediately to the left of the $j_{0}, \dots, j_{s}$ are unoccupied, so Propositions \ref{move the big stretchy squares around}, \ref{Leibniz Rule}, and \ref{differentials are well defined as long as we have enough space} and Lemma \ref{differentials on Reut} tell us that the images and kernels of these differentials can be identified.

We say that the $E^{r}_{s,t}\big[SF_{n}(R_{w,h}; i_{0}, \dots, i_{p})\big]$-entry is ``too large'' if for $r'<r$, $E^{r'}_{s,t}\big[SF_{n}(R_{w,h}; i_{0}, \dots, i_{p})\big]$-entry is mapped into by $E^{r'}_{s,t}\big[SF_{n}(R_{w,h}; i_{0}, \dots, i_{p})\big]$ where $s>h-1$.
This domain might be zero, whereas the corresponding term in the $F_{n}(\R^{2}; i_{0}, \dots, i_{p})$ spectral sequence is nonzero.
Note that since $h\ge k+2$, these never lie on the $s+t=k$ diagonal of these spectral sequences.
It follows that these too large entries never map into $E^{r}_{-1, k}\big[SF_{n}(R_{w,h}; i_{0}, \dots, i_{p})\big]$.
That said, it could be the case that these too large entries have too big of an image in $E^{r}_{s-r, t+r-1}\big[SF_{n}(R_{w,h}; i_{0}, \dots, i_{p})\big]$ compared to the corresponding entry in the $F_{n}(\R^{2}; i_{0}, \dots, i_{p})$ augmented Mayer--Vietoris spectral sequence, and, as a result, quotient too much of their target, creating ``too small'' entries.
We claim that this is not a problem as the images from the these too small entries are isomorphic to the images of the corresponding entries in the $F_{n}(\R^{2}; i_{0}, \dots, i_{p})$ augmented Mayer--Vietoris spectral sequence.

Let $r$ be the smallest number such that $E^{r}_{s, t}\big[SF_{n}(R_{w,h}; i_{0}, \dots, i_{p})\big]$ is too small, that is, we have an isomorphism $E^{r-1}_{s, t}\big[SF_{n}(R_{w,h}; i_{0}, \dots, i_{p})\big]\cong E^{r-1}_{s, t}\big[F_{n}(\R^{2}; i_{0}, \dots, i_{p})\big]$, and the former is mapped onto by a too big entry.
Since $E^{r-1}_{s, t}\big[SF_{n}(R_{w,h}; i_{0}, \dots, i_{p})\big]$ and $E^{r-1}_{s, t}\big[F_{n}(\R^{2}; i_{0}, \dots, i_{p})\big]$ are isomorphic, the images and kernels of the $d^{r-1}$-differential on these entries are isomorphic via the above discussion.
Thus the basis for $E^{r}_{s, t}\big[F_{n}(\R^{2}; i_{0}, \dots, i_{p})\big]$ that arises via Proposition \ref{Er page for FnR2}, is a generating set for $E^{r}_{s, t}\big[SF_{n}(R_{w,h}; i_{0}, \dots, i_{p})\big]$, since we are quotienting by the image $E^{r-1}_{s+r-1, t-r+2}\big[SF_{n}(R_{w,h}; i_{0}, \dots, i_{p})\big]$, which contains the image of $E^{r-1}_{s+r-1, t-r+2}\big[F_{n}(\R^{2}; i_{0}, \dots, i_{p})\big]$.
Using Construction \ref{first construction for connectivity} and Propositions \ref{move the big squares around}, \ref{Leibniz Rule}, and \ref{differentials are well defined as long as we have enough space}, we can consider the $d^{r}$-differential on this generating set, and see that its image is the same as the $d^{r}$-differential on $E^{r}_{s, t}\big[F_{n}(\R^{2}; i_{0}, \dots, i_{p})\big]$.
It follows that the kernel of this differential $E^{r}_{s, t}\big[SF_{n}(R_{w,h}; i_{0}, \dots, i_{p})\big]$ is spanned by the basis for the kernel of the corresponding differential from $E^{r}_{s, t}\big[F_{n}(\R^{2}; i_{0}, \dots, i_{p})\big]$.
If $s\le h-r$, this kernel can only be quotiented by something whose image contains the image of the corresponding terms in the $F_{n}(\R^{2}; i_{0}, \dots, i_{p})$ augmented Mayer--Vietoris spectral sequence, so the basis for $E^{r+1}_{s, t}\big[F_{n}(\R^{2}; i_{0}, \dots, i_{p})\big]$ serves as a generating set for $E^{r+1}_{s, t}\big[SF_{n}(R_{w,h}; i_{0}, \dots, i_{p})\big]$ and the process repeats.
If $s> h-r$, then this image can only be quotiented by something that is contained in the image of the corresponding terms in the $F_{n}(\R^{2}; i_{0}, \dots, i_{p})$ augmented Mayer--Vietoris spectral sequence.
It follows that this entry becomes ``too large,'' as we can always consider the images of the chains representing the corresponding terms inside $E^{r'}_{s,t}\big[F_{n}(\R^{2}; i_{0}, \dots, i_{p})\big]$.

Thus, an entry in our augmented Mayer--Vietoris spectral sequence being ``too small'' or ``too large'' is inconsequential in our efforts to determine $E_{-1,k}^{\infty}\big[SF_{n}(R_{w,h}; i_{0}, \dots, i_{p})\big]$, as only isomorphic and ``too small'' entries can map into $E^{r}_{-1,k}\big[SF_{n}(R_{w,h}; i_{0}, \dots, i_{p})\big]$, and these have images that are isomorphic to the images of the corresponding terms in $E^{r}_{-1,k}\big[F_{n}(R_{w,h}; i_{0}, \dots, i_{p})\big]$.
By Proposition \ref{Er page for FnR2}, these images are all free $\Z$-modules, so $E^{1}_{-1,k}\big[SF_{n}(R_{w,h}; i_{0}, \dots, i_{p})]\cong H_{k}\big(SF_{n}(R_{w,h}; i_{0}, \dots, i_{p})\big)$ and $E^{1}_{-1,k}\big[F_{n}(\R^{2}; i_{0}, \dots, i_{p})]\cong H_{k}\big(F_{n}(\R^{2}; i_{0}, \dots, i_{p})\big)$ are isomorphic to their direct sum.
It follows that the inclusion of $SF_{n}(R_{w,h}; i_{0}, \dots, i_{p})$ into $F_{n}(\R^{2}; i_{0}, \dots, i_{p})$ yields the desired isomorphism. 
\end{proof}

All that remains to do to complete the proof of Proposition \ref{super proposition}, hence Theorem \ref{generalized main theorem}, is to show that under the appropriate circumstances
\[
H_{\kappa+1}\big(SF_{n}(R_{w, h};j_{0}, \dots, j_{s};i_{0},\dots, i_{p})\big)\cong H_{\kappa+1}\big(SF_{n}(R_{w, h};j_{0}, \dots, j_{s},i_{0},\dots, i_{p})\big).
\]
We do this in the following lemma, which in many ways is the technical crux of this paper. Here the assumption that $wh-n\ge hk+2$ plays a small but critical role, as it guarantees that for $k<\kappa+1$,
\[
H_{k}\big(SF_{n-p-1}(R_{w-1, h};j_{0}, \dots, j_{s})\big)\cong H_{k}\big(F_{n-p-1}(\R^{2};j_{0}, \dots, j_{s})\big).
\]

\begin{lem}\label{cover j,i}
Fix $\kappa\ge0$ and assume that for all $n$, $w$, $h$, $k$, and $p$ such that if $k\le \kappa$, $\min\{w-1, h\}\ge k+2$, $wh-n\ge \max\big\{(k+1)(k+2), hk+2\big\}$, and $0\le p+1\le h$, then the inclusion of $SF_{n}(R_{w, h};i_{0},\dots, i_{p})$ into $F_{n}(\R^{2};i_{0},\dots, i_{p})$ induces an isomorphism
\[
H_{k}\big(SF_{n}(R_{w, h};i_{0},\dots, i_{p})\big)\cong H_{k}\big(F_{n}(\R^{2};i_{0},\dots, i_{p})\big).
\]
Additionally, fix $\pi\ge0$, and assume that if $\min\{w-1,h\}\ge \kappa+3$, $wh-n\ge\max\big\{(\kappa+2)(\kappa+3), h(\kappa+1)+2\big\}$, and $\pi+1\le p+1\le h$, then the inclusion of 
$SF_{n}(R_{w, h};i_{0},\dots, i_{p})$ into $F_{n}(\R^{2};i_{0},\dots, i_{p})$ 
induces an isomorphism
\[
H_{\kappa+1}\big(SF_{n}(R_{w, h};i_{0},\dots, i_{p})\big)\cong H_{\kappa+1}\big(F_{n}(\R^{2};i_{0},\dots, i_{p})\big).
\]
Then if $n$, $w$, $h$, $p$ and $s$ are such that $\min\{w-1, h\}\ge \kappa+3$, $wh-n\ge\max\big\{(\kappa+2)(\kappa+3), h(\kappa+1)+2\big\}$, and $\pi+1\le s+1+p+1\le h$, there is an isomorphism
\[
H_{\kappa+1}\big(SF_{n}(R_{w, h};j_{0}, \dots, j_{s};i_{0},\dots, i_{p})\big)\cong H_{\kappa+1}\big(SF_{n}(R_{w, h};j_{0}, \dots, j_{s},i_{0},\dots, i_{p})\big).
\]
\end{lem}

\begin{proof}
We do this by considering an augmented Mayer--Vietoris spectral sequence distinct from one described in Subsection \ref{our mv ss}.
Recall that the set of shuffles of the ordered set $(j_{0},\dots, j_{s})$ into the ordered set $(i_{0},\dots, i_{p})$ is denoted by $\Sigma\big((j_{0},\dots, j_{s}), (i_{0},\dots, i_{p})\big)$; additionally, recall that Proposition \ref{squares in the two rightmost columns} tells us that in $SF_{n}(R_{w, h};j_{0}, \dots, j_{s};i_{0},\dots, i_{p})$, we may assume that the squares $i_{0}, \dots, i_{p}$ are in the right-most column of $R_{w, h}$, and the squares $j_{0}, \dots, j_{s}$ are in the two right-most columns.
Given $\sigma\in\Sigma\big((j_{0},\dots, j_{s}), (i_{0},\dots, i_{p})\big)$, let $W_{\sigma}$ denote the subspace of configurations in $SF_{n}(R_{w, h};j_{0}, \dots, j_{s};i_{0},\dots, i_{p})$ in which the squares $j_{0}, \dots, j_{s}$ are not in the right-most column of $R_{w, h}$ union the subspace of configurations in which the squares $j_{0}, \dots, j_{s}$ are either sliding into the right-most column of $R_{w, h}$ or are already in the right-most column of $R_{w, h}$ such that the squares $j_{0}, \dots, j_{s}, i_{0},\dots, i_{p}$ are arranged as in the shuffle $\sigma$ when read from top to bottom.
In particular, we can view $W_{\sigma}$ as a union of $SF_{n-p-1}(R_{w-1,h}; j_{0},\dots j_{s})$ and $SF_{n}(R_{w,h}; \sigma)$.
The $W_{\sigma}$ cover $SF_{n}(R_{w, h};j_{0}, \dots, j_{s};i_{0},\dots, i_{p})$, and, by subdividing the cube complex $DF_{n}(R^{*}_{w, h})$, we can ensure that these subspaces are cellular.
This yields an augmented Mayer--Vietoris spectral sequence, which we denote by $\tilde{E}\big[SF_{n}(R_{w, h};j_{0}, \dots, j_{s};i_{0},\dots, i_{p})\big]$. 
We will show that if $0\le a$ and $0\le a+b\le \kappa+1$ or $a=-1$ and $b\le \kappa+1$, then $\tilde{E}^{2}_{a,b}\big[SF_{n}(R_{w, h};j_{0}, \dots, j_{s};i_{0},\dots, i_{p})\big]=0$.
This would imply that 
\[
\tilde{E}^{1}_{-1,\kappa+1}\big[SF_{n}(R_{w, h};j_{0}, \dots, j_{s};i_{0},\dots, i_{p})\big]=H_{\kappa+1}\big(SF_{n}(R_{w, h};j_{0}, \dots, j_{s};i_{0},\dots, i_{p})\big)
\]
is the image of the $d^{1}$-differential from
\[
\tilde{E}^{1}_{0,\kappa+1}\big[SF_{n}(R_{w, h};j_{0}, \dots, j_{s};i_{0},\dots, i_{p})\big]=\bigoplus_{\sigma\in\Sigma\big((j_{0},\dots, j_{s}), (i_{0},\dots, i_{p})\big)} H_{\kappa+1}(W_{\sigma}),
\]
as every other entry that maps into the $\tilde{E}^{r}_{-1,\kappa+1}\big[SF_{n}(R_{w, h};j_{0}, \dots, j_{s};i_{0},\dots, i_{p})\big]$-entry is $0$.
Proving that this image is isomorphic to $H_{\kappa+1}\big(SF_{n}(R_{w, h},j_{0}, \dots, j_{s},i_{0},\dots, i_{p})\big)$ would complete the proof of the lemma.

Given two shuffles $\sigma, \tau\in \Sigma\big((j_{0},\dots, j_{s}), (i_{0},\dots, i_{p})\big)$, note that
\[
W_{\sigma}\cap W_{\tau}\simeq SF_{n-p-1}(R_{w-1, h}; j_{0}, \dots, j_{s}),
\]
as this intersection consists of all configurations in which the squares $j_{0},\dots, j_{s}$ are in the second right-most column of $R_{w,h}$ and nothing else.
Therefore, for $a\ge 1$ and all $b$,
\begin{align*}
\tilde{E}^{1}_{a,b}\big[SF_{n}(R_{w, h};j_{0}, \dots, j_{s};i_{0},\dots, i_{p})\big]&=\bigoplus_{\substack{I\subset \Sigma\big((j_{0},\dots, j_{s}), (i_{0},\dots, i_{p})\big)\\|I|=a+1}}H_{b}\left(\bigcap_{i\in I}W_{\sigma_{i}}\right)\\
&\cong\bigoplus_{\substack{I\subset \Sigma\big((j_{0},\dots, j_{s}), (i_{0},\dots, i_{p})\big)\\|I|=a+1}} H_{b}\big(SF_{n-p-1}(R_{w, h};j_{0},\dots, j_{s})\big).
\end{align*}
Furthermore, for $a=0$ and $b\le \kappa$, we have that
\begin{align*}
\tilde{E}^{1}_{0,b}\big[SF_{n}(R_{w, h};j_{0}, \dots, j_{s};i_{0},\dots, i_{p})\big]&=\bigoplus_{\sigma\in\Sigma\big((j_{0},\dots, j_{s}), (i_{0},\dots, i_{p})\big)}H_{b}(W_{\sigma_{i}})\\
&\supseteq\bigoplus_{\sigma\in\Sigma\big((j_{0},\dots, j_{s}), (i_{0},\dots, i_{p})\big)}H_{b}\big(SF_{n-p-1}(R_{w-1,h};j_{0}, \dots, j_{s})\big)
\end{align*}
as the inclusion of $SF_{n-p-1}(R_{w, h}; j_{0},\dots, j_{s})$ into $W_{\sigma}$ induces an injection on homology.
This follows from the assumption that $wh-n\ge h(\kappa+1)+2$, so the inclusion of $SF_{n-p-1}(R_{w,h};j_{0}, \dots, j_{s})$ into $F_{n-p-1}(\R^{2};j_{0}, \dots, j_{s})$ induces an isomorphism
\[
H_{b}\big(SF_{n-p-1}(R_{w,h};j_{0}, \dots, j_{s})\big)\cong H_{b}\big(F_{n-p-1}(\R^{2};j_{0}, \dots, j_{s})\big),
\]
and from Proposition \ref{in points can move two columns left}, which proves that this holds in the $F_{n}(\R^{2}; j_{0}, \dots, j_{s}; i_{0},\dots, i_{p})$ setting, with the inclusion being an equality.
It follows from the definition that for $b\le \kappa$, the $d^{1}$-differential on the $\tilde{E}^{1}_{a,b}\big[SF_{n}(R_{w, h};j_{0}, \dots, j_{s};i_{0},\dots, i_{p})\big]$-entries of this augmented Mayer--Vietoris spectral sequence is equivalent to the boundary map on the nerve of the cover $\{W_{\sigma}\}$ for $a\ge 1$ and $b\le \kappa$.
Since this nerve is homotopy equivalent to an $\big|\Sigma\big((j_{0},\dots, j_{s}), (i_{0},\dots, i_{p})\big)\big|$-simplex, we see that 
\[
\tilde{E}^{2}_{a,b}\big[SF_{n}(R_{w, h};j_{0}, \dots, j_{s};i_{0},\dots, i_{p})\big]=0
\]
for all $a\le 1$ and all $b\le \kappa$.
Since augmented Mayer--Vietoris spectral sequences collapse to $0$, this holds for $a=0$ and $a=-1$ for all $b\le \kappa$.

Therefore, the only $d^{r}$-differential into the $\tilde{E}^{r}_{-1,\kappa+1}\big[SF_{n}(R_{w, h};j_{0}, \dots, j_{s};i_{0},\dots, i_{p})\big]$-entry with potentially non-zero image is
\[
d^{1}:\tilde{E}^{1}_{0,\kappa+1}\big[SF_{n}(R_{w, h};j_{0}, \dots, j_{s};i_{0},\dots, i_{p})\big] \to \tilde{E}^{1}_{-1,\kappa+1}\big[SF_{n}(R_{w, h};j_{0}, \dots, j_{s};i_{0},\dots, i_{p})\big].
\]
Since
\[
\tilde{E}^{r}_{0,b}\big[SF_{n}(R_{w, h};j_{0}, \dots, j_{s};i_{0},\dots, i_{p})\big]=\bigoplus_{\sigma\in\Sigma\big((j_{0},\dots, j_{s}), (i_{0},\dots, i_{p})\big)}H_{b}(W_{\sigma_{i}}),
\]
we study the $H_{\kappa+1}(W_{\sigma_{i}})$.

Given a class $[\alpha_{\sigma}]$ in $H_{\kappa+1}(W_{\sigma})$, there are two distinct possibilities: In every cellular $(\kappa+1)$-chain representing $[\alpha_{\sigma}]$, the squares $j_{0}, \dots, j_{s}$ are never in a position in which they can be slid into the right-most column of $R_{w, h}$ so that $j_{0}, \dots, j_{s}, i_{0}, \dots, i_{p}$ are arranged in the shuffle $\sigma$; alternatively, there is some $(\kappa+1)$-chain representing $[\alpha_{\sigma}]$ in which we can slide the squares $j_{0}, \dots, j_{s}$ to be in the right-most column of $R_{w, h}$ so that the squares $j_{0}, \dots, j_{s};i_{0},\dots, i_{p}$ are arranged as in the shuffle $\sigma$.
In the former case, we can interpret $[\alpha_{\sigma}]$ as a class $[\alpha]$ of $H_{\kappa+1}\big(SF_{n-p-1}(R_{w, h}; j_{0}, \dots, j_{s})\big)\subset E^{1}_{1, \kappa+1}\big[SF_{n}(R_{w, h};j_{0}, \dots, j_{s};i_{0},\dots, i_{p})\big]$.
Since $s+1+p+1\le h+1$, there must be some shuffle $\tau$ for which there is some element of some $(\kappa+1)$-chain representing $[\alpha]$, such that we can slide the squares $j_{0}, \dots, j_{s}$ into the right-most column of $R_{w, h}$ so that the squares $j_{0}, \dots, j_{s}, i_{0}, \dots, i_{p}$ are arranged in the shuffle $\tau$.
We write $[\alpha_{\tau}]$ for the corresponding class in $H_{\kappa+1}(W_{\tau})$.
Since $d^{1}\big([\alpha]\big)=[\alpha_{\sigma}]-[\alpha_{\tau}]$, viewing $[\alpha]$ as a class in $H_{\kappa+1}(W_{\sigma}\cap W_{\tau})$, we may assume without loss of generality there is some cellular $(\kappa+1)$-chain representing $[\alpha_{\sigma}]$ such there is a configuration in some cell in this chain in which we may slide the $j_{0}, \dots, j_{s}$ into the right-most column of $R_{w, h}$ so that the squares $j_{0}, \dots, j_{s}, i_{0}, \dots, i_{p}$ are arranged as in the shuffle $\sigma$.

Given such an $[\alpha_{\sigma}]\in H_{\kappa+1}(W_{\sigma})$, we know that there is some cell in a cellular $(\kappa+1)$-chain representing $[\alpha_{\sigma}]$ in which we may push the squares $j_{0}, \dots, j_{s}$ into the right-most column of $R_{w, h}$ so that the squares $j_{0}, \dots, j_{s}, i_{0}, \dots, i_{p}$ are arranged as in the shuffle $\sigma$.
By sliding these squares into this column, at this point, and keeping them vertically aligned with the squares $i_{0}, \dots, i_{p}$ while the other $n-s-1-p-1$ squares move about $R_{w, h}$ as in the representative for $[\alpha_{\sigma}]$ gives us a homologous cycle. 
This follows from the fact that none of these other squares can be to the right of the $j_{0}, \dots, j_{s}$ and if they are in the same column as $i_{0}, \dots, i_{p}$, then the squares $j_{0}, \dots, j_{s}, i_{0}, \dots, i_{p}$ are arranged as in the shuffle $\sigma$.
It follows that $[\alpha_{\sigma}]$ can be viewed as a class in $H_{\kappa+1}\big(SF_{n}(R_{w, h}; \sigma)\big)$.
Thus the image of
\[
d^{1}:\bigoplus_{\sigma\in\Sigma\big((j_{0},\dots, j_{s}), (i_{0},\dots, i_{p})\big)}H_{\kappa+1}(W_{\sigma})\to H_{\kappa+1}\big(SF_{n}(R_{w, h}; j_{0}, \dots, j_{s}; i_{0}, \dots, i_{p})\big)
\]
is the same as the image 
\[
d^{1}:\bigoplus_{\sigma\in\Sigma\big((j_{0},\dots, j_{s}), (i_{0},\dots, i_{p})\big)}H_{\kappa+1}(SF_{n}(R_{w, h}; \sigma)\big)\to H_{b}\big(SF_{n}(R_{w, h}; j_{0}, \dots, j_{s}; i_{0}, \dots, i_{p})\big),
\]
where this map is induced by the inclusion $SF_{n}(R_{w, h}; \sigma)\subset W_{\sigma}$.
As such it remains show that the coimage of 
\begin{equation}\label{avoid repetitions}
d^{1}:\bigoplus_{\substack{I\subset \Sigma\big((j_{0},\dots, j_{s}), (i_{0},\dots, i_{p})\big)\\|I|=2}}H_{\kappa+1}\left(\bigcap_{i\in I}W_{\sigma_{i}}\right)\to \bigoplus_{\sigma\in\Sigma\big((j_{0},\dots, j_{s}), (i_{0},\dots, i_{p})\big)}H_{\kappa+1}\big(SF_{n}(R_{w, h}; \sigma)\big)
\end{equation}
is isomorphic to $H_{\kappa+1}\big(SF_{n}(R_{w, h}; j_{0}, \dots, j_{s}, i_{0}, \dots, i_{p})\big)$.

To determine the coimage of this map, note that by assumption
\[
H_{\kappa+1}\big(SF_{n}(R_{w, h}; \sigma)\big)\cong H_{\kappa+1}\big(F_{n}(\R^{2}; \sigma)\big).
\]
As such, every class $[\alpha_{\sigma}]$ in $H_{\kappa+1}\big(SF_{n}(R_{w, h}; \sigma)\big)$ can be represented by a product of wheels.
Viewing the wheels as ``big squares,'' we may use Construction \ref{first construction for connectivity} to embed these big squares in $R_{w, h}$ so there is an unoccupied unit square immediately to the left of the wheels/squares $W(j_{0}),\dots, W(j_{s})$.
Moreover, that construction allows us to assume that the square immediately above the square to the left of the first $W(j_{l})$ that has an $W(i_{m})$ above it, is also unoccupied.
As such, sliding the $W(j_{0}),\dots, W(j_{s})$ to left, yields a class $[\alpha]\in H_{\kappa+1}(W_{\sigma}\cap W_{\tau})$, where $\tau$ results from swapping $j_{l}$ and $i_{m}$ in the shuffle $\sigma$.
It follows that $d^{1}\big([\alpha]\big)=[\alpha_{\sigma}]-[\alpha_{\tau}]$, where $[\alpha_{\tau}]$ is the corresponding class in $H_{\kappa+1}\big(SF_{n}(R_{w, h}; \tau)\big)$.
Therefore, the coimage of (\ref{avoid repetitions})
is isomorphic to a quotient of $H_{\kappa+1}\big(SF_{n}(R_{w, h}; j_{0}, \dots, j_{s}, i_{0}, \dots, i_{p})\big)$.

The image of a class $[a_{\sigma}]$ in $H_{\kappa+1}\big(SF_{n}(R_{w, h}; j_{0}, \dots, j_{s}, i_{0}, \dots, i_{p})\big)$ under the $d^{1}$-differential cannot be trivial in $H_{\kappa+1}\big(SF_{n}(R_{w, h}; j_{0}, \dots, j_{s}; i_{0}, \dots, i_{p})\big)$.
This follows from the fact that this would imply that the corresponding class in $H_{\kappa+1}\big(F_{n}(\R^{2}; j_{0}, \dots, j_{s}, i_{0}, \dots, i_{p})\big)$ is trivial in $H_{\kappa+1}\big(F_{n}(\R^{2}; j_{0}, \dots, j_{s}; i_{0}, \dots, i_{p})\big)$ under the $d^{1}$-differential of the corresponding spectral sequence for $F_{n}(\R^{2}; j_{0}, \dots, j_{s}; i_{0}, \dots, i_{p})$, a contradiction to the fact that in this case the $d^{1}$-differential restricted to $H_{\kappa+1}\big(F_{n}(\R^{2}; j_{0}, \dots, j_{s}, i_{0}, \dots, i_{p})\big)$ is an isomorphism induced by the homotopy equivalence 
\[
F_{n}(\R^{2}; j_{0}, \dots, j_{s}; i_{0}, \dots, i_{p})\simeq F_{n}(\R^{2}; j_{0}, \dots, j_{s}, i_{0}, \dots, i_{p})
\]
of Proposition \ref{in points can move two columns left}.
Therefore
\[
H_{\kappa+1}\big(SF_{n}(R_{w, h}; j_{0}, \dots, j_{s}; i_{0}, \dots, i_{p})\big)\cong H_{\kappa+1}\big(SF_{n}(R_{w, h}; j_{0}, \dots, j_{s}, i_{0}, \dots, i_{p})\big),
\]
completing the proof.
\end{proof}

Having completed our proof of Theorem \ref{space homological stability}, we turn to removing the constraint that $wh-n\ge hk+2$ in the case $k=1$.
Like Lemma \ref{cover j,i}, this will involve trying to compute various $H_{k}\big(SF_{n}(R_{w,h}; j_{0}, \dots,j_{s}; i_{0},\dots, i_{p})\big)$ via augmented Mayer--Vietoris spectral sequences, but unlike that lemma we will see that these homology groups might not be isomorphic to the corresponding $H_{k}\big(F_{n}(\R^{2}; j_{0}, \dots,j_{s}; i_{0},\dots, i_{p})\big)$.
Instead, using a different cover of $SF_{n}(R_{w,h}; j_{0}, \dots,j_{s}; i_{0},\dots, i_{p})$, we find a generating set for these homology groups and determine how these elements behave under the $d^{1}$- and $d^{2}$-differentials of the $E\big[SF_{n}(R_{w,h}; i_{0},\dots, i_{p})\big]$ augmented Mayer--Vietoris spectral sequence.

\section{A Tighter Bound for $k=1$}\label{tighter bound section}

In this section, we prove that if $k=1$, then we can sharpen the bounds of Theorem \ref{space homological stability}, removing the constraint that $wh-n\ge h+2$, that is, we prove:
\begin{T2}
  \thmtextone
\end{T2} 

As we did in our proof of Theorem \ref{space homological stability}, we instead prove a stronger theorem, which yields Theorem \ref{almost sharp homological stability for k=1} as a corollary.
Our proof of this stronger theorem is similar to that of our proof of Theorem \ref{generalized main theorem}, being a series of intermediate propositions that establish the steps of a doubly inductive argument.

\begin{thm}\label{k 1 bound p}
Let $n, w, h$, and $p$ be such that either
\begin{enumerate}
    \item $h\ge w=3$, $wh-n\ge 6$, and $p=-1$, or
    \item $\min\{w-1,h\}\ge 3$, $wh-n\ge 6$, and $0\le p+1\le h$,
\end{enumerate}
then, the inclusion of $SF_{n}(R_{w,h};i_{0},\dots, i_{p})$ into $F_{n}(\R^{2}; i_{0}, \dots, i_{p})$ induces an isomorphism
\[
H_{1}\big(SF_{n}(R_{w,h};i_{0},\dots, i_{p})\big)\cong H_{1}\big(F_{n}(\R^{2}; i_{0}, \dots, i_{p})\big).
\]
\end{thm}

To recover Theorem \ref{almost sharp homological stability for k=1} from this theorem, it suffices to only consider $p=-1$.

\begin{proof}
Our proof is doubly inductive.
After proving that the theorem holds for $h=3$ and all $w$, we flip the rectangle on its side, while holding $h\ge 3$ fixed, and induct on $w$.
To complete the inductive step on $w$, we induct downward on $p$.

Fixing $h=3$, the theorem holds for $w=3$, $3h-n\ge6$, and $p=-1$, and for all $w\ge 4$, $wh-n\ge6$, and $0\le p+1\le 3$.
This follows directly from Theorem \ref{generalized main theorem}, as $6>3+2$.

Next, we flip our rectangle on its side, and hold $h\ge 3$ fixed while inducting on $w$.
We begin by noting that the $h=3$ case proves that if $h\ge w=3$ and $wh-n\ge 6$, then the inclusion of $SF_{n}(R_{3, h})\simeq SF_{n}(R_{h, 3})$ into $F_{n}(\R^{2})$ induces isomorphisms
\[
H_{1}\big(SF_{n}(R_{3, h})\big)\cong H_{1}\big(SF_{n}(R_{h, 3})\big)\cong H_{1}\big(F_{n}(\R^{2})\big).
\]

To see that the theorem holds for $w=4$, $4h-n\ge 6$, and $0\le p+1\le h$, note that by Proposition \ref{beginning of inductive step on w} the theorem holds for $w=4$, $p+1=h$, and $4h-n\ge 6$. 
Next, assuming that the theorem holds for $p>\pi$, Proposition \ref{inductive step on p for k=1} tells us that it holds for $p=\pi$, completing the inductive step on $p$ for $w=4$, completing the base case of our inductive argument on $w$.

Finally, assuming that there is some $\omega\ge 4$, such that the theorem holds for $w$ such that $\omega\ge w\ge 4$, $wh-n\ge 6$, and $0\le p+1\le h$, we wish to show that it holds for $w=\omega+1$.
Applying Proposition \ref{beginning of inductive step on w}, we see that it holds for $w=\omega+1$ and $p+1=h$.
Inducting downwards on $p$ via Proposition \ref{inductive step on p for k=1} tells us that it holds for all $0\le p+1\le h$.
This completes the inductive step on $w$ and, thus, the proof.
\end{proof}

All that remains is to prove Proposition \ref{inductive step on p for k=1}, the sharper $k=1$ analogue of Proposition \ref{super proposition}, which completes the inductive step of our downward induction on $p$.

\begin{prop}\label{inductive step on p for k=1}
Fix $\pi$, and let $n$, $w$, and $h$ be such that $\min\{w-1,h\}\ge 3$ and $wh-n\ge 6$.
If the inclusion of $SF_{n}(R_{w,h}; i_{0},\dots, i_{\pi+1})$ into $F_{n}(\R^{2}; i_{0},\dots, i_{\pi+1})$ induces an isomorphism 
\[
H_{1}\big(SF_{n}(R_{w,h}; i_{0},\dots, i_{\pi+1})\big)\cong H_{1}\big(F_{n}(\R^{2}; i_{0},\dots, i_{\pi+1})\big),
\]
then the inclusion of $SF_{n}(R_{w,h}; i_{0},\dots, i_{\pi})$ into $F_{n}(\R^{2}; i_{0},\dots, i_{\pi})$ induces an isomorphism 
\[
H_{1}\big(SF_{n}(R_{w,h}; i_{0},\dots, i_{\pi})\big)\cong H_{1}\big(F_{n}(\R^{2}; i_{0},\dots, i_{\pi})\big).
\]
\end{prop}

As one can see, this is where we tighten our bounds on $w$ and $h$---one should compare this result to Proposition \ref{super proposition}.
To prove this proposition, we take a different cover of $SF_{n}(R_{w,h}; j_{0}; i_{0}, \dots, i_{\pi})$ than we did in Lemma \ref{cover j,i}.
Unlike the previous cover, which resulted in an augmented Mayer--Vietoris spectral sequence where the challenge is concentrated in the $E^{1}_{0,\kappa+1}$-entry, this cover yields an augmented Mayer--Vietoris spectral sequence where the $E^{2}_{1,0}$-entry is the problem.
Fortunately, this entry only depends on determining the connectedness of a square configuration space, which is within our grasp. 

To simplify our proof of Proposition \ref{inductive step on p for k=1} we break it into two cases that have rather different proofs and use the fact that $wh-n\ge 6$ is equivalent to $wh-6\ge n$.

\begin{prop}\label{inductive step on p for k=1 and n small}
Fix $\pi$, and let $n$, $w$, and $h$ be such that $\min\{w-1,h\}\ge 3$ and $\min\big\{wh-6, (w-1)h+\pi+1\big\}\ge n$.
If the inclusion of $SF_{n}(R_{w,h}; i_{0},\dots, i_{\pi+1})$ into $F_{n}(\R^{2}; i_{0},\dots, i_{\pi+1})$ induces an isomorphism 
\[
H_{1}\big(SF_{n}(R_{w,h}; i_{0},\dots, i_{\pi+1})\big)\cong H_{1}\big(F_{n}(\R^{2}; i_{0},\dots, i_{\pi+1})\big),
\]
then the inclusion of $SF_{n}(R_{w,h}; i_{0},\dots, i_{\pi})$ into $F_{n}(\R^{2}; i_{0},\dots, i_{\pi})$ induces an isomorphism 
\[
H_{1}\big(SF_{n}(R_{w,h}; i_{0},\dots, i_{\pi})\big)\cong H_{1}\big(F_{n}(\R^{2}; i_{0},\dots, i_{\pi})\big).
\]
\end{prop}

\begin{prop}\label{inductive step on p for k=1 and n large}
Fix $\pi$, and let $n$, $w$, and $h$ be such that $\min\{w-1,h\}\ge 3$ and $wh-6\ge n>(w-1)h+\pi+1.$
If the inclusion of $SF_{n}(R_{w,h}; i_{0},\dots, i_{\pi+1})$ into $F_{n}(\R^{2}; i_{0},\dots, i_{\pi+1})$ induces an isomorphism 
\[
H_{1}\big(SF_{n}(R_{w,h}; i_{0},\dots, i_{\pi+1})\big)\cong H_{1}\big(F_{n}(\R^{2}; i_{0},\dots, i_{\pi+1})\big),
\]
then the inclusion of $SF_{n}(R_{w,h}; i_{0},\dots, i_{\pi})$ into $F_{n}(\R^{2}; i_{0},\dots, i_{\pi})$ induces an isomorphism 
\[
H_{1}\big(SF_{n}(R_{w,h}; i_{0},\dots, i_{\pi})\big)\cong H_{1}\big(F_{n}(\R^{2}; i_{0},\dots, i_{\pi})\big).
\]
\end{prop}

Though the proofs of these propositions are highly dependent on $n$, the overarching proof structure is the same.
Namely, we first prove that the $E^{1}_{0,1}$-terms of our spectral sequences do not give us any trouble; then we tame the $E^{2}_{1,0}$-terms.
The next pair of propositions establishes the first step in these arguments.

\begin{prop}\label{k=1 images of d1 differentials are the same n small}
Fix $\pi$, and let $n$, $w$, and $h$ be such that $\min\{w-1,h\}\ge 3$ and $\min\big\{wh-6, (w-1)h+\pi+1\big\}\ge n.$
Moreover, assume that the inclusion of $SF_{n}(R_{w,h}; i_{0},\dots, i_{\pi+1})$ into $F_{n}(\R^{2}; i_{0},\dots, i_{\pi+1})$ induces an isomorphism 
\[
H_{1}\big(SF_{n}(R_{w,h}; i_{0},\dots, i_{\pi+1})\big)\cong H_{1}\big(F_{n}(\R^{2}; i_{0},\dots, i_{\pi+1})\big).
\]
If we cover $SF_{n}(R_{w,h}; i_{0}, \dots, i_{\pi})$ and $F_{n}(\R^{2}; i_{0}, \dots, i_{\pi})$ by sets of the form $SF_{n}(R_{w,h}; j_{0}; i_{0}, \dots, i_{\pi})$ and $F_{n}(\R^{2}; j_{0}; i_{0}, \dots, i_{\pi})$, respectively, then the inclusion of $SF_{n}(R_{w,h}; i_{0}, \dots, i_{\pi})$ into $F_{n}(\R^{2}; i_{0}, \dots, i_{\pi})$ induces an isomorphism between the images of the $E^{1}_{0,1}\big[SF_{n}(R_{w,h}; i_{0}, \dots, i_{\pi})\big]$- and $E^{1}_{0,1}\big[F_{n}(\R^{2}; i_{0}, \dots, i_{\pi})\big]$-entries under their respective $d^{1}$-differentials in the corresponding augmented Mayer--Vietoris spectral sequences.
\end{prop}

\begin{prop}\label{k=1 images of d1 differentials are the same n large}
Fix $\pi$, and let $n$, $w$, and $h$ be such that $\min\{w-1,h\}\ge 3$ and $wh-6\ge n>(w-1)h+\pi+1.$
Moreover, assume that the inclusion of $SF_{n}(R_{w,h}; i_{0},\dots, i_{\pi+1})$ into $F_{n}(\R^{2}; i_{0},\dots, i_{\pi+1})$ induces an isomorphism 
\[
H_{1}\big(SF_{n}(R_{w,h}; i_{0},\dots, i_{\pi+1})\big)\cong H_{1}\big(F_{n}(\R^{2}; i_{0},\dots, i_{\pi+1})\big),
\]
If we cover $SF_{n}(R_{w,h}; i_{0}, \dots, i_{\pi})$ and $F_{n}(\R^{2}; i_{0}, \dots, i_{\pi})$ by sets of the form $SF_{n}(R_{w,h}; j_{0}; i_{0}, \dots, i_{\pi})$ and $F_{n}(\R^{2}; j_{0}; i_{0}, \dots, i_{\pi})$, respectively, then the inclusion of $SF_{n}(R_{w,h}; i_{0}, \dots, i_{\pi})$ into $F_{n}(\R^{2}; i_{0}, \dots, i_{\pi})$ induces an isomorphism between the images of the $E^{1}_{0,1}\big[SF_{n}(R_{w,h}; i_{0}, \dots, i_{\pi})\big]$- and $E^{1}_{0,1}\big[F_{n}(\R^{2}; i_{0}, \dots, i_{\pi})\big]$-entries under their respective $d^{1}$-differentials in the corresponding augmented Mayer--Vietoris spectral sequences.
\end{prop}

\begin{proof}[Proof of Proposition \ref{k=1 images of d1 differentials are the same n small}]
Since
\[
E^{1}_{0,1}\big[SF_{n}(R_{w, h}; i_{0}, \dots, i_{\pi})\big]=\bigoplus_{j_{0}\neq i_{0}, \dots, i_{\pi}}H_{1}\big(SF_{n}(R_{w, h}; j_{0}; i_{0}, \dots, i_{\pi})\big),
\]
and 
\begin{align*}
E^{1}_{0,1}\big[F_{n}(\R^{2}; i_{0}, \dots, i_{\pi})\big]&=\bigoplus_{j_{0}\neq i_{0}, \dots, i_{\pi}}H_{1}\big(F_{n}(\R^{2}; j_{0}; i_{0}, \dots, i_{\pi})\big)\\
&\cong\bigoplus_{j_{0}\neq i_{0}, \dots, i_{\pi}}H_{1}\big(F_{n}(\R^{2}; j_{0}, i_{0}, \dots, i_{\pi})\big),
\end{align*}
where the isomorphism is given by Proposition \ref{in points can move two columns left}, we need to see how $H_{1}\big(SF_{n}(R_{w, h}; j_{0}; i_{0}, \dots, i_{\pi})\big)$ compares to $H_{1}\big(F_{n}(\R^{2}; j_{0}; i_{0}, \dots, i_{\pi})\big)\cong H_{1}\big(F_{n}(\R^{2}; j_{0}, i_{0}, \dots, i_{\pi})\big)$.

By Proposition \ref{squares in the two rightmost columns}, we can assume that the square $j_{0}$ is in one of the two right-most columns of $R_{w, h}$.
With this in mind, we cover $SF_{n}(R_{w, h}; j_{0}; i_{0}, \dots, i_{\pi})$ with sets $V_{l}$ where $-1\le l\le \pi$.
Here we define $V_{l}$ to be the set of configurations in $SF_{n}(R_{w, h}; j_{0}; i_{0}, \dots, i_{\pi})$ in which the squares $j_{0}, i_{0}, \dots, i_{\pi}$ are in the same column with $i_{l}$ above $j_{0}$, which, in turn, is above $i_{l+1}$, together with the configurations, such that after sliding $i_{0}, \dots, i_{\pi}$ up and down and $j_{0}$ directly to the right, we have that $j_{0}$ is between $i_{l}$ and $i_{l+1}$, along the configurations in which $j_{0}$ is moving vertically between two grid-squares of $R_{w,h}$ such that at once it stops in one of the gird-squares, we can slide the $i_{0}, \dots, i_{\pi}$ up and down so that $j_{0}$ can be slid directly to the right to be between $i_{l}$ and $i_{l+1}$.
This is well-defined since $(w-1)h+\pi+1\ge n$ and we only need to concern ourselves with $\pi<h-1$.
Moreover, by taking the barycentric subdivision of $R^{*}_{w,h}$ and further subdividing the resulting discrete configuration space by hyperplanes of the form $x_{j}=x_{m}$ for $m\in \{1,\dots, n\}$, we get a cellular model for $SF_{n}(R_{w, h}; i_{0}, \dots, i_{\pi})$ that contains the $V_{l}$ as a subspace.
This allows us to construct an augmented Mayer--Vietoris spectral sequence arising from the cover $\{V_{l}\}^{\pi}_{l=-1}$, which we call $E'\big[SF_{n}(R_{w, h}; j_{0}; i_{0}, \dots, i_{\pi})\big]$.
Note that the $E'^{1}_{-1, 1}\big[SF_{n}(R_{w, h}; j_{0}; i_{0}, \dots, i_{\pi})\big]$-entry of this spectral sequence is $H_{1}\big(SF_{n}(R_{w, h}; j_{0}; i_{0}, \dots, i_{\pi})\big)$.
We wish to know how the images of classes coming from $E'^{1}_{0, 1}\big[SF_{n}(R_{w, h}; j_{0}; i_{0}, \dots, i_{\pi})\big]$ and $E'^{2}_{1, 0}\big[SF_{n}(R_{w, h}; j_{0}; i_{0}, \dots, i_{\pi})\big]$ in $E'^{r}_{-1, 1}\big[SF_{n}(R_{w, h}; j_{0}; i_{0}, \dots, i_{\pi})\big]$ under the $d'^{r}$-differentials behave under the $d^{1}$-differential of the $E\big[SF_{n}(R_{w, h}; i_{0}, \dots, i_{\pi})\big]$ spectral sequence.

By construction, $V_{l}$ deformation retracts onto $SF_{n}(R_{w, h}; i_{0}, \dots, i_{l}, j_{0}, i_{l+1},\dots, i_{\pi})$, so
\begin{align*}
E'^{1}_{0,1}\big[SF_{n}(R_{w, h}; j_{0}; i_{0}, \dots, i_{\pi})\big]&=\bigoplus_{-1\le l\le \pi}H_{1}(V_{l})\\&\cong\bigoplus_{-1\le l\le \pi}H_{1}\big(SF_{n}(R_{w, h}; i_{0}, \dots, i_{l}, j_{0}, i_{l+1},\dots, i_{\pi})\big).
\end{align*}
By definition, the $d'^{1}$-differential treats these classes as classes in $H_{1}\big(SF_{n}(R_{w, h}; j_{0}; i_{0}, \dots, i_{\pi})\big)$.

Since $H_{1}\big(SF_{n}(R_{w, h}; i_{0}, \dots, i_{l}, j_{0}, i_{l+1},\dots, i_{\pi})\big)\cong H_{1}\big(F_{n}(\R^{2}; i_{0}, \dots, i_{l}, j_{0}, i_{l+1},\dots, i_{\pi})\big)$ by assumption, we can use Propositions \ref{basis for homology of FnR2} and \ref{forget points on the right} to get a basis for $H_{1}\big(SF_{n}(R_{w, h}; i_{0}, \dots, i_{l}, j_{0}, i_{l+1},\dots, i_{\pi})\big)$ consisting of products of wheels.
Consider a pair of basis elements $[\alpha_{l}]\in H_{1}\big(SF_{n}(R_{w, h}; i_{0}, \dots, i_{l}, j_{0}, i_{l+1},\dots, i_{\pi})\big)$ and $[\alpha_{m}]\in H_{1}\big(SF_{n}(R_{w, h}; i_{0}, \dots, i_{m}, j_{0}, i_{m+1},\dots, i_{\pi})\big)$ that only differ by how $j_{0}$ is shuffled into $(i_{0}, \dots, i_{\pi})$.
Treating the wheels constituting these classes as big squares, place them in $R_{w,h}$ as in the target configuration of Proposition \ref{free aj in k=1}.
As the $d^{1}$-differential of the $E\big[SF_{n}(R_{w,h}; i_{0}, \dots, i_{\pi})\big]$ spectral sequence has the effect of freeing the square $j_{0}$, Proposition \ref{free aj in k=1} tells us that $d^{1}\Big(d'^{1}\big([\alpha_{l}]\big)\Big)$ and $d^{1}\Big(d'^{1}\big([\alpha_{m}]\big)\Big)$ correspond to big-square configurations in the same path-component, so they represent homologous classes.  
It follows that
\begin{align*}
d'^{1}\Big(E'^{1}_{0,1}\big[SF_{n}(R_{w, h}; j_{0}; i_{0}, \dots, i_{\pi})\big]\Big)&=d'^{1}\Big(\bigoplus_{-1\le l\le \pi}H_{1}\big(SF_{n}(R_{w, h}; i_{0}, \dots, i_{l}, j_{0}, i_{l+1},\dots, i_{\pi})\big)\Big)\\
&=d'^{1}\Big(H_{1}\big(SF_{n}(R_{w, h}; j_{0}, i_{0}, \dots, i_{\pi})\big)\Big).
\end{align*}
Furthermore, note that 
\begin{align*}
H_{1}\big(SF_{n}(R_{w, h}; j_{0}, i_{0}, \dots, i_{\pi})\big)&\cong H_{1}\big(F_{n}(\R^{2}; j_{0}, i_{0}, \dots, i_{\pi})\big)\\
&\cong H_{1}\big(F_{n}(\R^{2}; j_{0}; i_{0}, \dots, i_{\pi})\big)
\end{align*}
where the first isomorphism follows by assumption, and the second by Proposition \ref{in points can move two columns left}.
This tells us that 
\[
d'^{1}\Big(E'^{1}_{0,1}\big[SF_{n}(R_{w, h}; j_{0}; i_{0}, \dots, i_{\pi})\big]\Big)\cong H_{1}\big(SF_{n}(R_{w, h}; j_{0}, i_{0}, \dots, i_{\pi})\big),
\]
as if this did not hold, the corresponding statement would not be true for $H_{1}\big(SF_{n}(R_{w, h}; j_{0}, i_{0}, \dots, i_{\pi})\big)$.
As such, the only meaningful discrepancies between $E^{1}_{0,1}\big[SF_{n}(R_{w, h}; i_{0}, \dots, i_{\pi})\big]$ and $E^{1}_{0,1}\big[F_{n}(\R^{2}; i_{0}, \dots, i_{\pi})\big]$ can come from $E'^{2}_{1, 0}\big[SF_{n}(R_{w, h}; j_{0}; i_{0}, \dots, i_{\pi})\big]$.
Next, we bound the size of $E'^{2}_{1, 0}\big[SF_{n}(R_{w, h}; j_{0}; i_{0}, \dots, i_{\pi})\big]$ and determine its image in $H_{1}\big(SF_{n}(R_{w, h}; j_{0}; i_{0}, \dots, i_{\pi})\big)$ under the composition $d^{1}\circ d'^{2}$.

We will show that the $E'^{2}_{1,0}\big[SF_{n}(R_{w, h}; j_{0}; i_{0}, \dots, i_{\pi})\big]$-entry is isomorphic to a quotient of the group
\[ \Z^{(\pi+1)\Big(\big((w-1)h-1\big)!(h-\pi-1)-1\Big)}.\]
Given this, we compute a generating set for the images of these classes in $H_{1}\big(SF_{n}(R_{w, h}; j_{0}; i_{0}, \dots, i_{\pi})\big)$.

First, note that the structure of our cellular model for $SF_{n}(R_{w, h}; j_{0}; i_{0}, \dots, i_{\pi})$ tells us that any configuration of squares in $R_{w,h}$ can be pushed onto a configuration in which all the squares lie in grid-square.
By Proposition \ref{squares in the two rightmost columns}, we may assume that in $V_{l}$, if the square $j_{0}$ is in the second right-most column of $R_{w, h}$, then it must be in either the $(l+1)^{\text{th}}$, $(l+2)^{\text{th}}$, $\dots$, or $(h-\pi+l)^{\text{th}}$ grid-square in this column.
It follows that in $V_{l}\cap V_{l+1}$ the square $j_{0}$ may be in any one of the $(l+3)^{\text{th}}$-through $\big(h-\pi+l)^{\text{th}}$-grid-squares in the second right-most column of $R_{w, h}$.
Therefore, in $V_{l}\cap V_{l+1}$, the square $j_{0}$ can be in one of $h-\pi-1$ grid-squares of the $R_{w-1, h}$ found in the left side of $R_{w, h}$.
Moreover, since the remaining $n-\pi-2$ squares cannot be to the right of $j_{0}$, there are $\big((w-1)h-1\big)!$ ways of arranging them in $R_{w-1, h}$.
Note, that by adding ``ghost squares'' in the unoccupied squares in $R_{w-1,h}$, we may assume that $n-\pi-1=(w-1)h$.
While $V_{l}$ might have far fewer components in actuality, this over-counting is not an issue as we are dealing with first homology and only seeking a generating set.
Since only $(w-1)h$ many squares can fit in this rectangle, none of the squares not named $i_{0}, \dots, i_{\pi}$ can move, so there are $\big((w-1)h-1\big)!(h-\pi -1)$ components in $V_{l}\cap V_{l+1}$.
Letting $l$ range from $0$ to $\pi$, we see that $\bigsqcup_{l=0}^{\pi}V_{l}\cap V_{l+1}$ has $(\pi+1)\big((w-1)h-1\big)!(h-\pi -1)$ components.

Next, note that we need only to be concerned with the intersection of $V_{l}$ and $V_{l+1}$, as if we have any other intersection $V_{l}\cap V_{m}$ where $|m-l|>1$, then every component of $V_{l}\cap V_{m}$, i.e., a generator of $H_{0}(V_{l}\cap V_{m})$, is equivalent to a sum of elements in $\bigcup_{t=l}^{t=m-1} H_{0}(V_{t}\cap V_{t+1})$ via a sum of elements in $\bigcup_{t=l}^{t=m-2} H_{0}(V_{t}\cap V_{t+1}\cap V_{m})$.
Since the $1$-skeleton of the nerve complex of the cover $\{V_{l}\}_{l=-1}^{\pi}$ is connected and $E'^{1}_{0,0}\big[SF_{n}(R_{w, h}; j_{0};i_{0}, \dots, i_{\pi})\big]\cong \Z^{\pi+2}$,
we have that $E'^{2}_{1,0}\big[SF_{n}(R_{w, h}; j_{0};i_{0}, \dots, i_{\pi})\big]$
is a quotient of $\Z^{(\pi+1)\Big(\big((w-1)h-1\big)!(h-\pi-1)-1\Big)}$.
In particular, we get a generating set for this entry.
Namely, for each of the $h-\pi-1$ possible locations of $j_{0}$ in $R_{w, h}$ in a component of $V_{l}\cap V_{l+1}$, fix a base configuration.
There are $\big((w-1)h-1\big)!-1$ other configurations of $V_{l}\cap V_{l+1}$, where $j_{0}$ is in the same location.
Take all $\big((w-1)h-1\big)!-1$ differences of these configurations with the base configuration.
Letting $j_{0}$ range over the $h-\pi-1$ possible locations gives $\Big(\big((w-1) h-1\big)!-1\Big)(h-\pi-1)$ linearly independent classes.
Additionally, take the $h-\pi-2$ differences of base configurations where we only take the differences where the two $j_{0}$s are in adjacent squares.
Together these yield $\big((w-1)h-1\big)!(h-\pi-1)-1$ classes.
Letting $l$ vary from $0$ to $\pi$ gives a generating set for all of $E'^{2}_{1,0}\big[SF_{n}(R_{w, h}; j_{0};i_{0}, \dots, i_{\pi})\big]$.

Having found a generating set for $E'^{2}_{1,0}\big[SF_{n}(R_{w, h}; j_{0};i_{0}, \dots, i_{\pi})\big]$, we wish to find representatives for the images of these elements under the $d'^{2}$ differential. 
There are two such classes we need to consider, namely, the $(\pi+1)(h-\pi-2)$ differences of the base classes $[\beta]$, and the $(\pi+1)\Big(\big((w-1)h-1\big)!(h-\pi-2)\Big)$ other classes that arise from fixing $l$ and a location for $j_{0}$.

To handle the difference of base classes $[\beta_{1}]-[\beta_{2}]$, we may assume that the corresponding configurations differ only by switching location of $j_{0}$ and the square immediately below it, see Figure \ref{twobaseclassesk1}.
It follows from the definition of the differentials of the double complex giving rise to $E'\big[SF_{n}(R_{w, h}; j_{0};i_{0}, \dots, i_{\pi})\big]$ that one can represent the image of such a class in $H_{1}\big(SF_{n}(R_{w, h}; j_{0};i_{0}, \dots, i_{\pi})\big)$ by the cycle depicted in Figure \ref{differenceofbaseclassesk1}.
Applying the $d^{1}$-differential to these classes as elements of  $E^{1}_{0,1}\big[SF_{n}(R_{w,h}; i_{0}, \dots, i_{\pi})\big]$, that is their images under $d'^{2}$, gives the trivial class in $E^{1}_{0,1}\big[SF_{n}(R_{w,h}; i_{0}, \dots, i_{\pi})\big]$.
This follows from the fact that $wh-6\ge n$, so we can isolate the three moving squares in a $3$ by $3$ subrectangle of $R_{w,h}$.
Inside this $R_{3,3}$ the corresponding $1$-cycle is contractible; see Figure \ref{classH1SF43}. 

\begin{figure}[h]
\centering
\captionsetup{width=.8\linewidth}
\includegraphics[width=12cm]{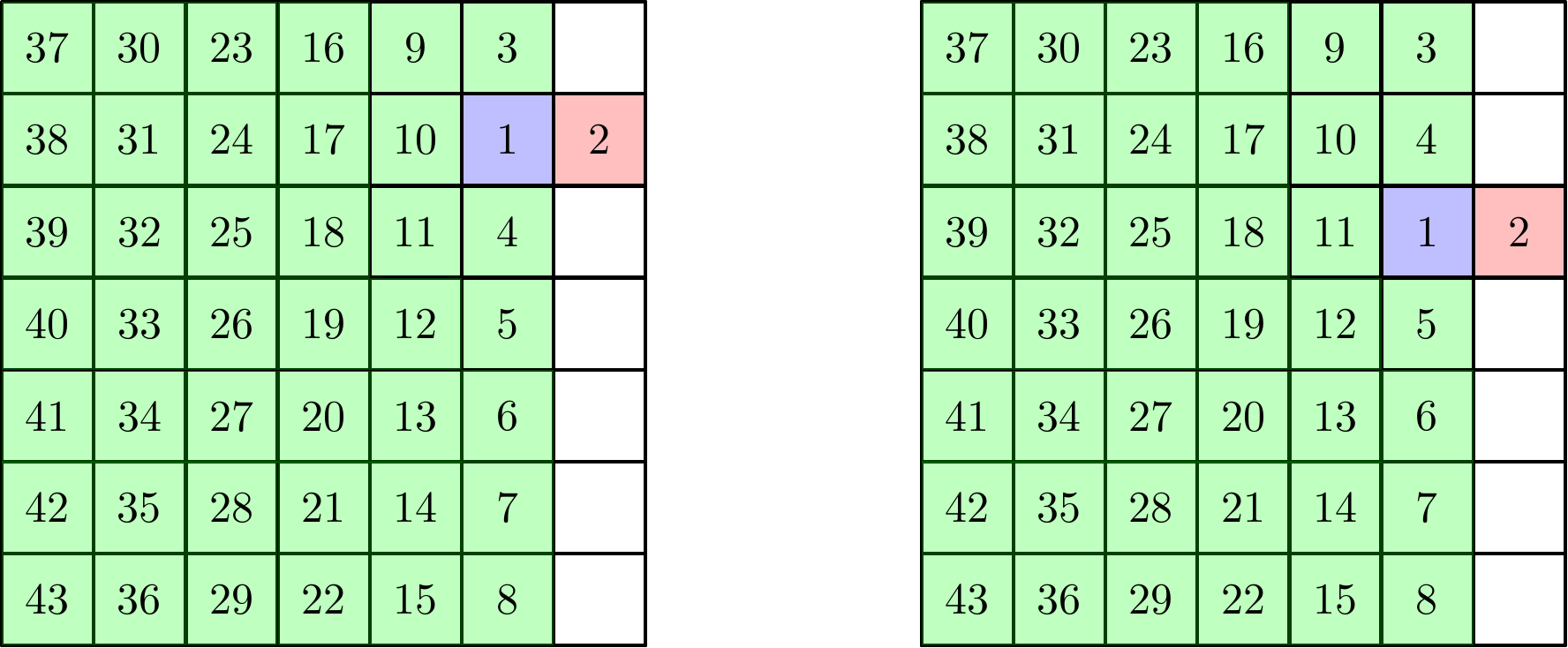}
\caption{Two base classes $[\beta_{1}]$ and $[\beta_{2}]$, whose difference yields a class in $H_{1}\big(SF_{43}(R_{7,7}; 1;2)\big)$.
See Figure \ref{differenceofbaseclassesk1}.
}
\label{twobaseclassesk1}
\end{figure}

\begin{figure}[h]
\centering
\captionsetup{width=.8\linewidth}
\includegraphics[width=12cm]{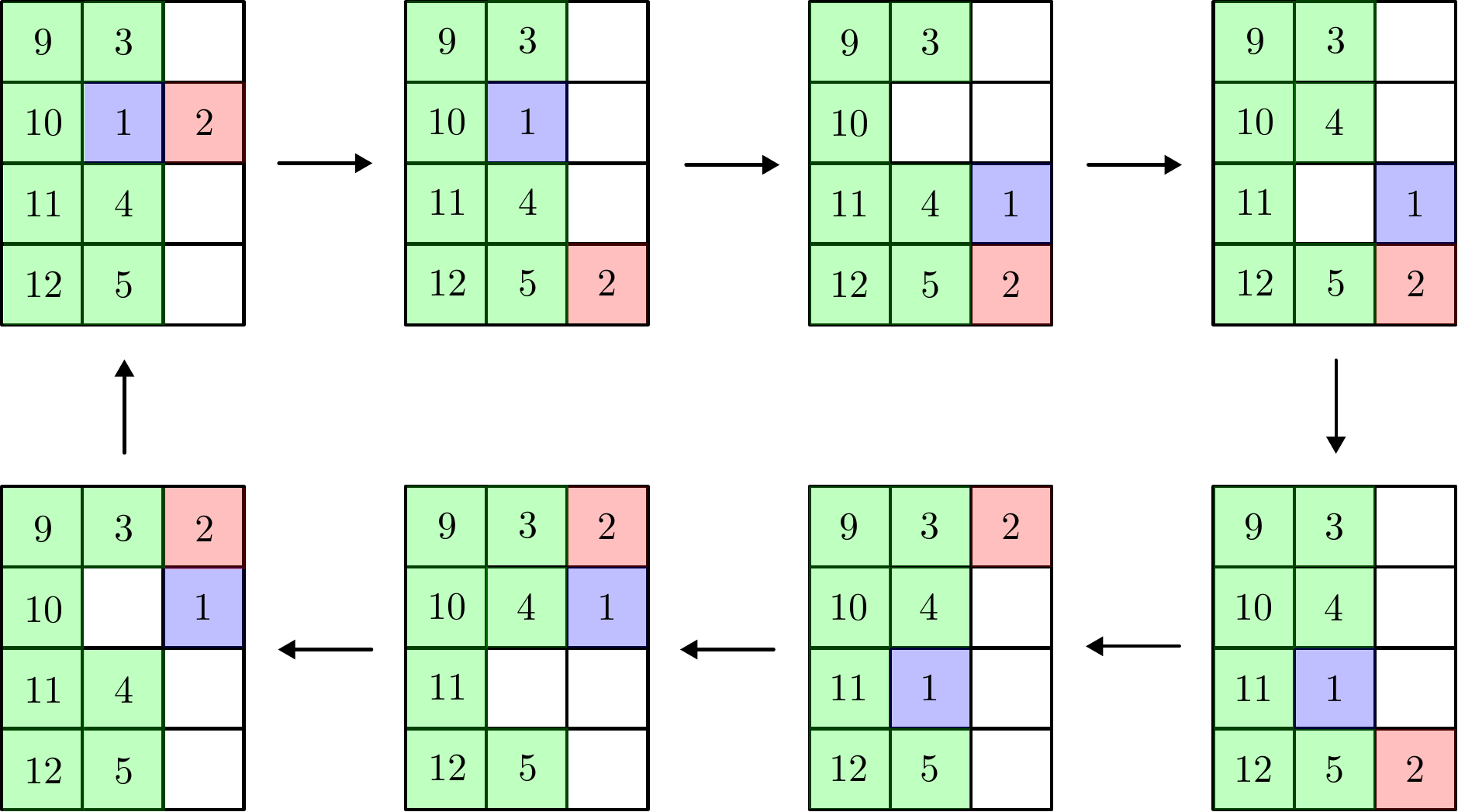}
\caption{The class in $H_{1}\big(SF_{43}(R_{7,7}; 1;2)\big)$ arising from the differences of the classes depicted in Figure \ref{twobaseclassesk1}.
We have zoomed in to the top right corner of $R_{7,7}$.
}
\label{differenceofbaseclassesk1}
\end{figure}

\begin{figure}[h]
\centering
\captionsetup{width=.8\linewidth}
\includegraphics[width=12cm]{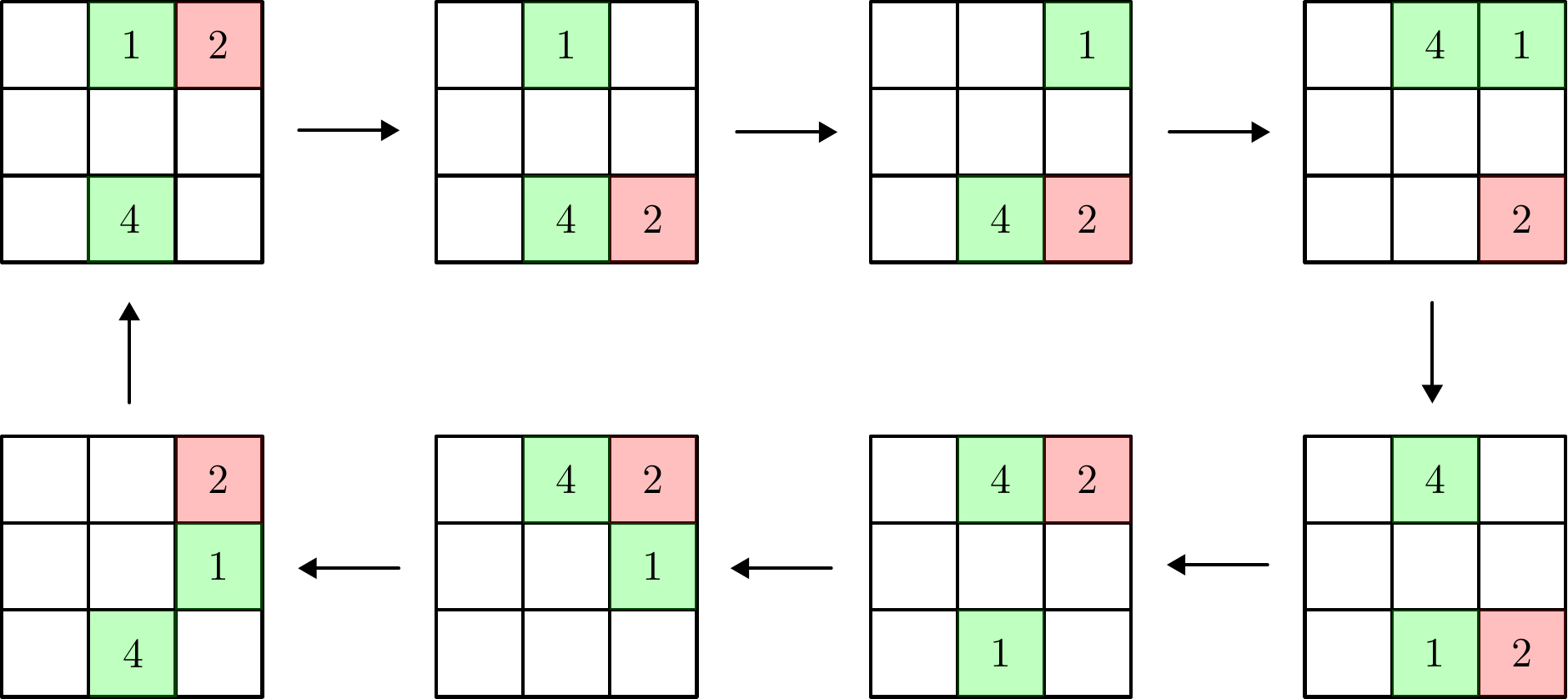}
\caption{The class depicted in Figure \ref{differenceofbaseclassesk1}, after applying the $d^{1}$-differential of the $E\big[SF_{43}(R_{7,7}; 2)\big]$ spectral sequence.
This differential frees square $1$, hence the color change from blue to green, and allows us to clear out a $3\times 3$ square in $R_{7,7}$, which is only occupied by squares $1$, $2$, and $4$.
One can easily see that this class is trivial, as the cycle representing it is contractible.
}
\label{classH1SF43}
\end{figure}

Handling the other classes directly is slightly harder, and instead of working with the given generating set, we work with a generating set corresponding to transpositions.
Recall that the symmetric group $S_{(w-1)h-1}$ is generated by transpositions; in particular, it is generated by transpositions of consecutive elements.
It follows that if we fix $l$ and the location of $j_{0}$, the span of these elements in $E'^{2}_{1,0}\big[SF_{n}(R_{w, h}; j_{0};i_{0}, \dots, i_{\pi})\big]$ is generated by taking all configurations in $H_{0}\big(V_{l}\cap V_{l+1})$ with $j_{0}$ fixed, and taking the difference with all configurations where we have swapped the position of one pair of adjacent squares $a_{1}$ and $a_{2}$, see Figure \ref{twononbaseclassesk1}.
By the definition of the differentials of the augmented Mayer--Vietoris spectral sequence $E'\big[SF_{n}(R_{w, h}; j_{0};i_{0}, \dots, i_{\pi})\big]$ a representative for the image of such a class in $H_{1}\big(SF_{n}(R_{w, h}; j_{0};i_{0}, \dots, i_{\pi})\big)$ under the $d'^{2}$-differential is depicted in Figure \ref{differenceofnonbaseclassesk1}.
Applying the $d^{1}$-differential to these classes as elements of $E^{1}_{0,1}\big[SF_{n}(R_{w,h}; i_{0}, \dots, i_{\pi})\big]$ yields a class in $E^{1}_{0,1}\big[SF_{n}(R_{w,h}; i_{0}, \dots, i_{\pi})\big]$, that is, the product of $W(a_{1}, a_{2})$ with singleton wheels. 
By Proposition \ref{free aj in k=1}, we can see that the corresponding big square configuration space with one $2\times 2$ square and at most $wh-8$ unit squares is path-connected. 
As a result, the image of this class is homologous to the image of a class coming from the $E'^{1}_{0,1}\big[SF_{n}(R_{w,h}; j_{0}; i_{0}, \dots, i_{\pi})\big]$-entry, so it does not contribute anything new to $H_{1}\big(SF_{n}(R_{w,h}; i_{0}, \dots, i_{\pi})\big)$.

\begin{figure}[h]
\centering
\captionsetup{width=.8\linewidth}
\includegraphics[width=12cm]{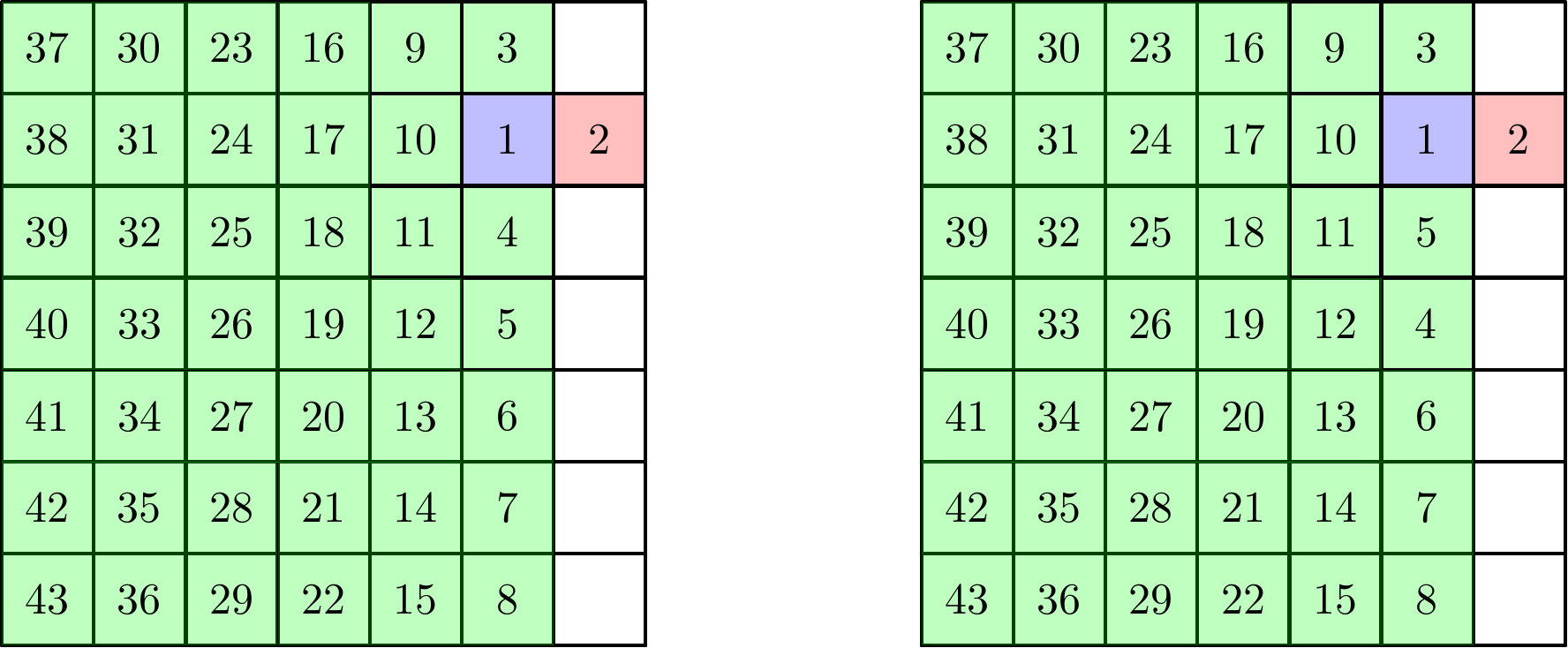}
\caption{Two classes in $E'^{1}_{1,0}\big[SF_{43} (R_{7.7};1;2)\big]$ whose difference yields a class in $H_{1}\big(SF_{43}(R_{7,7}; 1;2)\big)$; see Figure \ref{differenceofnonbaseclassesk1}.
}
\label{twononbaseclassesk1}
\end{figure}

\begin{figure}[h]
\centering
\captionsetup{width=.8\linewidth}
\includegraphics[width=16cm]{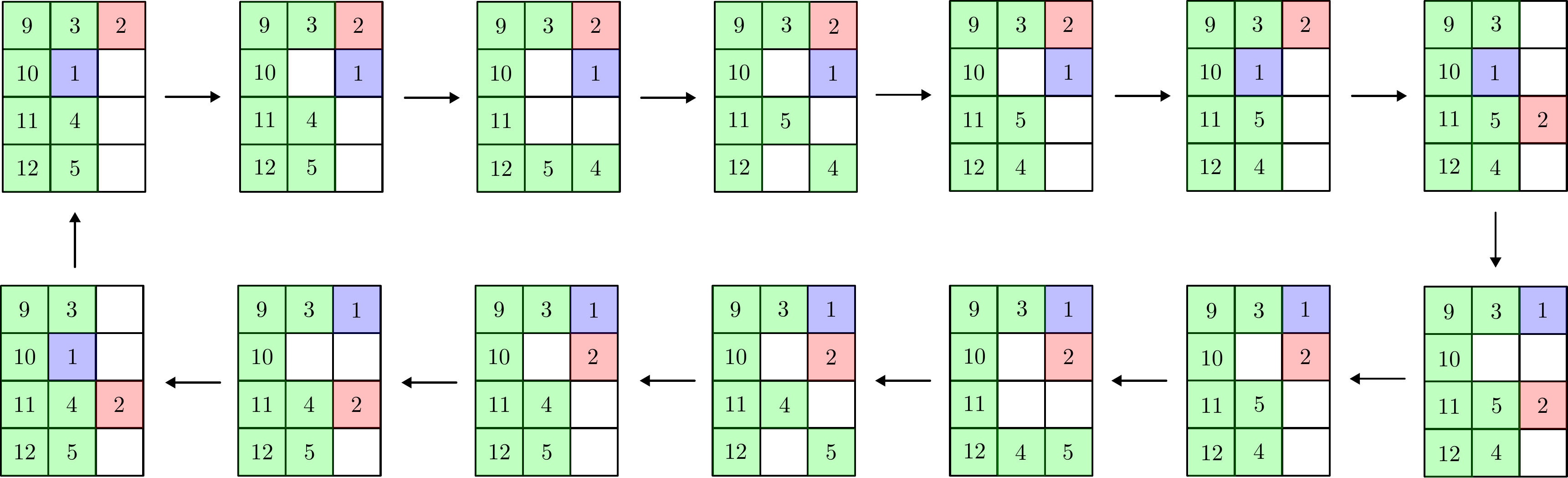}
\caption{The class in $H_{1}\big(SF_{43}(R_{7,7}; 1;2)\big)$ arising from taking the difference of the two classes in $E'^{1}_{1,0}\big[SF_{43}(R_{7,7}; 1;2)\big]$ depicted in  Figure 
\ref{twononbaseclassesk1}.
Note that by freeing square $1$, i.e., acting by the $d^{1}$-differential of the augmented Mayer--Vietoris spectral sequence $E\big[SF_{43}(R_{7,7}; 2)\big]$, the class in $H_{1}\big(SF_{43}(R_{7,7}; 2)\big)$ is the product of the wheel $W(4,5)$ with the wheels  $W(1)$, $W(2)$, $W(3)$, $W(6)$, $\dots$, $W(43)$.
Given any two differences of classes arising from swapping adjacent free squares, we may homotope the resulting cycle to look like this one.
}
\label{differenceofnonbaseclassesk1}
\end{figure}

Thus, the image of the $d^{1}$-differential from $E^{1}_{0,1}\big[SF_{n}(R_{w,h};i_{0}, \dots, i_{\pi})\big]$-entry of the augmented Mayer--Vietoris spectral sequence for $SF_{n}(R_{w,h};i_{0}, \dots, i_{\pi})$ is naturally isomorphic to the image of the $d^{1}$-differential from the $E^{1}_{0,1}$-entry of the augmented Mayer--Vietoris spectral sequence for $F_{n}(\R^{2}; i_{0}, \dots, i_{\pi})$.
\end{proof}

\begin{remark}\label{not really that much extra}
In most cases we constructed far more classes than necessary in $E'^{2}_{1,0}\big[SF_{n}(R_{w,h}; j_{0}; i_{0}, \dots, i_{\pi})\big]$.
Indeed, if $(w-1)h+\pi-1\ge n$, then one can show that all of the generators we constructed represent trivial class, and if $(w-1)h+\pi= n$, there are only the $\pi+1$ difference of base configuration classes. 
Only in the $(w-1)h+\pi+1=n$ case do all of the classes that we constructed actually come into play.
\end{remark}

\begin{proof}[Proof of Proposition \ref{k=1 images of d1 differentials are the same n large}]

Recall that
\[
E^{1}_{0,1}\big[SF_{n}(R_{w, h}; i_{0}, \dots, i_{\pi})\big]=\bigoplus_{j_{0}\neq i_{0}, \dots, i_{\pi}}H_{1}\big(SF_{n}(R_{w, h}; j_{0}; i_{0}, \dots, i_{\pi})\big),
\]
and 
\begin{align*}
E^{1}_{0,1}\big[F_{n}(\R^{2}; i_{0}, \dots, i_{\pi})\big]&=\bigoplus_{j_{0}\neq i_{0}, \dots, i_{\pi}}H_{1}\big(F_{n}(R_{w, h}; j_{0}; i_{0}, \dots, i_{\pi})\big)\\
&\cong\bigoplus_{j_{0}\neq i_{0}, \dots, i_{\pi}}H_{1}\big(F_{n}(R_{w, h}; j_{0}, i_{0}, \dots, i_{\pi})\big),
\end{align*}
where the isomorphism is given by Proposition \ref{in points can move two columns left}, we need to see how $H_{1}\big(SF_{n}(R_{w, h}; j_{0}; i_{0}, \dots, i_{\pi})\big)$ compares to $H_{1}\big(F_{n}(\R^{2}; j_{0}; i_{0}, \dots, i_{\pi})\big)\cong H_{1}\big(F_{n}(\R^{2}; j_{0}, i_{0}, \dots, i_{\pi})\big)$.

Since $n-\pi-1>(w-1)h$, the square $j_{0}$ must be in the right-most column of $R_{w,h}$.
It follows from Proposition \ref{connected single right configuration space} that the square configuration space $SF_{n}(R_{w, h}; j_{0}; i_{0}, \dots, i_{\pi})$ consists of $\pi+2$ different components $SF_{n}(R_{w, h}; i_{0}, \dots, i_{l}, j_{0}, i_{l+1},\dots, i_{\pi})$, where $l$ ranges from $-1$ to $\pi$.
Therefore
\[
H_{1}\big(SF_{n}(R_{w, h}; j_{0}; i_{0}, \dots, i_{\pi})\big)=\bigoplus_{-1\le l\le \pi}H_{1}\big(SF_{n}(R_{w, h}; i_{0}, \dots, i_{l}, j_{0}, i_{l+1},\dots, i_{\pi})\big).
\]
By assumption, the inclusion of $SF_{n}(R_{w, h}; i_{0}, \dots, i_{l}, j_{0}, i_{l+1},\dots, i_{\pi})$ into $F_{n}(\R^{2}; i_{0}, \dots, i_{l}, j_{0}, i_{l+1},\dots, i_{\pi})$ induces an isomorphism 
\[
H_{1}\big(SF_{n}(R_{w, h}; i_{0}, \dots, i_{l}, j_{0}, i_{l+1},\dots, i_{\pi})\big)\cong H_{1}\big(F_{n}(\R^{2}; i_{0}, \dots, i_{l}, j_{0}, i_{l+1},\dots, i_{\pi})\big).
\]
It follows that 
\begin{align*}
E^{1}_{0,1}\big[SF_{n}(R_{w, h}; i_{0}, \dots, i_{\pi})\big]&=\bigoplus_{j_{0}\neq i_{0}, \dots, i_{\pi}}\bigoplus_{-1\le l\le \pi}H_{1}\big(SF_{n}(R_{w, h}; i_{0}, \dots, i_{l}, j_{0}, i_{l+1},\dots, i_{\pi})\big)\\
&\cong \bigoplus_{j_{0}\neq i_{0}, \dots, i_{\pi}}\bigoplus_{-1\le l\le \pi}H_{1}\big(F_{n}(\R^{2}; i_{0}, \dots, i_{l}, j_{0}, i_{l+1},\dots, i_{\pi})\big).
\end{align*}

By Propositions \ref{basis for homology of FnR2} and \ref{forget points on the right} we can get a basis for $H_{1}\big(SF_{n}(R_{w, h}; i_{0}, \dots, i_{l}, j_{0}, i_{l+1},\dots, i_{\pi})\big)$ consisting of products of wheels.
Consider a distinct pair of basis classes $[\alpha_{l}]\in H_{1}\big(SF_{n}(R_{w, h}; i_{0}, \dots, i_{l}, j_{0}, i_{l+1},\dots, i_{\pi})\big)$ and $[\alpha_{m}]\in H_{1}\big(SF_{n}(R_{w, h}; i_{0}, \dots, i_{m}, j_{0}, i_{m+1},\dots, i_{\pi})\big)$ that only differ by how $j_{0}$ is shuffled into $(i_{0}, \dots, i_{\pi})$.
Treating the wheels comprising $[\alpha_{l}]$ and $[\alpha_{m}]$ as big squares, we can place them in $R_{w,h}$ as in the target configuration of Proposition \ref{free aj in k=1}.
One can check that upon freeing $j_{0}$, i.e., applying $d^{1}$, we get two path-connected big square configurations in $SF_{n}(R_{w, h}; i_{0},\dots, i_{\pi})$, so they are homologous.
Thus,
\begin{align*}
d^{1}\Big(E^{1}_{0,1}\big[SF_{n}(R_{w, h}; i_{0}, \dots, i_{\pi})\big]\Big)&=d^{1}\left(\bigoplus_{j_{0}\neq i_{0}, \dots, i_{\pi}}\bigoplus_{-1\le l\le \pi}H_{1}\big(SF_{n}(R_{w, h}; i_{0}, \dots, i_{l}, j_{0}, i_{l+1},\dots, i_{\pi})\big)\right)\\
&=d^{1}\left(\bigoplus_{j_{0}\neq i_{0}, \dots, i_{\pi}}H_{1}\big(SF_{n}(R_{w, h};j_{0},i_{0},\dots, i_{\pi})\big)\right).
\end{align*}
Proposition \ref{free aj in k=1}, tell us that this is isomorphic to the image of $E^{1}_{0,1}\big[F_{n}(\R^{2}; i_{0}, \dots, i_{\pi})\big]$ under its $d^{1}$-differential.
\end{proof}

Having established the first step of our proofs of Propositions \ref{inductive step on p for k=1 and n small} and \ref{inductive step on p for k=1 and n large}, we now complete their proofs by determining how the images of the $E^{2}_{1,0}$-entries behave.

\begin{proof}[Proof of Proposition \ref{inductive step on p for k=1 and n small}]
Letting $j_{0}$ range over $\{1, \dots, n\}-\{i_{0}, \dots, i_{\pi}\}$, we can cover $SF_{n}(R_{w,h}; i_{0}, \dots, i_{\pi})$ and $F_{n}(\R^{2}; i_{0}, \dots, i_{\pi})$ with sets of the form $SF_{n}(R_{w, h}; j_{0}; i_{0}, \dots, i_{\pi})$ and $F_{n}(\R^{2}; j_{0}; i_{0}, \dots, i_{\pi})$, respectively.
Consider the resulting augmented Mayer--Vietoris spectral sequences; note
$E^{1}_{-1,1}\big[SF_{n}(R_{w, h};i_{0},\dots, i_{\pi})\big]=H_{1}\big(SF_{n}(R_{w, h};i_{0},\dots, i_{\pi})\big)$ and $E^{1}_{-1,1}\big[F_{n}(\R^{2};i_{0},\dots, i_{\pi})\big]=H_{1}\big(F_{n}(\R^{2};i_{0},\dots, i_{\pi})\big)$ by Proposition \ref{E0E1EinftyMV} , and that for $r\ge 1$, the only potentially non-zero terms that map into the $E^{r}_{-1,1}$-entries under the $d^{r}$-differentials are $E^{1}_{0,1}$ and $E^{2}_{1,0}$.
As such, it suffices to calculate these terms and their images under the corresponding differentials. 

By Proposition \ref{k=1 images of d1 differentials are the same n small}, we have that 
\[
d^{1}\Big(E^{1}_{0,1}\big[SF_{n}(R_{w, h}; i_{0}, \dots, i_{\pi})\big]\Big)\cong d^{1}\Big(E^{1}_{0,1}\big[F_{n}(\R^{2}; i_{0}, \dots, i_{\pi})\big]\Big),
\]
so it suffices to prove that the images of the $d^{2}$-differentials are isomorphic.

If $2\ge s\ge 0$ and $s+1+\pi+1\le h$, Proposition \ref{connected double right configuration space} tells us that
\begin{align*}
H_{0}\big(SF_{n}(R_{w, h};j_{0}, \dots, j_{s};i_{0},\dots, i_{\pi})\big)&\cong H_{0}\big(SF_{n}(R_{w, h};j_{0}, \dots, j_{s}, i_{0},\dots, i_{\pi})\big)\\ &\cong H_{0}\big(F_{n}(\R^{2};j_{0}, \dots, j_{s}, i_{0},\dots, i_{\pi})\big)\\
&\cong H_{0}\big(F_{n}(\R^{2};j_{0}, \dots, j_{s}; i_{0},\dots, i_{\pi})\big).
\end{align*}

Otherwise, if $2\ge s\ge 0$ and  $s+1+\pi+1> h$, then our assumptions on $n$ tell us that
\begin{align*}
H_{0}\big(SF_{n}(R_{w, h};j_{0}, \dots, j_{s};i_{0},\dots, i_{\pi})\big)&\cong H_{0}\big(SF_{n-\pi-1}(R_{w, h};j_{0}, \dots, j_{s})\big)\\ &\cong H_{0}\big(F_{n-\pi-1}(\R^{2};j_{0}, \dots, j_{s})\big)\\
&\cong H_{0}\big(F_{n}(\R^{2};j_{0}, \dots, j_{s}; i_{0},\dots, i_{\pi})\big).
\end{align*}
As a result, the $E^{1}_{0,0}$, $E^{1}_{1,0}$, and $E^{1}_{2,0}$-entries of the Mayer--Vietoris spectral sequence for $SF_{n}(R_{w,h}; i_{0}, \dots, i_{\pi})$ are isomorphic to the corresponding entries in the $E\big[F_{n}(\R^{2}; i_{0}, \dots, i_{\pi})\big]$.

By Proposition \ref{free aj in k=1}, the $d^{1}$-differentials between these entries can be identified, so the $E^{2}_{1,0}$-entries of these spectral sequences can be identified. 
By Lemma \ref{MV is Arc} and Propositions \ref{forget points on the right} and \ref{E2pageforFnR2}, it follows that 
\begin{align*}
E^{2}_{1,0}\big[SF_{n}(R_{w, h};i_{0},\dots, i_{\pi})\big]&\cong E^{2}_{1,0}\big[F_{n}(\R^{2};i_{0},\dots, i_{\pi})\big]\\
&\cong \text{Ind}^{S_{n-\pi-3}}_{S_{n-\pi-1}\times S_{2}}H^{\mathcal{A}_{1}}_{0}\Big(H_{0}\big(F(\R^{2})\big)\Big)_{n-\pi-3}\boxtimes \mathcal{T}_{2}.
\end{align*}
By Proposition \ref{basis for homology of FnR2}, this is $0$ unless $n=\pi+3$, in which case it is $\mathcal{T}_{2}=[j_{0}, j_{1}]$.
Moreover, if $n=\pi+3$, 
\[
E^{1}_{0,1}\big[SF_{\pi+3}(R_{w,h}; i_{0}, \dots, i_{\pi})\big]\cong E^{1}_{0,1}\big[F_{\pi+3}(\R^{2}; i_{0}, \dots, i_{\pi})\big]=0,
\]
so it suffices to note that 
\[
d^{2}\Big(E^{2}_{1,0}\big[SF_{n}(R_{w, h};i_{0},\dots, i_{\pi})\big]\Big)\cong d^{2}\Big(E^{2}_{1,0}\big[F_{n}(\R^{2};i_{0},\dots, i_{\pi})\big]\Big),
\]
which follows directly from Proposition \ref{move the big squares around}, completing the proof.
\end{proof}

\begin{proof}[Proof of Proposition \ref{inductive step on p for k=1 and n large}]
Letting $j_{0}$ range over $\{1, \dots, n\}-\{i_{0}, \dots, i_{\pi}\}$, we can cover $SF_{n}(R_{w,h}; i_{0}, \dots, i_{\pi})$ and $F_{n}(\R^{2}; i_{0}, \dots, i_{\pi})$ with sets of the form $SF_{n}(R_{w, h}; j_{0}; i_{0}, \dots, i_{\pi})$ and $F_{n}(\R^{2}; j_{0}; i_{0}, \dots, i_{\pi})$, respectively.
Consider the resulting augmented Mayer--Vietoris spectral sequences; note $E^{1}_{-1,1}\big[SF_{n}(R_{w, h};i_{0},\dots, i_{\pi})\big]=H_{1}\big(SF_{n}(R_{w, h};i_{0},\dots, i_{\pi})\big)$ and $E^{1}_{-1,1}\big[F_{n}(\R^{2};i_{0},\dots, i_{\pi})\big]=H_{1}\big(F_{n}(\R^{2};i_{0},\dots, i_{\pi})\big)$ by Proposition \ref{E0E1EinftyMV}, and that for $r\ge 1$, the only potentially non-zero terms that map into the $E^{r}_{-1,1}$-entries under the $d^{r}$-differentials are $E^{1}_{0,1}$ and $E^{2}_{1,0}$.
As such, it suffices to calculate these terms and their images under the corresponding differentials. 

In Proposition \ref{k=1 images of d1 differentials are the same n large}, we have that 
\[
d^{1}\Big(E^{1}_{0,1}\big[SF_{n}(R_{w, h}; i_{0}, \dots, i_{\pi})\big]\Big)\cong d^{1}\Big(E^{1}_{0,1}\big[F_{n}(\R^{2}; i_{0}, \dots, i_{\pi})\big]\Big),
\]
so it suffices to prove that the images of the $d^{2}$-differentials are isomorphic.

Since $(w-1)h<n-\pi-1$, the configuration spaces $SF_{n}(R_{w, h};j_{0};i_{0},\dots, i_{\pi})$, $SF_{n}(R_{w, h};j_{0}, j_{1};i_{0},\dots, i_{\pi})$, and $SF_{n}(R_{w, h};j_{0}, j_{1}, j_{2};i_{0},\dots, i_{\pi})$ are not connected.
This follows from the pigeonhole principle, as for any $s>0$, the $j_{0}, \dots, j_{s}$ cannot be in the second right-most column of $R_{w, h}$ since they must be the right-most of the $n-\pi-1$ squares not labeled $i_{0}, \dots, i_{\pi}$.
It follows that any configuration in $SF_{n}(R_{w, h}; j_{0}, \dots, j_{s}; i_{0}, \dots, i_{\pi})$ can be viewed as a configuration in $SF_{n}(R_{w, h}; \sigma)$, where $\sigma\in \Sigma\big((j_{0}, \dots, j_{s}), (i_{0},\dots, i_{\pi})\big)$ is a shuffle of $(j_{0}, \dots, j_{s})$ into $(i_{0},\dots, i_{\pi})$.
By Proposition \ref{connected single right configuration space}, the space $SF_{n}(R_{w, h}; \sigma)$ is connected, so for $2\ge s\ge 0$, we have that 
\[
H_{0}\big(SF_{n}(R_{w, h}; j_{0}, \dots, j_{s}; i_{0}, \dots, i_{\pi} )\big)\cong \Z^{\binom{s+\pi+2}{\pi+1}}.
\]
It follows that, for $0\le s\le 2$
\[
E^{1}_{s,0}\big[SF_{n}(R_{w, h}; i_{0}, \dots, i_{\pi})\big]\cong \Z^{(n-\pi-1)\cdots(n-\pi-1-s)\binom{s+\pi+2}{\pi+1}}.
\]
These entries correspond to the $2$-skeleton of the nerve complex $K$ of the cover of $SF_{n}(R_{w,h};i_{0}, \dots, i_{\pi})$ by the $l+2$ sets of the form $SF_{n}(R_{w,h};i_{0}, \dots, i_{l}, j_{0}, i_{l+1},\dots, i_{\pi})$, and that the boundary operator of this complex is the $d^{1}$-differential of our augmented Mayer--Vietoris spectral sequence.
We will show that $K$ has trivial first homology, proving that the $E^{2}_{1,0}\big[SF_{n}(R_{w, h}; i_{0}, \dots, i_{\pi})\big]=0$---Propositions \ref{forget points on the right} and \ref{E2pageforFnR2} prove that the same is true for $E^{2}_{1,0}\big[F_{n}(\R^{2}; i_{0}, \dots, i_{\pi})\big]$.
To see that $H_{1}(K)=0$, we will prove that all cellular loops of length $2$ or $3$ in $K$ are contractible and that all loops of length $4$ or more can be shortened.
Since $K$ has no loops of length $1$, this shows that all loops in $K$ are contractible.

First, note that any loop $\gamma_{2}$ of length $2$ in $K$ connects a pair of vertices that are indexed by a pair of spaces of the form $SF_{n}(R_{w, h}; i_{0}, \dots, i_{a}, j_{0}, i_{a+1}, \dots, i_{\pi})$ and $SF_{n}(R_{w, h}; i_{0}, \dots, i_{a}, j_{1}, i_{a+1}, \dots, i_{\pi}),$ where $j_{0}\neq j_{1}$.
Between such a pair of vertices, there are precisely two edges indexed by $SF_{n}(R_{w, h}; i_{0}, \dots, i_{a}, j_{0}, j_{1}, i_{a+1}, \dots, i_{\pi})$ and $SF_{n}(R_{w, h}; i_{0}, \dots, i_{a}, j_{1}, j_{0}, i_{a+1}, \dots, i_{\pi})$.
It follows that
\[
\gamma_{2}=SF_{n}(R_{w, h}; i_{0}, \dots, i_{a}, j_{0}, j_{1}, i_{a+1}, \dots, i_{\pi})+SF_{n}(R_{w, h}; i_{0}, \dots, i_{a}, j_{1}, j_{0}, i_{a+1}, \dots, i_{\pi}).
\]
To see this, note that there are no edges between vertices of the form $SF_{n}(R_{w, h}; i_{0}, \dots, i_{a}, j_{0}, i_{a+1}, \dots, i_{\pi})$ and $SF_{n}(R_{w, h}; i_{0}, \dots, i_{b}, j_{0}, i_{b+1}, \dots, i_{\pi})$ for $a\neq b$, and between vertices $SF_{n}(R_{w, h}; i_{0}, \dots, i_{a}, j_{0}, i_{a+1}, \dots, i_{\pi})$ and $SF_{n}(R_{w, h}; i_{0}, \dots, i_{b}, j_{1}, i_{b+1}, \dots, i_{\pi})$ where $j_{0}\neq j_{1}$ and $a<b$ there is only one edge, which is indexed by $SF_{n}(R_{w, h}; i_{0}, \dots, i_{a}, j_{0}, i_{a+1}, \dots,i_{b}, j_{1}, i_{b+1}, \dots,  i_{\pi})$.
As such, it cannot be part of a loop of length $2$.

Next, we check that $\gamma_{2}$ is contractible.
Since $\min\{w-1, h\}\ge 3$ and $n-\pi-1\ge (w-1)h$, there are more than 9 squares not labeled $i_{0}, \dots, i_{\pi}$. 
Letting $j_{2}$ be any one of the squares not labeled $j_{0}$ or $j_{1}$, we have that  $\gamma_{2}$ is the boundary of the $2$-chain
\[
SF_{n}(R_{w, h}; i_{0}, \dots, i_{a}, j_{0}, j_{1}, j_{2}, i_{a+1}, \dots, i_{\pi})+SF_{n}(R_{w, h}; i_{0}, \dots, i_{a}, j_{1}, j_{0}, j_{2}, i_{a+1}, \dots, i_{\pi}),
\]
making it contractible.

Next, we show that all loops $\gamma_{3}$ of length $3$ are contractible. 
The above remark proves that such loops connect three vertices of the form $SF_{n}(R_{w, h}; i_{0}, \dots, i_{a}, j_{0}, i_{a+1}, \dots, i_{\pi})$, $SF_{n}(R_{w, h}; i_{0}, \dots, i_{b}, j_{1}, i_{b+1}, \dots, i_{\pi})$, and $SF_{n}(R_{w, h}; i_{0}, \dots, i_{c}, j_{2}, i_{c+1}, \dots, i_{\pi})$, where $j_{0}$, $j_{1}$, and $j_{2}$ are all distinct.
Note, there need not be any relation between $a$, $b$, and $c$, so without loss of generality we assume that $a\le b\le c$, and if any of $a$, $b$, and $c$ are equal, the length $2$ case shows that we can replace an edge with another.
It follows that we can write $\gamma_{3}$ in the form
\begin{multline*}
\gamma_{3}=SF_{n}(R_{w, h}; i_{0}, \dots, i_{a}, j_{0}, i_{a+1}, \dots, i_{b}, j_{1}, i_{b+1},\dots, i_{\pi})\\
+SF_{n}(R_{w, h}; i_{0}, \dots, i_{b}, j_{1}, i_{b+1}, \dots, i_{c}, j_{2}, i_{c+1},\dots, i_{\pi})\\
-SF_{n}(R_{w, h}; i_{0}, \dots, i_{a}, j_{0}, i_{a+1}, \dots, i_{c}, j_{2}, i_{c+1},\dots, i_{\pi}).
\end{multline*}
This loop is the boundary of the $2$-cell $SF_{n}(R_{w, h}; i_{0},\dots, i_{a}, j_{0}, i_{a+1},\dots, i_{b}, j_{1}, i_{b+1}, \dots, i_{c}, j_{2}, i_{c+1}, \dots, i_{\pi}),$ so it is contractible.

Finally, we wish to show that any loop $\gamma_{l}$ of length $l\ge 4$ can be shortened to a loop of length at most $l-1$, and is therefore, contractible. 
Let $\delta_{3}\subset \gamma_{l}$ be a path of length $3$ in $\gamma_{l}$.
We will show that there is a path $\delta_{2}$ of length $2$, such that $\delta_{3}\delta_{2}^{-1}$ is contractible in $K$.
The $\gamma_{2}$ and $\gamma_{3}$ cases allow us to assume that $\delta_{3}$ is of the form
\begin{multline*}
 \delta_{3}=SF_{n}(R_{w, h}; i_{0}, \dots, i_{a}, j_{0}, i_{a+1}, \dots, i_{b}, j_{1}, i_{b+1},\dots, i_{\pi})\\
 +SF_{n}(R_{w, h}; i_{0}, \dots, i_{b}, j_{1}, i_{b+1}, \dots, i_{c}, j_{0}, i_{c+1},\dots, i_{\pi})\\
 +SF_{n}(R_{w, h}; i_{0}, \dots, i_{c}, j_{0}, i_{c+1}, \dots, i_{d}, j_{1}, i_{d+1},\dots, i_{\pi}).
\end{multline*}
Note, we might have that $c<b$ or $d<c$, but $d\neq a$.
Since $n-\pi-1\ge (w-1)h\ge 9$, there are more than 9 squares not labeled $i_{0}, \dots, i_{\pi}$, so there is some square $j_{2}$ not labeled $j_{0}, j_{1}$.
Setting 
\[
\delta_{2}=SF_{n}(R_{4, h}; j_{2}, i_{0}, \dots, i_{a}, j_{0}, i_{a+1}, \dots, i_{\pi})+SF_{n}(R_{4, h}; j_{2}, i_{0}, \dots, i_{d}, j_{1}, i_{d+1}, \dots, i_{\pi}),
\]
we see that $\delta_{3}\delta_{2}^{-1}$ is the boundary of the $2$-chain
\begin{multline*}
 SF_{n}(R_{w, h}; j_{2}, i_{0}, \dots, i_{a}, j_{0}, i_{a+1}, \dots, i_{b}, j_{1}, i_{b+1},\dots, i_{\pi})\\
 +SF_{n}(R_{w, h}; j_{2}, i_{0}, \dots, i_{b}, j_{1}, i_{b+1}, \dots, i_{c}, j_{0}, i_{c+1},\dots, i_{\pi})\\
 +SF_{n}(R_{w, h}; j_{2}, i_{0}, \dots, i_{c}, j_{0}, i_{c+1}, \dots, i_{d}, j_{1}, i_{d+1},\dots, i_{\pi}).
\end{multline*}
It follows that $\gamma_{l}$ can be shortened, hence contracted.
Thus $H_{1}(K)\cong E^{2}_{1,0}\big[SF_{n}(R_{w, h}; i_{0}, \dots, i_{\pi})\big]=0$, so the image of this entry in $E^{2}_{1,0}\big[SF_{n}(R_{w, h}; i_{0}, \dots, i_{\pi})\big]$ must be trivial, just like the image of $E^{2}_{1,0}\big[F_{n}(\R^{2}; i_{0}, \dots, i_{\pi})\big]$ in $E^{2}_{1,0}\big[F_{n}(\R^{2}; i_{0}, \dots, i_{\pi})\big]$, completing the proof.
\end{proof}

\bibliographystyle{amsalpha}
\bibliography{StabilityinSpaceDirection}

\bigskip

Department of Mathematics, Cinvestav, Mexico City, Mexico.

\textit{Email address:} jesus.glz-espino@cinvestav.mx and jesus@math.cinvestav.mx

\medskip

Department of Mathematics, The Ohio State University, Columbus, OH 43210, USA.

\textit{Email address:} mkahle@math.osu.edu 

\medskip
Department of Mathematics, The University of Chicago, Chicago, IL 60605, USA.

\textit{Email address:} wawrykow@uchicago.edu

\end{document}